\documentclass[a4paper,leqno,11pt]{amsart}

\usepackage{amsfonts,amssymb,verbatim,amsmath,amsthm,latexsym,textcomp,amscd}
\usepackage{latexsym,amsfonts,amssymb,epsfig,verbatim}
\usepackage{amsmath,amsthm,amssymb,latexsym,graphics,textcomp}
\usepackage{graphicx}
\usepackage{color}
\usepackage{url}
\usepackage{enumerate}
\usepackage[mathscr]{euscript}

\input xy
\xyoption{all}

\usepackage{hyperref}

\newcommand{\uM}{\underline{M}}
\newcommand{\uN}{\underline{N}}

\newcommand{\Z}{\mathbb Z}
\newcommand{\bZ}{\langle \Z \rangle}
\newcommand{\bbZ}{\langle \langle \Z \rangle \rangle}

\newcommand{\bbC}{\langle \langle C \rangle \rangle}
\newcommand{\bbZp}{\langle \langle \Z/p \rangle \rangle}
\newcommand{\bbZq}{\langle \langle \Z/q \rangle \rangle}
\newcommand{\bbZpq}{\langle \langle \Z/pq \rangle \rangle}
\newcommand{\bZp}{\langle \Z/p \rangle}
\newcommand{\bZq}{\langle \Z/q \rangle}

\newcommand{\bC}{\langle  C \rangle }

\newcommand{\F}{\mathbb F}
\newcommand{\cA}{\mathcal A}
\newcommand{\BB}{\mathcal B}
\newcommand{\GG}{\mathcal G}

\newcommand{\MM}{\mathcal{M}}
\newcommand{\CC}{\mathcal{C}}
\newcommand{\DD}{\mathcal{D}}
\newcommand{\KK}{\mathcal{K}}
\newcommand{\QQ}{\mathcal{Q}}
\newcommand{\OO}{\mathcal{O}}
\newcommand{\PP}{\mathcal{P}}
\newcommand{\UU}{\mathcal{U}}

\newcommand{\Ep}{\mathcal{E}_p}

\newcommand{\C}{\mathbb C}

\newcommand{\res}{\mathit{res}}
\newcommand{\tr}{\mathit{tr}}

\newcommand{\Hom}{\mathit{Hom}}
\newcommand{\Ext}{\mathit{Ext}}

\newcommand{\coker}{\mbox{coker}}
\newcommand{\uH}{\underline{H}}
\newcommand{\uA}{\underline{A}}
\newcommand{\ua}{\underline{a}}
\newcommand{\ub}{\underline{b}}

\theoremstyle{plain}
\newtheorem{theorem}{Theorem}[section]
\newtheorem{thm}[theorem]{Theorem}
\newtheorem{lemma}[theorem]{Lemma}
\newtheorem{cor}[theorem]{Corollary}
\newtheorem{prop}[theorem]{Proposition}

\newtheorem{remark}[theorem]{Remark}

\theoremstyle{definition}
\newtheorem{defn}[theorem]{Definition}

\newtheorem*{thma}{Theorem A}
\newtheorem*{thmb}{Theorem B}
\newtheorem*{thmc}{Theorem C}

\author{Samik Basu, Surojit Ghosh}

\address{Stat-Math unit,
Indian Statistical Institute,
Kolkata - 700108,  India.}
\email{samik.basu2@gmail.com; samikbasu@isical.ac.in  }

\address{Department of Mathematics,
University of Haifa,
3498838 Haifa, Israel.} 
\email{surojitghosh89@gmail.com}


\begin{document}

\title{Computations in $C_{pq}$-Bredon cohomology}

\begin{abstract}
In this paper, we compute the $RO(C_{pq})$-graded cohomology of $C_{pq}$-orbits. We deduce that in all the cases the Bredon cohomology groups are a function of the fixed point dimensions of the underlying virtual representations. Further, when thought of as a Mackey functor, the same independence result holds in almost all cases. This generalizes earlier computations of Stong and Lewis for the group $C_p$. 

The computations of cohomology of orbits are used to prove a freeness theorem. The analogous result for the group $C_p$ was proved by Lewis. We demonstrate that certain complex projective spaces and complex Grassmannians satisfy the freeness theorem. 
\end{abstract}

\maketitle


\section{Introduction}
Bredon cohomology defined in \cite{Bre67} arises as the natural generalization of ordinary cohomology to the category of $G$-spaces, especially in the context of $G$-equivariant obstruction theory. In the literature, a more favoured cohomology theory for $G$-spaces is the cohomology of the Borel construction, for which non-equivariant techniques may be used to make many computations. On the other hand, Bredon cohomology has turned out to be difficult to compute even for simple spaces like spheres with a linear $G$-action. Their application  in various important examples (for example, \cite{HHR09} uses certain computations for the group $C_{2^n}$) demonstrate that the computations are not without interest. 

The Bredon cohomology groups $H^n_G( - ;\uM)$ are defined for ``coefficient systems" $\uM$ which are functors from the orbit category $\OO_G$ to Abelian groups (see \cite[Section I.4]{May96}). The interesting case is when $\uM$ is actually the restriction of a Mackey functor, in which case the cohomology theory extends to one indexed by virtual representations. For $\alpha \in RO(G)$ the real representation ring of $G$, we write $\uH^\alpha_G(X;\uM)$ for the Mackey functor valued cohomology at the index $\alpha$. In this case if we put $X=S^0$, these groups encode the various cohomology and homology groups of the $G$-representation spheres.      

Perhaps the most important Mackey functor is the Burnside ring Mackey functor $\uA$. This plays a similar role in the category of Mackey functors as $\Z$ plays in the category of Abelian groups which we may think of as coefficients for ordinary cohomology. A precise version of this statement is that the Eilenberg MacLane spectrum $H\uA$ is a commutative ring spectrum and for any Mackey functor $\uM$, $H\uM$ is a module over $H\uA$ (Proposition 5.4 of \cite{Lew96}). The starting point of computations for $RO(G)$-graded cohomology must therefore be $\uH^\ast_G(X; \uA)$, where $X$ varies over the set of $G$-points, that is, the set of all orbits $G/H$. 

The Mackey functors $\uH^\ast_G(S^0; \uA)$ were computed by Stong for the group $C_2$ and by Lewis and Stong for the groups $C_p$ (\cite{Lew88}). The multiplicative structure was also computed by Lewis in the same paper. With constant coefficients $\underline{\Z/p}$, the computations of the cohomology of $C_p$-orbits were performed by Stong (Appendix of \cite{Car00}). In this paper, we carry the computations of $\uH^\ast_G(S^0; \uA)$ further to $G=C_{pq}$ where $p$ and $q$ are distinct odd primes. We are more interested in the additive structure of the cohomology Mackey functors and avoid computations of the ring structure. In the following, we describe our results in more detail.  Throughout we use the notation $\uH_G^\alpha(X;\uM)$ to denote the Mackey functor valued cohomology of $X$, and the notation $H_G^\alpha(X;\uM)$ to denote the associated cohomology group in the grading $\alpha \in RO(G)$. This means that $\alpha$ is a virtual representation so that we can make sense of $|\alpha^H|$, the dimension of the $H$-fixed points of $\alpha$. In the case $\uM=\uA$, we often drop the coefficient from the notation and simply write  $\uH^\alpha_{C_{pq}}(S^0)$.    
 
\subsection{Computations of $\uH^\ast_{C_{pq}}(S^0;\uA)$} 
The computations of Stong and Lewis for the group $C_p$ (Proposition \ref{Lewis orbits}) write $\uH^\alpha_{C_p}(S^0;\uA)$ as a function of $|\alpha|$ and $|\alpha^{C_p}|$. When either of these integers are non-zero, the Mackey functor $\uH^\alpha_{C_p}(S^0;\uA)$  is determined up to isomorphism by the two integers. When both $|\alpha|$ and $|\alpha^{C_p}|$ are $0$, the Mackey functor is of the form $A[d_\alpha]$ and so it depends on an additional integer $d_\alpha$. We investigate the Mackey functors $\uH^\ast_{C_{pq}}(S^0;\uA)$ along similar lines. We prove (Theorem \ref{ind})
\begin{thma}
{\it For every subgroup $H$ of $C_{pq}$, the groups $H^\alpha_{C_{pq}}(C_{pq}/H;\uA)$ depend only on the dimensions of the fixed points of $\alpha$.}   
\end{thma}
The independence theorem above is extended to Mackey functor valued cohomology in many cases. That is, given the four integers $|\alpha|$, $|\alpha^{C_p}|$, $|\alpha^{C_q}|$ and $|\alpha^{C_{pq}}|$, in many cases we write down $\uH^\alpha_{C_{pq}}(S^0)$ as a function of these four integers. Similar to the $C_p$ case, we observe that the Mackey functors $\uH^\alpha_{C_{pq}}(S^0)$ depend only on the fixed points of $\alpha$ unless many of these fixed points are of dimension $0$. More precisely (Theorem \ref{indmack}) 
\begin{thmb}
{\it Suppose $\alpha \in RO(C_{pq})$ is such that at least one of $|\alpha^H|$ or $|\alpha^K|$ is non-zero whenever $K$ and $H$ are distinct subgroups of $C_{pq}$ with exactly one of them  $\in \{ e, C_{pq}\}$. Then, up to isomorphism, the Mackey functor $\uH^\alpha_{C_{pq}}(S^0;\uA)$ depends only on the dimension of the fixed points of $\alpha$.}
\end{thmb}

The proof of both the Theorems are done through explicit computations of the Mackey functors $\uH^\alpha_{C_{pq}}(S^0)$. We provide explicit formulas for every $\alpha$ as in Theorem B, listed in the tables \ref{odd-tab}, \ref{evpos-tab}, \ref{evneg-tab}, \ref{evzer-tab} below, keeping in mind the obvious symmetry between $p$ and $q$. The notations for the Mackey functors used are defined in Section \ref{mack}.

\begin{table}[ht]

\begin{tabular}{ |p{5.1cm}|p{2.5cm}||p{5.1cm}| p{1.3cm}|  }
 \hline
{\tiny $\alpha$ } & { \tiny $\uH_{C_{pq}}^\alpha(S^0)$} & {\tiny $\alpha$ } & {\tiny $\uH_{C_{pq}}^\alpha(S^0)$} \\
 \hline
{\tiny $|\alpha|>0,|\alpha^{C_{pq}}|\leq 1$}  & {\tiny $0$ } 
& {\tiny $|\alpha|>0,|\alpha^{C_p}|>0,|\alpha^{C_q}|>0,|\alpha^{C_{pq}}|> 1$}  &  {\tiny $0$ }  \\
\hline
{\tiny $|\alpha|>0,|\alpha^{C_p}|<0,|\alpha^{C_q}|<0,|\alpha^{C_{pq}}|> 1$}  &  {\tiny $\bbZpq$ }
& {\tiny $|\alpha|>0,|\alpha^{C_p}|<0,|\alpha^{C_q}|>0,|\alpha^{C_{pq}}|> 1$}  &  {\tiny $\bbZq$ }\\
\hline
{\tiny $|\alpha|<0,|\alpha^{C_p}|\leq 1,|\alpha^{C_q}|\leq 1,|\alpha^{C_{pq}}| \leq 1$}  &  {\tiny $0$ }
& {\tiny $|\alpha|<0,|\alpha^{C_p}|=1,|\alpha^{C_q}|=1,|\alpha^{C_{pq}}|> 1$}  &  {\tiny $0$ }\\
\hline
{\tiny $|\alpha|<0,|\alpha^{C_p}|<0,|\alpha^{C_q}|<0,|\alpha^{C_{pq}}| > 1$}  &  {\tiny $\bbZpq$ }
& {\tiny $|\alpha|<0,|\alpha^{C_p}|=1,|\alpha^{C_q}|<0,|\alpha^{C_{pq}}|> 1$}  &  {\tiny $\bbZp$ }\\
\hline
{\tiny $|\alpha|<0,|\alpha^{C_p}|>1,|\alpha^{C_q}|<0,|\alpha^{C_{pq}}| > 1$}  &  {\tiny $\KK_p \bZp\oplus \bbZp$ }
& {\tiny $|\alpha|<0,|\alpha^{C_p}|>1,|\alpha^{C_q}|=1$}  &  {\tiny $\KK_p \bZp$ }\\
 \hline
{\tiny $|\alpha|<0,|\alpha^{C_p}|>1,|\alpha^{C_q}|>1$}  &  {\tiny $\KK_p \bZp\oplus \KK_q \bZq$ }
& {\tiny $|\alpha|<0,|\alpha^{C_p}|>1,|\alpha^{C_q}|\leq 1, |\alpha^{C_{pq}}|\leq 1$}  &  {\tiny $\KK_p \bZp$ }\\
 \hline
\end{tabular}
\vspace{.2cm}
\caption{Formula for $\uH_{C_{pq}}^\alpha(S^0;\uA)$ for $|\alpha|$ odd.}
\label{odd-tab}
\end{table}

\begin{table}[ht]

\begin{tabular}{ |p{4.2cm}|p{2.35cm}||p{4.2cm}| p{3.2cm}|  }
 \hline
{\tiny $\alpha$ } & { \tiny $\uH_{C_{pq}}^\alpha(S^0)$} & {\tiny $\alpha$ } & {\tiny $\uH_{C_{pq}}^\alpha(S^0)$} \\
 \hline
{\tiny $|\alpha^{C_p}|>0,|\alpha^{C_q}|>0,|\alpha^{C_{pq}}|>0$}  & {\tiny $0$ } 
& {\tiny $|\alpha^{C_p}|>0,|\alpha^{C_q}|>0,|\alpha^{C_{pq}}|=0$}  &  {\tiny $\bbZ$ }  \\
\hline
{\tiny $|\alpha^{C_p}|>0,|\alpha^{C_q}|>0,|\alpha^{C_{pq}}|<0$}  &  {\tiny $\bbZpq$ }
& {\tiny $|\alpha^{C_p}|>0,|\alpha^{C_q}|<0,|\alpha^{C_{pq}}|> 0$}  &  {\tiny $\CC_q\bZq$ }\\
\hline
{\tiny $|\alpha^{C_p}|>0,|\alpha^{C_q}|<0,|\alpha^{C_{pq}}| <0$}  &  {\tiny $\CC_q \bZq \oplus \bbZp$ }
& {\tiny $|\alpha^{C_p}|>0, |\alpha^{C_q}|<0,|\alpha^{C_{pq}}|=0$}  &  {\tiny $\CC_q \bZq \oplus \bbZ$ }\\
\hline
{\tiny $|\alpha^{C_p}|>0,|\alpha^{C_q}|=0,|\alpha^{C_{pq}}| > 0$}  &  {\tiny $\CC_q\bZ$ }
& {\tiny $|\alpha^{C_p}|>0,|\alpha^{C_q}|=0,|\alpha^{C_{pq}}|<0$}  &  {\tiny $\bbZq \oplus \KK_q\bZ$ }\\
\hline
{\tiny $|\alpha^{C_p}|<0,|\alpha^{C_q}|<0,|\alpha^{C_{pq}}|\neq 0$}  &  {\tiny $\CC_p \bZp\oplus \CC_q\bZq$ }
& {\tiny $|\alpha^{C_p}|<0,|\alpha^{C_q}|<0,|\alpha^{C_{pq}}|=0$}  &  {\tiny $\bbZ\oplus \CC_p \bZp\oplus \CC_q\bZq$ }\\
 \hline
{\tiny $|\alpha^{C_p}|<0,|\alpha^{C_q}|=0, |\alpha^{C_{pq}}|>0$}  &  {\tiny $\CC_p \bZp\oplus \CC_q\bZ$ }
& {\tiny $|\alpha^{C_p}|<0,|\alpha^{C_q}|=0, |\alpha^{C_{pq}}|<0$}  &  {\tiny $\KK_q \bZ \oplus \CC_p\bZp$ }\\
\hline
{\tiny $|\alpha^{C_p}|=0,|\alpha^{C_q}|=0, |\alpha^{C_{pq}}|>0$}  &  {\tiny $\CC_p \bZ\oplus \CC_q\bZ$ }
& {\tiny $|\alpha^{C_p}|=0,|\alpha^{C_q}|=0, |\alpha^{C_{pq}}|<0$}  &  {\tiny $\KK_p \bZ \oplus \KK_q\bZ$ }\\
 \hline
\end{tabular}
\vspace{.2cm}
\caption{Formula for $\uH_{C_{pq}}^\alpha(S^0;\uA)$ for $|\alpha|>0$ even.}
\label{evpos-tab}
\end{table}

\begin{table}[ht]

\begin{tabular}{ |p{4.2cm}|p{2.35cm}||p{4.2cm}| p{2.35cm}|  }
 \hline
{\tiny $\alpha$ } & { \tiny $\uH_{C_{pq}}^\alpha(S^0)$} & {\tiny $\alpha$ } & {\tiny $\uH_{C_{pq}}^\alpha(S^0)$} \\
 \hline
{\tiny $|\alpha^{C_p}|< 0,|\alpha^{C_q}|<0,|\alpha^{C_{pq}}|\neq 0$}  & {\tiny $0$ } 
& {\tiny $|\alpha^{C_p}|\neq 0,|\alpha^{C_q}|\neq 0,|\alpha^{C_{pq}}|=0$}  &  {\tiny $\bbZ$ }  \\
\hline
{\tiny $|\alpha^{C_p}|\neq 0,|\alpha^{C_q}|>0,|\alpha^{C_{pq}}|>0$}  &  {\tiny $0$ }
& {\tiny $|\alpha^{C_p}|<0,|\alpha^{C_q}|>0,|\alpha^{C_{pq}}|< 0$}  &  {\tiny $\bbZp$ }\\
\hline
 {\tiny $|\alpha^{C_p}|\neq 0, |\alpha^{C_q}|=0,|\alpha^{C_{pq}}|>0$}  &  {\tiny $\CC_q \bZ$ }
&{\tiny $|\alpha^{C_p}|<0,|\alpha^{C_q}|=0,|\alpha^{C_{pq}}| < 0$}  &  {\tiny $\KK_q\bZ$ }\\
\hline
 {\tiny $|\alpha^{C_p}|>0,|\alpha^{C_q}|=0,|\alpha^{C_{pq}}|<0$}  &  {\tiny $\KK_q\bZ\oplus \bbZq$ }
&{\tiny $|\alpha^{C_p}|>0,|\alpha^{C_q}|>0,|\alpha^{C_{pq}}|< 0$}  &  {\tiny $\bbZpq$ }
\\
 \hline
{\tiny $|\alpha^{C_p}|=0,|\alpha^{C_q}|=0, |\alpha^{C_{pq}}|>0$}  &  {\tiny $\CC_p \bZ\oplus \CC_q\bZ$ }
& {\tiny $|\alpha^{C_p}|=0,|\alpha^{C_q}|=0, |\alpha^{C_{pq}}|<0$}  &  {\tiny $\KK_p \bZ \oplus \KK_q\bZ$ }\\
\hline
\end{tabular}
\vspace{.2cm}
\caption{Formula for $\uH_{C_{pq}}^\alpha(S^0;\uA)$ for $|\alpha|<0$ even.}
\label{evneg-tab}
\end{table}

\begin{table}[ht]

\begin{tabular}{ |p{4.2cm}|p{2.35cm}||p{4.2cm}| p{2.35cm}|  }
 \hline
{\tiny $\alpha$ } & { \tiny $\uH_{C_{pq}}^\alpha(S^0)$} & {\tiny $\alpha$ } & {\tiny $\uH_{C_{pq}}^\alpha(S^0)$} \\
 \hline
{\tiny $|\alpha^{C_p}|> 0,|\alpha^{C_q}|>0,|\alpha^{C_{pq}}|> 0$}  & {\tiny $L_{pq}$ } 
& {\tiny $|\alpha^{C_p}|> 0,|\alpha^{C_q}|> 0,|\alpha^{C_{pq}}|=0$}  &  {\tiny $L_{pq}\oplus \bbZ$ }  \\
\hline
{\tiny $|\alpha^{C_p}|> 0,|\alpha^{C_q}|>0,|\alpha^{C_{pq}}|<0$}  &  {\tiny $L_{pq} \oplus \bbZpq$ }
& {\tiny $|\alpha^{C_p}|>0,|\alpha^{C_q}|<0,|\alpha^{C_{pq}}|> 0$}  &  {\tiny $\KK_pL_p$ }\\
\hline
 {\tiny $|\alpha^{C_p}|> 0, |\alpha^{C_q}|<0,|\alpha^{C_{pq}}|=0$}  &  {\tiny $\KK_pL_p\oplus \bbZ$ }
&{\tiny $|\alpha^{C_p}|>0,|\alpha^{C_q}|<0,|\alpha^{C_{pq}}| < 0$}  &  {\tiny $\KK_pL_p\oplus \bbZq$ }\\
\hline
 {\tiny $|\alpha^{C_p}|<0,|\alpha^{C_q}|<0,|\alpha^{C_{pq}}|\neq 0$}  &  {\tiny $R_{pq}$ }
&{\tiny $|\alpha^{C_p}|<0,|\alpha^{C_q}|<0,|\alpha^{C_{pq}}|= 0$}  &  {\tiny $R_{pq}\oplus \bbZ$ }
\\
\hline
\end{tabular}
\vspace{.2cm}
\caption{Formula for $\uH_{C_{pq}}^\alpha(S^0;\uA)$ for $|\alpha|=0$.}
\label{evzer-tab}
\end{table}

Once the values for $\uH^\alpha_{C_{pq}}(S^0;\uA)$ are known for $\alpha$ as in Theorem B, one may compute the values for the cohomology groups $H^\alpha_{C_{pq}}(C_{pq}/H)$ for the other $\alpha$ using Theorem A, and Propositions \ref{shex-e}, \ref{shex-p}. We list the values of these groups in Table \ref{indval-tab}. It is important to note that for the values of $\alpha$ which are not as in Theorem B, the Mackey functor structure of $\uH^\alpha_{C_{pq}}(S^0;\uA)$ is not determined up to isomorphism by the values of the fixed points of $\alpha$. 

\begin{table}[ht]

\begin{tabular}{ |p{5.1cm}|p{2cm}|p{2cm}| p{2cm}|p{2cm}|  }
 \hline
{\tiny $\alpha$ } & { \tiny $H_{C_{pq}}^\alpha(C_{pq}/C_{pq})$} & {\tiny $H_{C_{pq}}^\alpha(C_{pq}/C_p)$ } & {\tiny $H_{C_{pq}}^\alpha(C_{pq}/C_q)$} & {\tiny $H_{C_{pq}}^\alpha(C_{pq}/e)$} \\
 \hline
{\tiny $|\alpha|=0,|\alpha^{C_p}|= 0,|\alpha^{C_q}|\neq 0,|\alpha^{C_{pq}}|> 0$}  & {\tiny $\Z^2$ } &{\tiny $\Z^2$} & {\tiny $\Z$} &{\tiny $\Z$} \\
 \hline
{\tiny $|\alpha|=0,|\alpha^{C_p}|= 0,|\alpha^{C_q}|>0,|\alpha^{C_{pq}}|< 0$}  & {\tiny $\Z^2\oplus \Z/p$ } &{\tiny $\Z^2$} & {\tiny $\Z$} &{\tiny $\Z$} \\
 \hline
{\tiny $|\alpha|=0,|\alpha^{C_p}|= 0,|\alpha^{C_q}|\neq 0,|\alpha^{C_{pq}}|= 0$}  & {\tiny $\Z^3$ } &{\tiny $\Z^2$} & {\tiny $\Z$} &{\tiny $\Z$} \\
 \hline
{\tiny $|\alpha|=0,|\alpha^{C_p}|= 0,|\alpha^{C_q}|=0,|\alpha^{C_{pq}}|= 0$}  & {\tiny $\Z^4$ } &{\tiny $\Z^2$} & {\tiny $\Z^2$} &{\tiny $\Z$} \\
 \hline
{\tiny $|\alpha|=0,|\alpha^{C_p}|= 0,|\alpha^{C_q}|=0,|\alpha^{C_{pq}}|\neq 0$}  & {\tiny $\Z^3$ } &{\tiny $\Z^2$} & {\tiny $\Z^2$} &{\tiny $\Z$} \\
 \hline
{\tiny $|\alpha|=0,|\alpha^{C_p}|= 0,|\alpha^{C_q}|<0,|\alpha^{C_{pq}}|< 0$}  & {\tiny $\Z^2$ } &{\tiny $\Z^2$} & {\tiny $\Z$} &{\tiny $\Z$} \\
 \hline
{\tiny $|\alpha|<0,|\alpha^{C_p}|= 0,|\alpha^{C_q}|\neq 0,|\alpha^{C_{pq}}|= 0$}  & {\tiny $\Z^2$ } &{\tiny $\Z$} & {\tiny $0$} &{\tiny $0$} \\
 \hline
{\tiny $|\alpha|>0,|\alpha^{C_p}|= 0,|\alpha^{C_q}|>0,|\alpha^{C_{pq}}|= 0$}  & {\tiny $\Z^2$ } &{\tiny $\Z$} & {\tiny $0$} &{\tiny $0$} \\
 \hline
{\tiny $|\alpha|>0,|\alpha^{C_p}|= 0,|\alpha^{C_q}|<0,|\alpha^{C_{pq}}|= 0$}  & {\tiny $\Z^2\oplus \Z/q$ } &{\tiny $\Z$} & {\tiny $\Z/q$} &{\tiny $0$} \\
\hline
\end{tabular}
\vspace{.2cm}
\caption{Values for $H_{C_{pq}}^\alpha(C_{pq}/H;\uA)$ which are not covered in Theorem B.}
\label{indval-tab}
\end{table}

The main techniques used in the computation are certain cofibre sequences associated to the complex 1-dimensional representations of $C_{pq}$. Given such a representation $V$, we have a cofibre sequence $S(V)_+ \to S^0 \to S^V$. Further, if $J$ is the kernel of the representation we may form a cofibre sequence $C_{pq}/J \to C_{pq}/J \to S(V)$. These induce long exact sequences on the $RO(C_{pq})$-graded cohomology which we compute in a similar way to the calculations in the Appendix of \cite{Car00}. We know the values of $H^\ast_{C_{pq}}(C_{pq}/J)$ from earlier computations of Lewis and Stong in \cite{Lew88}, and apply these to the latter cofibre sequence to compute $\uH^\ast_{C_{pq}}(S(V)_+)$. Then we apply the former cofibre sequence to deduce certain relations in $\uH^\ast_{C_{pq}}(S^0)$. Although these computations are written when both $p$ and $q$ are odd, the same computations hold if we allow $p=2$  where $\alpha$ satisfies the additional condition that it is the restriction of a complex representation. 

Using the computation of $\uH^\alpha_{C_{pq}}(S^0;\uA)$, we compute the cohomology with coefficients in the constant Mackey functor $R_{pq}$. This Mackey functor is also denoted by $\underline{\Z}$. We list their values in  Table \ref{odd-tabr} and \ref{ev-tabr} up to symmetry of $p$ and $q$ (see Theorem \ref{compr}). Curiously the values for $\uH_{C_{pq}}^\alpha(S^0;R_{pq})$ depend only on $|\alpha|, |\alpha^{C_p}|,$ and $|\alpha^{C_q}|$.  
\begin{table}[ht]

\begin{tabular}{ |p{4cm}|p{2.5cm}||p{4cm}| p{2.5cm}|  }
 \hline
{\tiny $\alpha$ } & { \tiny $\uH_{C_{pq}}^\alpha(S^0)$} & {\tiny $\alpha$ } & {\tiny $\uH_{C_{pq}}^\alpha(S^0)$} \\
 \hline
{\tiny $|\alpha|>0$}  & {\tiny $0$ } 
& {\tiny $|\alpha|<0,|\alpha^{C_p}|\leq 1,|\alpha^{C_q}|\leq 1$}  &  {\tiny $0$ }  \\
\hline
{\tiny $|\alpha|<0,|\alpha^{C_p}|\leq 1,|\alpha^{C_q}|>1$}  &  {\tiny $\KK_q\bZq$ }
& {\tiny $|\alpha|<0,|\alpha^{C_p}|>1,|\alpha^{C_q}|>1$}  &  {\tiny $\KK_p\bZp\oplus \KK_q\bZq$ }\\
\hline
\end{tabular}
\vspace{.2cm}
\caption{Formula for $\uH_{C_{pq}}^\alpha(S^0;R_{pq})$ for $|\alpha|$ odd.}
\label{odd-tabr}
\end{table}

\begin{table}[ht]

\begin{tabular}{ |p{4.2cm}|p{2.35cm}||p{4.2cm}| p{2.35cm}|  }
 \hline
{\tiny $\alpha$ } & { \tiny $\uH_{C_{pq}}^\alpha(S^0)$} & {\tiny $\alpha$ } & {\tiny $\uH_{C_{pq}}^\alpha(S^0)$} \\
 \hline
{\tiny $|\alpha|< 0$}  & {\tiny $0$ } 
& {\tiny $|\alpha|>0, |\alpha^{C_p}|> 0,|\alpha^{C_q}|> 0$}  &  {\tiny $0$ }  \\
\hline
{\tiny $|\alpha|>0, |\alpha^{C_p}|\leq 0,|\alpha^{C_q}|>0$}  &  {\tiny $\KK_p\bZp$ }
& {\tiny $|\alpha|>0, |\alpha^{C_p}|\leq 0,|\alpha^{C_q}|\leq 0$}  &  {\tiny $\KK_p\bZp\oplus \KK_q\bZq$ }\\
\hline
 {\tiny $|\alpha|=0, |\alpha^{C_p}|> 0, |\alpha^{C_q}|> 0$}  &  {\tiny $L_{pq}$ }
&{\tiny $|\alpha|=0, |\alpha^{C_p}|\leq 0,|\alpha^{C_q}|\leq 0$}  &  {\tiny $R_{pq}$ }\\
\hline
 {\tiny $|\alpha|=0, |\alpha^{C_p}|>0,|\alpha^{C_q}|\leq 0$}  &  {\tiny $\KK_pL_p$ }
&  & 
\\
 \hline
\end{tabular}
\vspace{.2cm}
\caption{Formula for $\uH_{C_{pq}}^\alpha(S^0;R_{pq})$ for $|\alpha|$ even.}
\label{ev-tabr}
\end{table}

\subsection{Freeness Theorems} 
As an application of computations of $\uH^\ast_{C_{pq}}(S^0)$, we deduce ``Freeness Theorems" for $C_{pq}$ spaces with even dimensional cells. What we mean by ``Freeness Theorem" is the equivariant generalization of the Theorem : {\it If $X$ is a CW complex with cells only in even dimensions, then the cohomology of $X$ is free module over $\tilde{H}^\ast(S^0;\Z)$.} 

In the equivariant case, the coefficients $\Z$ is replaced by the Burnside ring Mackey functor $\uA$. For the group $C_p$, a freeness theorem was proved in \cite{Lew88}, which carried the hypothesis that the cells are even and must be attached under some dimension restrictions. For finite CW complexes, the hypothesis of dimension restrictions  was removed in \cite{Fer99} and it was proved that any $C_p$-CW-complex with even cells have free cohomology (see also \cite{FL04}). In the former case, it was shown that under the hypothesis, the attaching map of a new even cell must induce the zero-map on cohomology, thus deducing that the cohomology is a direct sum with the summands being in one-to-one correspondence with the cells. This also happens in the case of the non-equivariant freeness theorem. However, in the latter case, the attaching maps may not induce the zero-map and so the freeness theorem works out with certain dimension shifts.     

In this paper, we prove freeness theorems for $C_{pq}$ along the lines of \cite{Lew88}, that is, for those cases, in which the attaching maps induce the zero-map on cohomology. More precisely we prove (see Theorem \ref{main} for a more detailed version) 
\begin{thmc}
{\it Suppose $X$ is a $C_{pq}$-CW-complex with only even cells satisfying the condition: if the cell $V$ is attached after $W$, then $W\ll V$. Then, the cohomology of $X$ is free with generators in one-to-one correspondence with the cells of $X$. }
\end{thmc}
It must be noted from \cite[Remark 2.2]{Lew92} and \cite{FL04}, generalizations of the freeness theorem proved in \cite{Lew88} are not expected for groups like $C_{p^2}$ and $C_p\times C_p$, and thus, also for groups containing these as subgroups. Theorem C demonstrates that results of this type may be proved for other groups which do not contain $C_p\times C_p$ and $C_{p^2}$ for any prime $p$. As an application of the freeness theorem we prove that  the hypothesis is satisfied for certain projective spaces and Grassmannians (Theorem \ref{freeproj}, Theorem \ref{freegra}). 

\subsection{Organization of the paper}
We start by introducing certain preliminaries on $RO(G)$-graded cohomology in Section \ref{prelim}. In Section \ref{mack}, we introduce certain Mackey functors for the group $C_{pq}$  and make some calculations of the set of maps and Ext groups. In section \ref{sph-coh}, we compute the cohomology of spheres in the complex $1$-dimensional representations of $C_{pq}$. Using these we deduce relations in the cohomology of $S^0$, and as an application, we demonstrate some explicit calculations. These facts are used in Section \ref{comp1}, in which the first computations are made in indices whose fixed points are mostly non-zero. Once we have enough computations we prove our first independence Theorem (Theorem A) at the end of Section \ref{comp1}. In Section \ref{comp2}, we complete our computations  and prove the second independence theorem (Theorem B). In Section \ref{constcoeff}, we compute the cohomology for the constant Mackey functor. Finally in Section \ref{free}, we deduce the freeness theorem using the computations in Section \ref{comp1} and \ref{comp2}. 

\mbox{  } \\
{\bf Notation.} Throughout the document $p$ and $q$ denotes distinct odd primes. For the rest of the document $\GG$ denotes the group $C_{pq}$. We also use the notations 
$$\res_p := \res^{C_{pq}}_{C_p} ,~ \res_q := \res^{C_{pq}}_{C_q},~ \res^p := \res^{C_p}_{e},~\res^q := \res^{C_{q}}_{e},$$
$$\tau_p := tr^{C_{pq}}_{C_p},~\tau_q := tr^{C_{pq}}_{C_q},~\tau^p := tr^{C_{p}}_{e},~\tau^q := tr^{C_{q}}_{e}.$$

\section{Preliminaries}\label{prelim}

In this section, we recall certain definitions and notations from equivariant homotopy and cohomology theories. Most of the ideas in this section are from \cite{May96}. We always assume that $G$ is a finite group. 

\subsection{Equivariant homotopy and cohomology}
We work in the category of $G$-spaces. Some natural examples of $G$-spaces are 
$$D(V) :=\{ v \in V : ||v|| \leq 1\}$$ 
and  
$$S(V) :=\{ v \in V : ||v|| =1\}$$
for an orthogonal representation $V$ of $G$. We call these spaces representation-disks and representation-spheres respectively. We also use the notation $S^V$ to denote the one-point compactification of $V$ which is in turn homeomorphic to $S(V\oplus \epsilon)$, where $\epsilon$ is the trivial one-dimensional representation. 

Next, we come to the idea of a $G$-CW-complex. Usually, these are defined as spaces obtained by iteratively attaching cells of the type $G/H\times \DD^n$ along an attaching map out of $G/H\times \partial \DD^n$. For this paper, we use generalized cell complexes defined in \cite{FL04} using cells of the type $G\times_H D(W)$ where $W$ is a representation of the group $H$.  Cells of this type arise naturally in equivariant Morse theory (\cite{Was69}). 
\begin{defn}
A generalized $G$-cell complex $X$ is a $G$-space with a filtration $\{ X_n\}_{n \geq 0}$ of subspaces such that \\
a) $X_0$ is a finite $G$-set. \\
b) For each $n,$ $X_{n+1}$ is built up form $X_n$ by attaching cells of the form $G \times _K D(V)$ along the boundary $G \times _K S(V).$ Here $K$ is a subgroup of $G$, $V$ is a $K$-representation. \\
c) $X = \cup_{n \geq 0} X_n$ has the colimit topology.
\end{defn}

As an example, we note that the $G$-space $S^V$ is a generalized $G$-cell complex via the union $pt \cup G\times_G D(V)$. On the other hand, expressing this as a union of cells of the type $G/H\times \DD^n$ is far more involved. 

In the category of $G$-spaces, one does homotopy theory by defining the set of weak equivalences to be those equivariant maps which induce weak equivalences on all the fixed point spaces. The resulting homotopy theory becomes equivalent to the homotopy theory of diagrams from a category $\OO_G$ to the category of topological spaces. The category $\OO_G$ has $G$-orbits as objects and the $G$-equivariant maps between them as morphisms. 

A coefficient system $\uM$ is defined to be a contravariant functor from $\OO_G$ to Abelian groups. Given a contravariant functor from $\OO_G$ to spaces, taking chains yields such a coefficient system of chain complexes. A particular example is the singular chain complex functor $\underline{C}_\ast(X)$ given by $G/H \mapsto C_\ast (X^H)$ for any $G$-space $X$.  
\begin{defn}
Let $\uM$ be a coefficient system. Define 
$$C^\ast_G(X;\uM) := \Hom_{\OO_G}(\underline{C}_\ast(X),\uM)$$
as the singular co-chain complex with coefficients in $\uM$. The cohomology of this complex is denoted by $H^\ast_G(X;\uM)$ and is called the Bredon cohomology of $X$ with coefficients in $\uM$. 
\end{defn}
These are examples of equivariant ordinary cohomology theories which are graded by integers. However, equivariant cohomology is more naturally graded on $RO(G)$, the real representation ring of $G$, which is necessary to formulate equivariant Poincar\'e duality. To obtain this one has to consider certain special coefficient systems called Mackey functors.  

\subsection{Mackey functors} 
Mackey functors are defined as contravariant functors from the Burnside category (denoted $\BB_G$) to Abelian groups. There is an inclusion $\OO_G \to \BB_G$ which means that every Mackey functor comes with an underlying coefficient system. 

We first introduce an auxiliary category $\BB^{\prime}_G$. The objects  of $\BB^{\prime}_G$ are the finite $G-$sets and the morphisms between two $G$-sets $T$ and $S$ are the isomorphism classes of diagrams of the type
   $$\xymatrix@R-=.25cm@C-=.25cm{ &  &U \ar[dl] \ar[dr]  \\ &T  &  &S}.$$ 
Two diagrams $T \leftarrow U \to S$ and $T \leftarrow V \to S$ are isomorphic if there is a commutative diagram as follows:   
    $$\xymatrix@R-=.25cm@C-=.25cm{&  &U \ar[dl] \ar[dd]^{\cong} \ar[dr]  \\ &T  &  &S \\& &V \ar[ul] \ar[ur]}$$
  The composition of  morphisms is induced by the pullback. 

The morphism sets in $\BB'_G$ are commutative monoids under disjoint union. The Burnside category is obtained by taking the group completion as defined below   
\begin{defn}
The Burnside category $\BB_G$ is a category whose objects are finite $G-$sets and the morphisms between two object $S$ and $T$ is the group completion of $\BB'_G(S,T)$ with respect to disjoint union. 
\end{defn}

One can easily figure out a generating set of morphisms of the Burnside category. Let  $f:S\to T$ be a $G$-map. Then, we obtain two morphisms $f_!$ and $f$ in $\BB_G$ as depicted in the pictures
$$\xymatrix@R-=.25cm@C-=.25cm{ &  &S \ar[dl]_{=} \ar[dr]^{f}  \\ &S  &  &T}  \; \; \text{and} \; \; \xymatrix@R-=.25cm@C-=.25cm{ &  &S \ar[dl]_{f} \ar[dr]^{=}  \\ &T  &  &S}$$
These morphisms generate all the morphisms of the Burnside category by taking disjoint union and compositions. One defines a Mackey functor as an additive contravariant functor from Burnside category $\BB_G$ to the category of Abelian groups $\mathcal{A}b.$ An alternative definition of Mackey functors is obtained by laying out this information in terms of the generating morphisms of the Burnside category. This definition is due to Dress (\cite{Dre72}) and we follow the summary of the definition in \cite{GM95}.  
\begin{defn}
A Mackey functor consists of a pair $\uM = (\uM_\ast , \uM^\ast )$ of functors from the category of finite $G$-sets to $\mathcal{A} b$, with $\uM_\ast$ covariant and $\uM^\ast$ contravariant. On every object $S$, $\uM^\ast$ and $\uM_\ast$ have the same value which we denote by $\uM(S)$,  and $\uM$ carries disjoint unions to direct sums. The functors are required to satisfy that for every pullback diagram of finite $G$-sets as below 
$$\xymatrix{ P \ar[r]^\delta    \ar[d]^\gamma                             & X \ar[d]^\alpha \\ 
                        Y \ar[r]^\beta                                                     &  Z,}$$
one has $\uM^\ast(\alpha) \circ \uM_\ast(\beta) = \uM_\ast(\delta) \circ \uM^\ast(\gamma).$ 
\end{defn}

For subgroups $H\leq K \leq G$, we denote by $\pi^K_H$ the quotient map $\pi^K_H :  G/H \to G/K$. For a Mackey functor $\uM,$ we denote $\res^K_H := \uM^\ast(\pi^K_H)$ and call it the restriction map,   and  $tr^K_H : = \uM_\ast(\pi^K_H)$ and call it the transfer map. Let $^gH = gHg^{-1}$ and $H^g=g^{-1}Hg$. We also have the map $c_g : \uM(G/H) \to \uM(G/^gH)$ induced by left multiplication by $g$ from $G/H \to G/^gH$. Then the last condition above turns into the double-coset formula  
$$\res^K_J tr ^K_H = \sum_{x\in [J \setminus K /H]} tr^{J}_{J\cap^xH} c_x \res^H_{J^x\cap H}$$ 
for all subgroups $J,H \leq K.$

We denote the category of $G$-Mackey functors by $\MM_G$. The objects of this category are the Mackey functors defined above, and the morphisms are group homomorphisms which commute with all the restriction and transfer maps, and the maps $c_g$. 

Obvious examples of Mackey functors are the representable ones. These are also projective in the category of Mackey functors. The Mackey functor $\BB_G( -, G/G)$ is called the Burnside ring Mackey functor and denoted by $\uA$. For every subgroup $H$, $\uA(G/H)$ is the Burnside ring of $H$, which is a free Abelian group on the finite $H$-sets. The restriction maps are induced by restricting actions, and the transfer maps are given by induction. Analogously, we also denote the Mackey functor $\BB_G(-,S)$ by $\uA_S$ for any finite $G$-set $S$. A similar construction works for the representation ring Mackey functor whose value at $G/H$ is the representation ring of $H$, the restriction maps are the usual restriction of representations and the transfer maps are given by induction. 

In this paper, the main examples we consider are the groups $G=C_p$ the cyclic group of order $p$ for $p$ prime, and $\GG$  the cyclic group of order $pq$ for $p$ and $q$ distinct primes. Following Lewis \cite{Lew88}, we denote  a $C_p$-Mackey functor $\uN$ by the picture
$$\xymatrix{   & \uN(C_p/C_p)  \ar@/_.5pc/[dd]_{\res^p}  \\  
\uN :   \\ 
 & \uN(C_p/e) \ar@/_.5pc/[uu]_{\tau^{p}}}$$
so that the Burnside ring Mackey functor looks like 
$$\xymatrix{ \Z^2  \ar@/_.5pc/[d]_{[1\,p]}  \\ \Z \ar@/_.5pc/[u]_{ \left[\begin{smallmatrix} 0 \\ 1\end{smallmatrix}\right] }  & ,  }$$
and a $\GG$-Mackey functor $\uM$ by the picture 
$$\xymatrix{& & \uM(\GG/\GG) \ar@/_1pc/[dl]_{\res_p}  \ar@/^1pc/[dr]^{\res_q}  \\ 
\uM: &  \uM(\GG/C_p) \ar@/_1pc/[ur]^{\tau_p} \ar@/_1pc/[dr]_{\res^p} & &\uM(\GG/C_q) \ar@/^1pc/[ul]_{\tau_q} \ar@/^1pc/[dl]^{\res^q} \\
& &  \uM(\GG/e) \ar@/_1pc/[ul]^{\tau^p} \ar@/^1pc/[ur] _{\tau^q}. }$$

\subsection{$RO(G)$-graded cohomology} 
Let $RO(G)$ denote the real representation ring of $G$. Recall the usual notions of restriction and induction of representations. For $\alpha\in RO(G)$ and a subgroup $H$, we shall  denote the restriction of $\alpha$ in $RO(H)$ by $\res_H(\alpha)$ and often simply as $\alpha$. The trivial part of $\res_H(\alpha)$ is the fixed point space of $\alpha$ whose dimension is denoted by $|\alpha^H|$.  
\begin{defn}
An element $\alpha \in RO(G)$ is called even (respectively, odd) if $|\alpha^{H}|$ is even (respectively, odd) for all subgroups $H \leq G.$ 
\end{defn}
If  $G =\GG (= C_{pq})$ with both $p$ and $q$ distinct odd primes, then every element $\alpha\in RO(\GG)$ is either even or odd, and this is determined by whether $|\alpha|$ is even or odd. 

Equivariant cohomology theories are represented by $G$-spectra. The naive $G$-spectra are those in which only the desuspension with respect to trivial $G$-spheres are allowed. Usually, what we mean by $G$-spectra are those in which desuspension with respect to all representation spheres are allowed. In the viewpoint of \cite{LMS86}, naive $G$-spectra are indexed over a trivial $G$-universe, and $G$-spectra are indexed over a complete $G$-universe. As we are allowed to take desuspension with respect to representation-spheres, the associated cohomology theories become $RO(G)$-graded.    
\begin{defn}
 A $RO(G)$-graded cohomology theory consists of functors $E^\alpha$ for $\alpha\in RO(G)$, from reduced equivariant CW complexes to Abelian groups which satisfy the usual axioms - homotopy invariance, excision, long exact sequence and the wedge axiom. 
\end{defn}
It is interesting to note that the suspension isomorphism for $RO(G)$-graded cohomology theories takes the form 
$E^{\alpha}(X) \cong E^{\alpha + V}(S^V \wedge X)$ for every based $G$-space $X$ and representation $V$. 

Every $G$-set $S$ gives a suspension spectrum $\Sigma^\infty_G S_+$ in the category of $G$-spectra. It turns out that the category with finite $G$-sets as objects and homotopy classes of spectrum maps as morphisms is naturally isomorphic to the Burnside category. Thus, the homotopy groups of $G$-spectra are naturally Mackey functors, and  equivariant Eilenberg MacLane spectra must arise from Mackey functors. This is a Theorem of Lewis, May, and Mcclure which we refer from Chapter XIII of \cite{May96}.  Therefore, we may argue that the integer-graded cohomology associated to coefficient systems extend to $RO(G)$-graded cohomology theories if and only if the coefficient system has an underlying Mackey functor structure. 

We recall that there are change of groups functors on equivariant spectra. The restriction functor from $G$-spectra to $H$-spectra has a left adjoint given by smashing with $G/H_+$. This also induces an isomorphism for cohomology with Mackey functor coefficients 
$$\tilde{H}^\alpha_G(G/H_+\wedge X ; \uM)\cong \tilde{H}^\alpha_H(X; \res_H(\uM))$$
Suppose we use the Burnside ring Mackey functor so that $\res_H(\uA)=\uA$, we have 
$$\tilde{H}^\alpha_G(G/H_+\wedge X ; \uA)\cong \tilde{H}^\alpha_H(X; \uA)$$
In particular, for $H=e$, we have $\tilde{H}^\alpha(G_+\wedge X;\uA) \cong \tilde{H}^{|\alpha|}(X; \Z)$. 

The $RO(G)$-graded theories may also be assumed to be Mackey functor-valued as in the definition below.   
\begin{defn}
Let $X$ be a pointed $G$-space, $\uM$ be any Mackey functor, $\alpha \in RO(G)$. Then the Mackey functor valued cohomology $\uH^{\alpha}_{G}(X;\uM)$ is defined as 
$$\uH^{\alpha}_{G}(X;\uM)(G/K) = \tilde{H}^{\alpha}_{G}({G/K}_+ \wedge X;\uM).$$
The restriction and transfer maps are induced by the appropriate maps of $G$-spectra. 
\end{defn}

We note that our notation for the Mackey functor-valued cohomology uses pointed spaces. We follow this convention throughout so that the cohomology of $G$-points are of the form $\uH^\alpha_G(G/H_+;\uM)$. If $\uM=\uA$, we usually drop the coefficients in the expression. 

With $\uA$ coefficients, or more generally with coefficients in a Green functor (\cite{Lew80}), the Mackey functor-valued cohomology has an induced ring structure. Therefore, for a pointed $G$-space $X$, $\uH^\ast_G(X)$ is an $RO(G)$-graded module over $\uH^\ast_G(S^0)$. We are interested in proving freeness theorems, which study when the cohomology of a generalized cell complex with even cells has ``free" cohomology. This idea is defined below 
\begin{defn}
Let $X$ be a pointed $G$-space. The cohomology $\uH^{\ast}_G(X)$ is said to be free as $\uH^{\ast}_G(S^0)$-module if $\uH^{\ast}_G(X)$ is a direct sum of $\uH^\ast_G(S^0)$-modules of the type $\Sigma^\alpha \uH^{\ast}_G(S_+)$ for a finite $G$-set $S$ and $\alpha \in RO(G)$. 
\end{defn}

\section{Some Mackey functors of $\GG$}\label{mack}

In this section, we introduce certain $\GG$-Mackey functors that will arise as values of $\uH^\alpha_{\GG}(S^0)$. Following this, we make some computations of maps between such functors and Ext groups. 

We start by recalling some $C_p$-Mackey functors from \cite{Lew88}. In the diagrams below the left element denotes the value at $C_p/C_p$ and the right denotes the value at $C_p/e$. We often introduce the subscript $p$ as below to remind ourselves that these are $C_p$-Mackey functors, but later on we shall often drop it to avoid notational complexities. 
$$\xymatrix{    & \Z \oplus \Z \ar@/_.5pc/[dd]_{[\begin{smallmatrix} d & p \end{smallmatrix}]}   && & & \Z  \ar@/_.5pc/[dd]_{Id}    \\  
  \uA_p[d] :      &                                                                                                                              && & R_p :   \\ 
  & \Z \ar@/_.5pc/[uu]_{\left[ \substack{0 \\ 1}\right]}  &&& & \Z \ar@/_.5pc/[uu]_{p}}$$

$$\xymatrix{   & \Z  \ar@/_.5pc/[dd]_{p} &&&  & C  \ar@/_.5pc/[dd] \\ 
L_p :                                                     &&&& \bC_p :       \\
 & \Z \ar@/_.5pc/[uu]_{Id}                    &&&   & 0 \ar@/_.5pc/[uu]      }$$ 
where $C$ is an Abelian group. We note that $\uA_p[1]$ is the usual Burnside ring Mackey functor $\BB_{C_p}(-,C_p/C_p)$ which we denote by $\uA_p$. In addition we consider  $E_p= \BB_{C_p}(-,C_p/e)$ described by the diagram
$$\xymatrix{ & \Z \ar@/_.5pc/[dd]_{\Delta} \\   
E_p : \\ 
& \Z^{ p} \ar@/_.5pc/[uu]_{\nabla}}$$

For the group $\GG$, we draw below the three kinds of representable Mackey functors $\uA_S = \BB_\GG(-,S)$ for $S=\GG/\GG, \GG/C_p,$ and $\GG/e$. 

$$\xymatrix@=1.5em{ & \Z \oplus \Z \ar@/_1pc/[dl]_{\Delta}  \ar@/^1pc/[dr]^{{\tiny \left[ \begin{smallmatrix} 1 &p \end{smallmatrix}\right]}}  &  &&  & \Z \ar@/_1pc/[dl]_{\tiny \Delta}  \ar@/^1pc/[dr]^{\tiny \Delta} & \\
 (\Z \oplus \Z)^q  \ar@/_/[ur]|{\tiny \nabla} \ar@/_1pc/[dr]_{{\tiny \left[\begin{smallmatrix} 1 &p \end{smallmatrix}\right]^ q}} & & \Z \ar@/^/[ul]|{{\tiny \left[\begin{smallmatrix} 0 \\1 \end{smallmatrix}\right]}} \ar@/^1pc/[dl]^{\tiny \Delta} && \Z^q \ar@/_/[ur]|{\tiny \nabla} \ar@/_1pc/[dr]_{\tiny \Delta^p} & & \Z^p\ar@/^/[ul]|{\tiny \nabla} \ar@/^1pc/[dl]^{\tiny \Delta^q} \\ 
&  \Z^{q} \ar@/_/[ul]|(.4){{\tiny \left[\begin{smallmatrix} 0 \\1 \end{smallmatrix}\right]^{q}}} \ar@/^/[ur]|{\tiny \nabla}  &  && & \Z^{pq} \ar@/_/[ul]
|{\tiny \nabla^p} \ar@/^/[ur]|{\tiny \nabla^q}\\
& \uA_{\GG/C_p} &&&&  \uA_{\GG/e}}$$

\vspace{.5cm}

$$\xymatrix{ & &\Z^4 \ar@/_1.1pc/[dll]_{{\tiny \left[\begin{smallmatrix} 1 &0 &q &0 \\ 0 &1 &0 &q \end{smallmatrix}\right]}}   \ar@/^1.1pc/[drr]^{{\tiny \left[\begin{smallmatrix} 1 & p &0 &0 \\ 0 & 0 &1 &p \end{smallmatrix}\right]}}  & &\\ 
 \Z^2 \ar@/_.7pc/[urr]|{{\tiny \left[\begin{smallmatrix} 0 &0 \\0 &0 \\ 1 & 0\\ 0 &1  \end{smallmatrix}\right]}} \ar@/_1.1pc/[drr]_{{\tiny \left[\begin{smallmatrix} 1 & p \end{smallmatrix}\right]}} & & & &\Z^2 \ar@/^.7pc/[ull]|{{\tiny \left[\begin{smallmatrix} 0 &0 \\1 &0 \\ 0 & 0\\ 0 &1  \end{smallmatrix}\right]}} \ar@/^1.1pc/[dll]^{{\tiny \left[\begin{smallmatrix} 1 &q \end{smallmatrix}\right]}} \\ 
 &  &  \Z \ar@/_.7pc/[ull]|{{\tiny \left[\begin{smallmatrix} 0 \\ 1 \end{smallmatrix}\right]}} \ar@/^.7pc/[urr]|{{\tiny \left[\begin{smallmatrix} 0 \\ 1 \end{smallmatrix}\right]}}  & & \\ 
& &  \uA_{\GG/\GG} } $$

Now observe that the Burnside category  $\BB_{\GG}$ is isomorphic to  $\BB_{C_{p}}\otimes \BB_{C_q}$ formed as the product set of objects and tensor product set of morphisms. Thus we may define a $\GG$-Mackey functor by tensoring Mackey functors on $C_p$ and $C_q$. The following Mackey functors have special importance in our case.

\begin{defn}
For a $C_p$-Mackey functor $\uM$, define $\GG$-Mackey functors
$$\Ep\uM : = \uM \otimes E_q,  \,  \QQ_p\uM: = \uM \otimes \bZ_q,\, \cA_p\uM := \uM\otimes \uA_q$$
$$\CC_p\uM: = \uM \otimes L_q, \,  \KK_p\uM : = \uM \otimes R_q.$$
\end{defn}
Also we denote  
$$ R_{pq} := R_p \otimes R_q, \, L_{pq} := L_p \otimes L_q, \, \langle\langle C \rangle \rangle := \langle C \rangle _p \otimes \bZ_q  \cong \bZ_p\otimes \langle C \rangle_q. $$

There is a functor $\Phi_p : \BB_{C_p} \to \BB_{\GG}$ given by inducing up a $C_p$-set to $\GG$ (that is, $S\mapsto S\times \GG/C_p$). This induces $\Phi_p^\ast : \MM_{\GG} \to \MM_{C_p}$. We note the formula
\begin{equation} \label{res}
\Phi_p^\ast(\CC_p \uM) \cong \Phi_p^\ast(\KK_p\uM) \cong \uM,
\end{equation} 
and the result
\begin{prop}\label{epadj}
The functor $\Phi_p^\ast$ is right adjoint to the functor $\Ep$. 
\end{prop}

\begin{proof}
We first define the unit and the counit morphisms and note that the induced maps between $\MM_{\GG}(\Ep\uM,\uN)$ and $\MM_{C_p}(\uM, \Phi_p^\ast \uN)$ are inverse isomorphisms. Note that $\Phi_p^\ast \Ep \uM = \uM^q$ and define $\eta: I \to \Phi_p^\ast \Ep$ is given by $i_1$, the inclusion of the first factor. The composite $\Ep \Phi_p^\ast$ is given by 
$$\Ep \Phi_p^\ast(\uN) (\GG/K)= \uN (\GG/C_p \times \GG/K)$$ 
Define the transformation $\kappa : \Ep \Phi_p^\ast \to I$ as $\kappa(\uN)(\GG/K)=$ the transfer map  $ \uN (\GG/C_p \times \GG/K) \to \uN (\GG/K)$. It follows from the double-coset formula that $\kappa(\uN)$ is indeed a map of Mackey functors. It is easily checked that the induced maps   between $\MM_{\GG}(\Ep\uM,\uN)$ and $\MM_{C_p}(\uM, \Phi_p^\ast \uN)$ are inverse to each other. 
\end{proof}

A similar adjunction holds for the functor $\cA_p$. We define $\rho_p:\BB_{C_p} \to \BB_{\GG}$ by sending a $C_p$-set to the corresponding $\GG$-set with trivial $C_q$-action. We have the corresponding result 
\begin{prop}\label{apadj}
The functor $\rho_p^\ast$ is the right adjoint of $\cA_p$. 
\end{prop}

\begin{proof}
Once again it suffices to define appropriate unit and counit maps and verify that they induce isomorphisms on Hom-sets. We note that $\rho_p^\ast \circ \cA_p$ sends $\uM$ to $\uM^2$ and define the unit $I \to \rho_p^\ast \circ \cA_p$ to be $i_1$, the inclusion on the first factor. For the counit, note that $\cA_p\circ \rho_p^\ast$ evaluated on a Mackey functor $\uM$ gives ($\iota_2$ standing for inclusion onto the second factor)
$$\xymatrix{ & \uM(\GG/\GG)^2  \ar@/_1pc/[dl]_{[I\,qI]}  \ar@/^1pc/[dr]  \\
  \uM(\GG/\GG) \ar@/_1pc/[dr] \ar@/_/[ur]_{\iota_2} & & \uM(\GG/C_q)^2 \ar@/^/[ul] \ar@/^1pc/[dl]^{[I\,qI]} \\
 &  \uM(\GG/C_q) \ar@/_/[ul] \ar@/^/[ur]^{\iota_2} }$$ 
Now define the map $\cA_p\circ \rho_p^\ast \to I$ as $[I\, qI]$ on the top and the right corners and the appropriate restriction maps at the other two corners. A routine check verifies that these maps induce inverse isomorphisms on Hom-sets.  
\end{proof}

We use the Propositions \ref{epadj} and \ref{apadj} to extend some results from $C_p$-Mackey functors to $\GG$-Mackey functors. The Mackey functors of interest are $\CC_p\uM$ and $\KK_p\uM$ where $\uM$ is one of $R_p$, $L_p$ or $\bC$. The Proposition below enlists their key properties as $C_p$-Mackey functors. Our strategy for the proof involves mapping the free Mackey functors $\uA_p$ and $E_p$  onto these. 
\begin{prop}\label{Cpmack} 
a) For all groups $C$, $\MM_{C_p}(L_p, \bC)=0$ and $\MM_{C_p}(R_p,\bC)=\{a\in C \mid pa=0\}$.\\
b) For all groups $C$,  
$$\Ext^1_{\BB_{C_p}}(\uM,\bC)=\begin{cases} 0 &\mbox{if}~\uM=\bZ~\mbox{or}~L_p \\ 
                                                                               C/pC &\mbox{if}~\uM=R_p  \end{cases} $$
c) $\Ext^1_{\BB_{C_p}}(\bZ,L_p)\cong \Z/p$.
\end{prop}

\begin{proof} 
The first statement of a) is clear from the fact that the transfer map in $L_p$ is an isomorphism, and that of $\bC$ is $0$. For $R_p$, note that $p\in R_p(C_p/C_p)$ is in the image of the transfer map, and the result follows from this observation. For b) we consider the three cases separately. Start with $\bZ$ and notice that the map $\uA_p\to \bZ$ induced by the generator of $\Z$ is surjective with kernel $L_p$. Therefore, we have the short exact sequence 
$$ 0 \to L_p \to \uA_p \to \bZ \to 0$$
leading to the long exact sequence
$$\cdots \to \MM_{C_p}(L_p,\bC) \to \Ext^1_{\BB_{C_p}}(\bZ,\bC) \to \Ext^1_{\BB_{C_p}}(\uA_p, \bC) \to \cdots $$
From a) we have $\MM_{C_p}(L_p,\bC)=0$ and $\Ext^1_{\BB_{C_p}}(\uA_p,\bC)=0$ as $\uA_p$ is projective as a Mackey functor. Thus, $ \Ext^1_{\BB_{C_p}}(\bZ,\bC)=0$. For $L_p$, we note the map $E_p \to L_p$ induced by $1\in \Z=L_p(C_p/e)$ is surjective and we denote the kernel by $\kappa_p$. Thus, $\kappa_p(C_p/C_p)=0$ and $\kappa_p(C_p/e) = \Z^{p-1}$. Therefore, we have the exact sequence
$$\cdots \to \MM_{C_p}(\kappa_p,\bC) \to \Ext^1_{\BB_{C_p}}(L_p,\bC) \to \Ext^1_{\BB_{C_p}}(E_p, \bC) \to \cdots $$
The Mackey functor $E_p$ is projective as it is representable, hence $\Ext^1_{\BB_{C_p}}(E_p, \bC)=0$. We also have $\MM_{C_p}(\kappa_p,\bC)=0$ as for a map $\kappa_p \to \bC$ at each level either the domain is $0$ or the range is $0$. Thus, $\Ext^1_{\BB_{C_p}}(L_p,\bC)=0$. This leaves the case  $\uM= R_p$ for which we use the short exact sequence 
$$0 \to \bZ \to \uA_p \to R_p \to 0$$ 
induced by $1 \in \Z= R_p(C_p/C_p)$. The associated long exact sequence takes the form 
$$\cdots \to \MM_{C_p}(\uA_p,\bC) \to \MM_{C_p}(\bZ,\bC) \to \Ext^1_{\BB_{C_p}}(R_p,\bC) \to 0$$
 Note that both $\MM_{C_p}(\uA_p,\bC)$ and $\MM_{C_p}(\bZ,\bC)$ are $C$ and the map in the above sequence is multiplication by $p$.  Therefore, $ \Ext^1_{\BB_{C_p}}(R_p,\bC)\cong C/pC$.

Finally, for c), we use the short exact sequence $0\to L_p \to \uA_p \to \bZ\to 0$ and obtain the exact sequence 
$$  \MM_{C_p}(\uA_p,L_p)\to \MM_{C_p}(L_p,L_p) \to \Ext^1_{\BB_{C_p}}(\bZ,L_p) \to \Ext^1_{\BB_{C_p}}(\uA_p,L_p) $$
The last term is $0$ as $\uA_p$ is projective. The first term is isomorphic to $L_p(C_p/C_p)\cong \Z$. We note that $\MM_{C_p}(L_p,L_p)\cong \Hom(L_p(C_p/e), L(C_p/e))\cong \Z$ as the transfer maps are isomorphisms. We note that a map $\uA_p \to L_p$ which is induced by $d\in \Z=L_p(C_p/C_p)$ induces the map $pd : \uA_p(C_p/e) \to L_p(C_p/e)$. Therefore $\Ext^1_{\BB_{C_p}}(\bZ,L_p)$ which equals the cokernel of $ \MM_{C_p}(\uA_p,L_p)\to \MM_{C_p}(L_p,L_p) $ is $\Z/p$.
\end{proof}

The next result deals with maps between some $\GG$-Mackey functors.
\begin{prop}\label{Cpqmap}
a) Any map of Mackey functors from $\CC_p\uM$ to $\uN$ is $0$ if $\uN(\GG/C_p)$  and $\uN(\GG/e)$ is $0$. \\
b) A map $L_{pq} \to \uN$ is $0$ if $\uN(\GG/e) =0$. \\
c) Any map from $R_{pq}$ to $\bbC$ is $0$ for any $C$. 
\end{prop}

\begin{proof}
In the short exact sequence $0 \to \kappa_q \to E_q \to L_q$, the Mackey functors are levelwise free, hence, this  induces a short exact sequence of $\GG$-Mackey functors 
$$0 \to \uM\otimes \kappa_q \to \Ep \uM \to \CC_p \uM \to 0.$$
Thus, $\MM_{\GG}(\CC_p\uM,\uN)$ injects into $\MM_{\GG}(\Ep\uM,\uN) \cong \MM_{C_p}(\uM, \Phi_p^\ast \uN)$. Now the given condition in a) implies $\Phi_p^\ast\uN=0$, and so, a) follows. The statement b) follows from the fact that the transfer maps in $L_{pq}$ are isomorphisms, and so, every element in each level is obtained by transfer on some class in $L_{pq}(\GG/e)$. Finally for the statement c), note that the image $a \in C$  of $1\in \Z = R_{pq} (\GG /\GG)$ must satisfy both $pa = 0$ and $qa= 0$ since $\tau_p (\res_p(1)) = q$ and $\tau_q (\res_q(1))= p$. This implies that $a=0$, as $p$ and $q$ are relatively prime.  
\end{proof}

We follow this with some computations of Ext groups. In the later sections, we will use these to deduce that certain short exact sequences are split exact. 
\begin{prop}\label{Cpqext}
a) $\Ext^1_{\BB_{\GG}}(\CC_p\uM,\uN)=0$ if $\uN(\GG/C_p)$ and $\uN(\GG/e)$ are $0$. \\
b) For all groups $C$, 
$$\Ext^1_{\BB_{\GG}}(\KK_p\uM, \bbC)=\begin{cases} 0             &\mbox{if}~\uM=R_p~\mbox{or}~L_p \\
                                                                                                     C/qC       &\mbox{if}~\uM=\bZ \end{cases} $$
c) $\Ext^1_{\BB_{\GG}}(\bbZ,L_{pq})=0$. 
\end{prop}

\begin{proof}
We observe that both the functors $\Ep$ and $\cA_p$ preserve exact sequences and take the representables $\uA_p$ and $E_p$ to representables. Thus, we deduce 
\begin{equation}\label{ident}
\Ext^n_{\BB_{\GG}}(\Ep\uM, \uN) \cong \Ext^n_{\BB_{C_p}}(\uM,\Phi_p^\ast \uN), \, \Ext^n_{\BB_{\GG}}(\cA_p\uM, \uN) \cong \Ext^n_{\BB_{C_p}}(\uM,\rho_p^\ast \uN). 
\end{equation}

For a), we use the exact sequence 
$$0 \to \uM\otimes \kappa_q \to \Ep \uM \to \CC_p \uM \to 0$$
which induces the long exact sequence 
$$\cdots \to \MM_{\GG}(\kappa_q\otimes \uM,\uN) \to \Ext^1_{\BB_{\GG}}(\CC_p\uM,\uN) \to \Ext^1_{\BB_{\GG}}(\Ep\uM, \uN) \to \cdots $$
We observe that $ \MM_{\GG}(\kappa_q\otimes \uM,\uN)$ is $0$ as levelwise either the domain or the range is $0$ from the given conditions. We also have from \eqref{ident} that $\Ext^1_{\BB_{\GG}}(\Ep\uM, \uN) \cong \Ext^1_{\BB_{C_p}}(\uM,\Phi_p^\ast \uN)$ which is $0$ by the given hypothesis. Thus, we obtain a).  

For b), we note that the exact sequence $0\to \bZ \to \uA_q \to R_q$ induces the exact sequence 
$$0 \to \QQ_p\uM \to \cA_p \uM \to \KK_p \uM \to 0$$
and hence the long exact sequence
$$  \MM_{\GG}(\cA_p \uM,\bbC)\to \MM_{\GG}(\QQ_p \uM,\bbC) \to \Ext^1_{\BB_{\GG}}(\KK_p\uM,\bbC) \to \Ext^1_{\BB_{\GG}}(\cA_p\uM, \bbC)  $$
We note that $\MM_{\GG}(\QQ_p\uM, \QQ_p\uN)\cong \MM_{C_p}(\uM,\uN)$, $\bbZ \cong \QQ_p\bZ$, and $\rho_p^\ast \bbC =\bC$. Therefore, we deduce 
 $$\MM_{\GG}(\QQ_p\uM,\bbC)\cong \MM_{C_p}(\uM,\bC), ~ \, \MM_{\GG}(\cA_p \uM, \bbC) \cong \MM_{C_p}(\uM, \bC).$$ 
The map $\QQ_p \uM \to \cA_p \uM$ induces multiplication by $q$ on $\MM_{C_p}(\uM,\bC)$ after applying the adjunctions above. We proceed in the three cases $\uM=\bZ, L_p$ and $R_p$ separately. For $\bZ$ and $L_p$, we have from \eqref{ident} that $\Ext^1_{\BB_{\GG}}(\cA_p\uM, \bbC)  \cong \Ext^1_{\BB_{C_p}}(\uM, \bC)$ which is $0$ by Proposition \ref{Cpmack} b).  Now the result for $L_p$ and $\bZ$ follows from $\MM_{C_p}(L_p,\bC)=0$ and $\MM_{C_p}(\bZ,\bC)=C$. 

For $R_p$, note that $\MM_{C_p}(R_p,\bC)$ is $p$-torsion, so that multiplication by $q$ is an isomorphism. Therefore, the map $\Ext^1_{\BB_{\GG}}(\KK_p R_p,\bbC) \to \Ext^1_{\BB_{\GG}}(\cA_pR_p, \bbC) $ is injective and the latter group can be computed using \eqref{ident}. We have  $\Ext^1_{\BB_{\GG}}(\cA_pR_p, \bbC)  \cong \Ext^1_{\BB_{C_p}}(R_p, \bC)\cong C/pC$ by Proposition \ref{Cpmack} b). We also have $\KK_pR_p\cong R_{pq} \cong \KK_qR_q$, so the same argument implies that $\Ext^1_{\BB_{\GG}}(\KK_pR_p, \bbC)$ also injects into $C/qC$. This implies $\Ext^1_{\BB_{\GG}}(\KK_pR_p,\bbC)=0$.  

Finally, we prove c). We have the short exact sequence 
$$0\to \CC_p \bZ \to \cA_p\bZ \to \bbZ \to 0$$
 which induces a long exact sequence 
$$\cdots \to \MM_{\GG}(\CC_p \bZ,L_{pq}) \to \Ext^1_{\BB_{\GG}}(\bbZ,L_{pq}) \to \Ext^1_{\BB_{\GG}}(\cA_p\bZ, L_{pq}) \to \cdots $$
We note that $\MM_{\GG}(\CC_p \bZ,L_{pq})=0$ as the restriction maps in $L_{pq}$ are injective and $\CC_p\bZ(\GG/e)=0$. We also have 
$$\Ext^1_{\BB_{\GG}}(\cA_p\bZ, L_{pq}) \cong \Ext^1_{\BB_{C_p}}(\bZ, \rho_p^\ast L_{pq})\cong \Ext^1_{\BB_{C_p}}(\bZ, L_p)$$
which is $\Z/p$ by Proposition \ref{Cpmack}. Hence $ \Ext^1_{\BB_{\GG}}(\bbZ,L_{pq})$ injects into $\Z/p$. Analogously using the short exact sequence 
$$0\to \CC_q \bZ \to \cA_q\bZ \to \bbZ \to 0,$$
 we observe that  $ \Ext^1_{\BB_{\GG}}(\bbZ,L_{pq})$ injects into $\Z/q$. Therefore,  $ \Ext^1_{\BB_{\GG}}(\bbZ,L_{pq})=0$. 
\end{proof}

We end the section with a crucial observation about the Mackey functors $\CC_p\bZp$ and $\KK_p\bZp$.  We note that as Mackey functors there is an isomorphism $\CC_p\bZp \cong \KK_p\bZp$. Moreover, if we have a map of Mackey functors 
$$\phi: \CC_p\bZp \to \uM$$ 
with the property that the map $\CC_p\bZp(\GG/C_p)\to \uM(\GG/C_p)$ is a split injection, then the map $\phi$ is a split inclusion. Suppose $\tau:\uM(\GG/C_p)\to \Z/p$ is the splitting. We may define a map $\psi:\uM \to \CC_p\bZp$ by $\psi(\GG/C_p)=\tau$,
$$\psi(\GG/\GG)= q^{-1}\tau_p(\CC_p\bZp) \circ \psi(\GG/C_p)\circ \res_p(\uM),$$ and $\psi=0$ at other levels. Similarly, any map of Mackey functors $\psi:\uM \to \KK_p\bZp$
which is a split surjection at level $\GG/C_p$ is a split surjection. We summarize the result in the Proposition below 
\begin{prop}\label{cpzp}
a) A map of Mackey functors $\CC_p\bZp \to \uM$ which is a split injection at level $\GG/C_p$ is a split injection of Mackey functors. \\
b) A map of Mackey functors $ \uM\to \KK_p\bZp$ which is a split surjection at level $\GG/C_p$ is a split surjection of Mackey functors. \\
\end{prop}

\section{Cohomology of representation spheres}\label{sph-coh}
In this section, we compute the cohomology of representation spheres. These play a crucial role in computing the cohomology of a point via the cofibre sequences 
$$S(V)_+ \to S^0 \to S^V$$
for any representation $V$ of $\GG$. In the associated long exact sequence, the cohomologies of $S^0$ and $S^V$ encode values of cohomology of a point. Therefore, if we know the value of $\uH^\ast_{\GG}(S(V)_+)$, then we deduce a relation between $\uH^{\ast-V}_{\GG}(S^0)$ and $\uH^\ast_{\GG}(S^0)$.

 The irreducible representations of $\GG$ are given by $\xi^k$ where $\xi$ is the $2$-dimensional representation induced by the $pq^{th}$-roots of unity. We compute the cohomology of $S(\xi^k)$ for various $k.$ This splits into three cases: $k$ relative to $pq,$ $k$ divisible by $p$ and $k$ divisible by $q.$
\begin{prop}\label{AG} \mbox{  } \\
 a) $\begin{displaystyle}
  \underline{H}^{\alpha}_{\GG}({\GG/e}_+; \uA) = 
  \begin{cases}
    \uA_{\GG/e}, & \mbox{for } |\alpha| = 0 \\
    0,  &\mbox{Otherwise}
    
 \end{cases}
\end{displaystyle}$\\
b) For $(j,pq) =1,$
$$
  \underline{H}^{\alpha}_{\GG}({S(\xi^j)}_+; \uA) = 
  \begin{cases}
    R_{pq}, & \mbox{for } |\alpha| = 0 \\
    
    L_{pq}, & \mbox{for } |\alpha| = 1 \\
    0,  &\mbox{Otherwise}
    
  \end{cases}
$$
\end{prop}

\begin{proof} To prove a) we use the following reductions
\begin{align*}
\uH^{\alpha}_{\GG}({\GG/e}_+; \uA)(\GG/H)  &\cong \tilde{H}^\alpha_{\GG}(\GG/e_+ \wedge \GG/H_+;\uA)  \\ 
                                                                                                            &\cong \tilde{H}^{|\alpha|}({\GG/H}_+; \Z)
\end{align*}
Hence a) follows. For b) consider the cofibre sequence of $\GG$-spectra 
\begin{equation}\label{cofe}
\xymatrix@C-=0.50cm{ {\GG/e}_{+} \ar[r]^{1-g} &{\GG/e}_{+} \ar[r] &S(\xi^j)_{+}}
\end{equation}
 for some generator $g$ of $\GG$. This induces the long exact sequence
$$\xymatrix@C-=0.25cm{\cdots \ar[r] &\underline{H}^{\alpha -1}_{\GG}({\GG/e}_+) \ar[r] &\underline{H}^{\alpha}_{\GG}({S(\xi^j)}_+) \ar[r] &\underline{H}^{\alpha}_{\GG}({\GG/e}_+) \ar[r] &\cdots}$$
which implies the short exact sequence
$$0 \to \coker (1-\uH_{\GG}^{\alpha -1}(g)) \to \uH^{\alpha}_{\GG}(S(\xi^j)_+) \to \ker(1-\uH_{\GG}^\alpha(g)) \to 0. $$
 From $a),$ the cohomology of $\underline{H}^{\alpha}_{\GG}(S(\xi^j)_+)$ is non zero only when $|\alpha|$ is {0} or $1.$ We further compute the kernel and cokernel factors by calculating 
$$\uH_{\GG}^{\alpha}(g)\colon \underline{H}^{\alpha}_{\GG}({\GG/e}_+)  \to \underline{H}^{\alpha}_{\GG}({\GG/e}_+).$$ 
The identification of $\underline{H}^{\alpha}_{\GG}({\GG/e}_+)$ with $\underline{H}_{\GG}^{0}({\GG/e}_+)$ is natural, so it suffices to compute $\underline{H}_{\GG}^{\alpha}(g)$ on $\underline{H}_{\GG}^0$. Thus we have,
$$ (1-\underline{H}_{\GG}^{\alpha}(g))(\GG/H)  = 
  \begin{cases}
    0 , & \mbox{for } H =\GG \\ 
     1 - \sigma_q, & \mbox{for } H = C_p \\
     1 - \sigma_p, & \mbox{for } H = C_q \\
     1 - \sigma_{pq}, & \mbox{for } H = \{e\} 
   
  \end{cases}$$
Where $\sigma_k$ is the cycle $(1,2, \cdots, k)$ in $\Sigma_k$, the symmetric group of $k$-elements. The result follows.
\end{proof}

We have the following cofibre sequence 
\begin{equation}\label{cof-e}
S(\xi)_+ \to S^0 \to S^\xi
\end{equation} 
which leads to  the long exact sequence 
\begin{equation}\label{ex-e}
\cdots \to \uH^{\alpha -1}_{\GG}(S(\xi)_+) \to \uH^{\alpha-\xi}_{\GG}(S^0) \to \uH^{\alpha}_{\GG}(S^0) \to \uH^\alpha_{\GG}(S(\xi)_+) \to \cdots
\end{equation}
For computations using this exact sequence, we need to understand the maps 
$$\uH^{\alpha +\xi -1}_{\GG}(S(\xi)_+)\stackrel{\delta_\xi}{\to} \uH^{\alpha}_{\GG}(S^0)$$ 
and 
$$\uH^\alpha_{\GG}(S^0)\stackrel{\pi_\xi}{\to} \uH^{\alpha}_{\GG}(S(\xi)_+).$$ 
The following Proposition is useful in this regard. 
\begin{prop} \label{sus-e} 
a) If $|\alpha|=0$, the map $\delta_\xi(\GG/e)$ is an isomorphism. In other cases, $\delta_\xi(\GG/e) =0$.\\
b) If $|\alpha|=0$, the map $\pi_\xi(\GG/e)$ is an isomorphism. In other cases, $\pi_\xi(\GG/e) =0$.\\
\end{prop}
\begin{proof}
The proof relies on the natural isomorphism $\uH^{\alpha}_{\GG}(X;\uA)(\GG/e)\cong \tilde{H}^{|\alpha|}(X;\Z)$. Non-equivariantly, the cofibre sequence \eqref{cof-e} reduces to the cofibre 
$$S^1_+\stackrel{\pi}{\to} S^0 \to S^2.$$
 The result now follows from the fact that $\pi$ induces an isomorphism on $\tilde{H}^0$ and the attaching map is an isomorphism on $\tilde{H}^1$.   
\end{proof}

This leads to the following result by applying Proposition \ref{AG}. 

\begin{prop}\label{shex-e} 
a) If $|\alpha|\geq 1$ or $|\alpha|\leq -3$,
$$\uH^{\alpha}_{\GG}(S^0)\cong \uH^{\alpha+\xi}_{\GG}(S^0).$$
b) If $|\alpha| = 0$, there is a short exact sequence of Mackey functors 
$$0\to L_{pq} \to \uH^\alpha_{\GG}(S^0) \to \uH^{\alpha + \xi}_{\GG}(S^0)\to 0$$
As a consequence, if $|\alpha|=-1$ and $\uH^\alpha_{\GG}(S^0)=0$, then $\uH^{\alpha+\xi}_{\GG}(S^0)=0$. \\
c) If $|\alpha|=-2$, as groups $\tilde{H}^{\alpha+\xi}_{\GG}(S^0)\cong \tilde{H}^\alpha_{\GG}(S^0)\oplus \Z$. \\
The results hold with $\xi$ replaced by $\xi^k$ where $k$ is prime to $pq$. 
\end{prop}

\begin{proof}
For the first statement note that $\uH^{\alpha +\xi -1}_{\GG}(S(\xi)_+)$ and $\uH^{\alpha + \xi}_{\GG}(S(\xi)_+)$ are $0$ if $|\alpha|\geq 1$ or $|\alpha|\leq -3$. Therefore a) follows from the exact sequence \eqref{ex-e}. When $|\alpha|=0$, $\uH^{\alpha +\xi -1}_{\GG}(S(\xi)_+)=L_{pq}$ and $\uH^{\alpha + \xi}_{\GG}(S(\xi)_+)=0$ by Proposition \ref{AG}. The map $\uH^{\alpha +\xi -1}_{\GG}(S(\xi)_+) \to \uH^\alpha_{\GG}(S^0)$ in the exact sequence \eqref{ex-e} is an isomorphism when evaluated at $\GG/e$ by Proposition \ref{sus-e}. As the restriction maps in $L_{pq}$ are all injective it follows that the map from $L_{pq}$ to  $\uH^\alpha_{\GG}(S^0)$ is indeed injective. Hence b) follows. To prove c) we consider \eqref{ex-e} indexed at $\alpha + \xi$ for $|\alpha|=-2$ and obtain
$$0\to \uH^\alpha_{\GG}(S^0)  \to \uH^{\alpha+\xi}_{\GG}(S^0) \to R_{pq} \to \cdots $$
The map  $ \uH^{\alpha+\xi}_{\GG}(S^0) (\GG/e) \to R_{pq}(\GG/e)$ is an isomorphism by Proposition \ref{sus-e}. We note that the image of 
$$tr_e^{\GG} : R_{pq}(\GG/e) (\cong \Z) \to R_{pq}(\GG/\GG) (\cong \Z)$$
is $pq\Z$. It follows that the map $ \tilde{H}^{\alpha+\xi}_{\GG}(S^0) \to R_{pq}(\GG/\GG)$ has image isomorphic to $\Z$. Hence c) follows. 
\end{proof}

In our computations of $\uH^\alpha_{\GG}(S^0)$, we write this Mackey functor up to isomorphism as a function of $|\alpha|$, $|\alpha^{C_p}|$, $|\alpha^{C_q}|$ and $|\alpha^{\GG}|$. In many cases, there is actually such a function. In Theorem \ref{ind}, we prove that the group $\tilde{H}^\alpha_{\GG}(S^0)$ is indeed a function of the four fixed points.   

\begin{remark}\label{cor-e}
We note that changing $\alpha$ by adding or subtracting copies of $\xi$ changes the values of $|\alpha|$ while keeping the values of the other fixed points unchanged. Therefore, Proposition \ref{shex-e} implies that the computation for $\uH^\alpha_{\GG}(S^0)$ for fixed values of $|\alpha^{C_p}|$, $|\alpha^{C_q}|$ and $|\alpha^{\GG}|$ splits into three parts : $|\alpha|\gg 0$, $|\alpha| \ll 0$, and $|\alpha|=0$. As $\xi$ changes $|\alpha|$ by an even number, one must also consider the cases $\alpha$ even and $\alpha$ odd separately among the three cases. In the following relations we shall use notations like `` $|\alpha| \gg 0~\mbox{even}$'' to mean the specific case that $|\alpha|$ is a large positive even number.  
\end{remark}

In order to compute the cohomology of the other representation spheres, one needs to know the values of $\uH^{\alpha}_{\GG}({\GG/C_p}_+)$ whose values may be computed from $\uH^{\alpha}_{C_p}(S^0)$ as we shall see below. In this regard, we have the following computation from \cite[ Theorem 4.9]{Lew88}.
\begin{prop} \label{Lewis orbits}
$$
  \uH^{\alpha}_{C_{p}}(S^0) = 
  \begin{cases}
      \uA[d_{\alpha}], & \mbox{for } |\alpha| = 0~\mbox{and}   \; |\alpha^{C_p}| = 0\\
     R, & \mbox{for } |\alpha| = 0 ~\mbox{and}  \; |\alpha^{C_p}| < 0\\
     L, & \mbox{for } |\alpha| = 0 ~\mbox{and}  \; |\alpha^{C_p}| > 0\\
   \bZ, & \mbox{for } |\alpha| \neq 0  ~\mbox{and} \; |\alpha^{C_p}| = 0\\
      \bZp, & \mbox{for } |\alpha| > 0 ~\mbox{and}  \; |\alpha^{C_p}| < 0 \;\mbox{and} \; |\alpha^{C_{p}}| \; \mbox{is even}\\
     \bZp, & \mbox{for } |\alpha| < 0  ~\mbox{and} \; |\alpha^{C_p}| > 1 \;\mbox{and} \; |\alpha^{C_{p}}| \; \mbox{is odd}\\
   
    0,  &\mbox{Otherwise}.
    
  \end{cases}
$$
\end{prop}

Observe that the formula for $\uH^\alpha_{C_p}(S^0)$ depends only on $|\alpha|$ and $|\alpha^{C_p}|$ except in the case when both the dimensions are $0$. In that case, the Mackey functor also depends on the integer $d_\alpha$ which is well-defined in $\F_p^\times/\{\pm 1\}$. An element $\alpha \in RO(C_p)$ with $|\alpha|=0$ and $|\alpha^{C_p}|=0$ may be expressed as $\sum_i \xi^{k_i} - \xi^{l_i}$, where $\xi$ is the irreducible representation on which the generator acts via multiplication by $e^\frac{2\pi i}{p}$. In this case $d_\alpha \equiv \prod_i k_i l_i^{-1} \pmod p$.  

 We have the following computation for the cohomology of spheres $S(\xi^k)$ when $k$ is divisible by $p$.
\begin{prop} \label{G/Cp} a) $\Phi_p^\ast \uH^\alpha_{\GG}(S^0) \cong \uH^\alpha_{C_p}(S^0)$.  \\
b) $\uH^{\alpha}_{\GG}({\GG/C_p}_+; \uA) = \Ep\uH^{\alpha}_{C_p}(S^0; \uA).$ \\
c) For $k$ divisible by $p$ but relatively prime to $q$ we have, 
$$\uH^\alpha_{\GG}({S(\xi^k)}_+; \uA) \cong \CC_p(\uH^{\alpha -1}_{C_p}(S^0;\uA)) \oplus \KK_p(\uH^\alpha_{C_p}(S^0;\uA)).$$
\end{prop}

\begin{proof}
The statement a) follows from the natural isomorphism
$$\tilde{H}^{\alpha}_{\GG}({\GG/C_p}_+ \wedge X) \cong \tilde{H}^{\res_{C_p}(\alpha)}_{C_p}(X).$$
We note that this isomorphism identifies the groups $\uH^\alpha_{\GG}(\GG/C_p)(S)$ with $\Ep(\uH^\alpha_{C_p}(S^0))(S)$ as groups. The restriction and transfer maps  for $\GG/C_q \to \GG/\GG$ may be read off from the corresponding $C_p$-equivariant maps. The other restrictions and transfer may be computed as in Proposition \ref{AG} to deduce b).

For the cohomology of the sphere, we use the cofibre sequence of $\GG$-spectra 
\begin{equation}\label{cofp}
{\GG/C_p}_+ \stackrel{1-g}{\to} \GG/{C_p}_+ \to S(\xi^k)_+ 
\end{equation}
for some generator $g$ of $C_p \subseteq \GG/C_p$. This induces the long exact sequence 
$$ \cdots \to \underline{H}^{\alpha -1}_{\GG}({\GG/C_p}_+) \to \underline{H}^{\alpha}_{\GG}({S(\xi^k)}_+) \to \underline{H}^{\alpha}_{\GG}({\GG/C_p}_+) \to \cdots$$
and hence the short exact sequence,
$$0\to \coker(1-\uH_{\GG}^{\alpha-1}(g)) \to \uH^{\alpha}_{\GG}(S(\xi^k)_+) \to  \ker(1-\uH_{\GG}^\alpha(g)) \to 0 $$

We deduce that $\ker(1-\uH_{\GG}^\alpha(g))$ and $\coker(1-\uH_{\GG}^{\alpha-1}(g))$ are $\mathcal{K}_p \underline{H}^{\alpha}_{C_{p}}(S^0)$ and $\mathcal{C}_p \underline{H}^{\alpha-1}_{C_{p}}(S^0)$ respectively analogously as in the proof of Proposition \ref{AG}.

It remains to check that the above sequence splits as a direct sum. Observe that both terms are non-zero only when $|\alpha|=0$ and $|\alpha^{C_p}|>2$. In this case, we have
$$0 \to \CC_p \bZp \to  \uH^{\alpha}_{\GG}({S(\xi^p)}_+) \to  \KK_p L \to 0.$$
As  $\KK_pL(\GG/C_p)\cong \Z$, the map $\CC_p\bZp \to  \uH^{\alpha}_{\GG}({S(\xi^p)}_+) $ is a split injection at $\GG/C_p$. Thus, we are done by Proposition \ref{cpzp}. 
\end{proof}

Using Proposition \ref{G/Cp} and Proposition \ref{Lewis orbits}, we may write down the explicit expression for $\uH^\alpha_{\GG}(S(\xi^k)_+)$ restricting to the important cases $|\alpha|\gg 0$ and $|\alpha| \ll 0$. 
\begin{cor}\label{p-sphere}
Let $\alpha \in RO(\GG)$ be such that $|\alpha|\neq 0$ or $1$, and $k$ be an integer prime to $q$ and divisible by $p$. Then,
$$
  \underline{H}^{\alpha}_{\GG}(S(\xi^k)_+) = 
  \begin{cases}
      \KK_p \bZ, & \text{if}~ |\alpha| >0,   \; |\alpha^{C_p}| = 0\\
    \CC_p \bZ, & \text{if} ~|\alpha| >0,   \; |\alpha^{C_p}| =1 \\
    \CC_p \bZp, & \text{if} ~|\alpha| >0,   \; |\alpha^{C_p}| <0, \;\text{and} \; |\alpha^{C_{p}}| \; \text{is odd}\\
    \KK_p \bZp ,  & \text{if} ~|\alpha| >0,   \; |\alpha^{C_p}| <0, \;\text{and} \;|\alpha^{C_{p}}| \; \text{is even}\\
    \CC_p \bZ, & \text{if} ~|\alpha| <0,   \; |\alpha^{C_p}| = 1\\
    \KK_p \bZ, & \text{if} ~|\alpha| <0,   \; |\alpha^{C_p}| = 0\\
    \CC_p \bZp,  & \text{if} ~|\alpha| < 0,  \; |\alpha^{C_p}| >2,\;\text{and} \; |\alpha^{C_{p}}| \; \text{is even}\\
    \KK_p \bZp \rangle,  & \text{if} ~|\alpha| <0,   \; |\alpha^{C_p}| > 2,\;\text{and} \; |\alpha^{C_{p}}| \; \text{is odd}\\
    0,  &\text{Otherwise}
    
  \end{cases}
$$
In particular if $(|\alpha|\geq 2, |\alpha^{C_p}|\geq 2)$ or $(|\alpha|\leq -1, |\alpha^{C_p}|\leq -1)$, $\uH^\alpha_{\GG}(S(\xi^k)_+)=0$. 
\end{cor}

Proceeding analogously as in the case of $\xi$, we consider the cofibre sequence 
\begin{equation}\label{cof-p}
S(\xi^p)_+ \to S^0 \to S^{\xi^p}
\end{equation} 
which leads to  the long exact sequence 
\begin{equation}\label{ex-p}
\cdots \to \uH^{\alpha -1}_{\GG}(S(\xi^p)_+) \to \uH^{\alpha-\xi^p}_{\GG}(S^0) \to \uH^{\alpha}_{\GG}(S^0) \to \uH^\alpha_{\GG}(S(\xi^p)_+) \to \cdots
\end{equation}
and thus, want to compute the maps 
$$\uH^{\alpha +\xi^p -1}_{\GG}(S(\xi^p)_+)\stackrel{\delta_{\xi^p}}{\to} \uH^{\alpha}_{\GG}(S^0)$$ 
and 
$$\uH^\alpha_{\GG}(S^0)\stackrel{\pi_{\xi^p}}{\to} \uH^{\alpha}_{\GG}(S(\xi^p)_+).$$ 
In the following Proposition, we compute $\Phi_p^\ast(\delta_{\xi^p})$ and $\Phi_p^\ast(\pi_{\xi^p})$. Note that $\Phi_p^\ast(\uH^\alpha_{\GG}(S^0)) \cong \uH^\alpha_{C_p}(S^0)$. 
\begin{prop} \label{sus-p} 
a) The map $\Phi_p^\ast\delta_{\xi^p} : \Phi_p^\ast \uH^{\alpha +\xi^p -1}_{\GG}(S(\xi^p)_+) \to   \uH^{\alpha}_{C_p}(S^0) $ expressed using the identification
\begin{align*}
 \Phi_p^\ast \uH^{\alpha +\xi^p -1}_{\GG}(S(\xi^p)_+) & \cong \Phi_p^\ast (\CC_p\uH^{\alpha+\xi^p-2}_{C_p}(S^0) \oplus \KK_p\uH^{\alpha+\xi^p-1}_{C_p}(S^0)) \\
       &  \cong \uH^\alpha_{C_p}(S^0) \oplus \uH^{\alpha+1}_{C_p}(S^0)   
\end{align*} 
is an isomorphism on the first factor and $0$ on the second factor.\\
b) The map $\Phi_p^\ast\pi_{\xi^p}$
$$ \uH^\alpha_{C_p}(S^0) \to \Phi_p^\ast \uH^\alpha_{\GG}(S(\xi^p)_+)\cong \Phi_p^\ast (\CC_p\uH^{\alpha- 1}_{C_p}(S^0) \oplus \KK_p\uH^{\alpha}_{C_p}(S^0)) \cong \uH^{\alpha-1}_{C_p}(S^0) \oplus \uH^{\alpha}_{C_p}(S^0)$$ 
is an isomorphism onto the second factor and $0$ onto the first factor.\\
\end{prop}
\begin{proof}
The Mackey functor $\Phi_p^\ast \uM$ is given by $S \mapsto \uM(S\times \GG/C_p)$. In view of the natural isomorphism $\tilde{H}^\alpha_{\GG}(X\wedge {\GG/C_p}_+; \uA) \cong \tilde{H}^{\res_{C_p}\alpha}_{C_p} (X; \uA)$, we have $\Phi_p^\ast(\uH^\alpha_{\GG}(X; \uA)) \cong \uH^\alpha_{C_p}(X; \uA)$ with the induced action of $C_p$ on $X$ by restriction. 

For $X=S(\xi^p)_+$ we have the restriction of the action to $C_p$ is the space $S^1_+$ with trivial action. The cofibre sequence \eqref{cof-p} restricted to $C_p$ becomes the cofibre sequence 
$$S^1_+\stackrel{\pi}{\to} S^0 \to S^2. $$
 As $\pi$ is a retraction, this induces an isomorphism $\uH^\alpha_{C_p}(S(\xi^p)_+)\cong \uH^\alpha_{C_p}(S^0) \oplus \uH_{C_p}^{\alpha-1}(S^0)$. It follows that the maps $\Phi_p^\ast\pi_{\xi^p}$ and $\Phi_p^\ast \delta_{\xi^p}$  are given by isomorphisms on the first and second factor respectively. 

It remains to identify the splitting above with that of Proposition \ref{G/Cp}. The restriction of the cofibre sequence \eqref{cofp} gives the sequence 
$$\bigvee_q S^0 \stackrel{1-g}{\to} \bigvee_q S^0 \to S^1_+$$
of $C_p$-spectra. The composite 
$$\bigvee_q S^0 \to S^1_+ \stackrel{\pi}{\to} S^0$$
 induces the diagonal map on cohomology which in turn is the kernel of $(1-g)^\ast$ as computed in Proposition \ref{G/Cp}. Analogously the composite $S^0 \to S^1_+ \to \vee_q S^1$ induces the fold map on cohomology which is an isomorphism on the cokernel of $(1-g)^\ast$. Since the inclusion $S^0\to S^1_+$ and $\pi : S^1_+ \to S^0$ are used to get the splitting above, these isomorphisms identify the splitting with that of Proposition \ref{G/Cp}.
\end{proof}

This leads to the following result by applying Corollary \ref{p-sphere}. 
\begin{prop}\label{shex-p} 
a) If $(|\alpha|\geq 1,|\alpha^{C_p}|\geq 1~ \mbox{or}~ \leq -3)$ or $(|\alpha|\leq -3,|\alpha^{C_p}|\leq -3~ or \geq 2)$,
$$\uH^{\alpha}_{\GG}(S^0)\cong \uH^{\alpha+\xi^p}_{\GG}(S^0).$$
b) If $(|\alpha|\geq 2,|\alpha^{C_p}|= 0)$ or $(|\alpha|\leq -3,|\alpha^{C_p}|= 0)$ , there is a short exact sequence of Mackey functors 
$$0\to \CC_p\langle \Z \rangle \to \uH^\alpha_{\GG}(S^0) \to \uH^{\alpha + \xi^p}_{\GG}(S^0)\to 0$$
As a consequence, if $|\alpha|\geq 2$ or $\leq -3$, $|\alpha^{C_p}|=-1$ and $\uH^\alpha_{\GG}(S^0)=0$, then $\uH^{\alpha+\xi^p}_{\GG}(S^0)=0$. \\
c) If $(|\alpha|\leq -1,|\alpha^{C_p}|= 1)$ , 
$$\uH^\alpha_{\GG}(S^0) \oplus \KK_p\bZp \cong \uH^{\alpha + \xi^p}_{\GG}(S^0)$$
d) If $|\alpha|>0$ and $|\alpha^{C_p}|=-2$ there is an Abelian group $A$ such that $\tilde{H}^{\alpha+\xi^p}_{\GG}(S^0)= A\oplus \Z$ and $\tilde{H}^{\alpha}_{\GG}(S^0) = A\oplus \Z/p$. \\
e) If $(|\alpha|\leq -4, |\alpha^{C_p}|=-2)$,  $\tilde{H}^{\alpha+\xi^p}_{\GG}(S^0) \cong \tilde{H}^{\alpha}_{\GG}(S^0) \oplus \Z$. \\
The results also hold when $\xi^p$ is replaced by $\xi^k$ where $k$ is divisible by $p$ but not by $q$.
\end{prop}

\begin{proof}
We use the long exact sequence \eqref{ex-p} indexed at $\alpha + \xi^p$, and Corollary \ref{p-sphere}. First we note that the fact that if $(|\alpha|\geq 2, |\alpha^{C_p}|\geq 2)$ or $(|\alpha|\leq -1, |\alpha^{C_p}|\leq -1)$, $\uH^\alpha_{\GG}(S(\xi^p)_+)=0$ implies the cases $(|\alpha|\geq 1,|\alpha^{C_p}|\geq 1)$ and $(|\alpha|\leq -3,|\alpha^{C_p}|\leq -3)$ of a). 

For the remainder of the cases of a), we obtain the exact sequence from \eqref{shex-p} and Corollary \ref{p-sphere}
$$\cdots \to  \CC_p\bZp \to \uH^{\alpha}_{\GG}(S^0) \to \uH^{\alpha+\xi^p}_{\GG}(S^0) \to \KK_p\bZp \to  \cdots$$
Proposition \ref{sus-p} implies that the leftmost and rightmost arrows are isomorphisms at $\GG/C_p$. So we deduce from Proposition \ref{cpzp} that there is a Mackey functor $\uN$ such that 
$$\uH^{\alpha+\xi^p}_{\GG}(S^0)\cong \uN \oplus \KK_p\bZp,$$
 and  
$$\uH^{\alpha}_{\GG}(S^0) \cong \uN \oplus \CC_p\bZp$$ 
which are isomorphic as $\KK_p\bZp$ and $\CC_p\bZp$ are. Hence a) follows. 

Next we prove b). Under the given assumptions on $\alpha$ we have $\uH^{\alpha+\xi^p}_{\GG}(S(\xi^p)_+)=0$ and $\uH^{\alpha+\xi^p-1}_{\GG}(S(\xi^p)_+)= \CC_p\bZ$. Thus the exact sequence \eqref{ex-p} in this range looks like 
$$\cdots \to \CC_p\bZ\to \uH^{\alpha}_{\GG}(S^0) \to \uH^{\alpha+\xi^p}_{\GG}(S^0) \to 0$$
where the left map is an isomorphism at $\GG/C_p$ by Proposition \ref{sus-p}. The restriction map $ \CC_p\bZ(\GG/\GG)\to  \CC_p\bZ(\GG/C_p)$ is injective, so that the map of Mackey functors  $\CC_p\bZ\to \uH^{\alpha}_{\GG}(S^0)$ has trivial Kernel. Thus b) follows. 

For c), it suffices to prove for $|\alpha|\ll 0$ by Remark \ref{cor-e}. If  $|\alpha|\ll 0$ odd and $|\alpha^{C_p}|=1$, which implies $|\alpha+\xi^p| \ll 0$ and $|(\alpha+\xi^p)^{C_p}|=3$ so that Corollary \ref{p-sphere} implies  $\uH^{\alpha+\xi^p}_{\GG}(S(\xi^p)_+) \cong \KK_p\bZp$. Therefore \eqref{ex-p} looks like 
$$0\to \uH^{\alpha}_{\GG}(S^0) \to \uH^{\alpha+\xi^p}_{\GG}(S^0) \to \KK_p\bZp \to  \cdots $$
where the right map is a split surjection at level $\GG/C_p$ by Proposition \ref{sus-p}. We now conclude c) from Proposition \ref{cpzp}.  In the case of e), the exact sequence \eqref{ex-p} reduces to 
$$0\to \uH^{\alpha}_{\GG}(S^0) \to \uH^{\alpha+\xi^p}_{\GG}(S^0) \to \KK_p\bZ \to \cdots $$
The right map is an isomorphism at $\GG/C_p$ which means that evaluating at $\GG/\GG$ the image is a non-trivial subgroup of $\Z$. Thus we have a short exact sequence 
$$0\to \tilde{H}^{\alpha}_{\GG}(S^0) \to \tilde{H}^{\alpha+\xi^p}_{\GG}(S^0) \to \Z \to 0 $$
and e) follows as $\Z$ is free. Lastly for d), \eqref{ex-p} looks like 
$$\cdots \to  \CC_p\bZp \to \uH^{\alpha}_{\GG}(S^0) \to \uH^{\alpha+\xi^p}_{\GG}(S^0) \to \KK_p\bZ \to  \cdots $$
The leftmost map satisfies is a split injection at $\GG/C_p$ by Proposition \ref{sus-p} and so is a split injection of Mackey functors by Proposition \ref{cpzp}. The rightmost map as in e) is a split surjection onto $\Z$ at $\GG/\GG$. Hence d) follows and the proof is complete. 
\end{proof}

\begin{remark}\label{cor-p}
Proposition \ref{shex-p} a) implies that  in computations of $\uH^\alpha_{\GG}(S^0)$ when $|\alpha| \neq 0$, we have three kinds of cases $|\alpha^{C_p}| \gg 0$, $|\alpha^{C_p}| \ll 0$ and $|\alpha^{C_p}|$ close to $0$. More explicitly, suppose $|\alpha^{C_q}|$ and $|\alpha^{\GG}|$ are fixed and $|\alpha|\neq 0$. As in remark \ref{cor-e}, it suffices to consider the following cases \\
1) For $|\alpha|>0$ even, three subcases $|\alpha^{C_p} |>0$, $|\alpha^{C_p} |< 0$ and $|\alpha^{C_p} |=0$. \\
2) For $|\alpha|>0$ odd, two subcases $|\alpha^{C_p} |>0$ and $|\alpha^{C_p} |< 0$. \\
3) For $|\alpha|<0$ even, three subcases $|\alpha^{C_p} |>0$, $|\alpha^{C_p} |< 0$ and $|\alpha^{C_p} |=0$. \\
4) For $|\alpha|<0$ odd, three subcases $|\alpha^{C_p} |<0$, $|\alpha^{C_p} |=1$ and $|\alpha^{C_p} |>1$. \\
Further, Proposition \ref{shex-p} c), provides a relation between the last two cases of 4). We note that a similar list of cases might be written down for $|\alpha|$ and $|\alpha^{C_q}|$ by applying Proposition \ref{shex-p} with $q$ instead of $p$. Thus, we essentially have a grid of cases to compute which we accomplish in the rest of the paper. 
\end{remark}

\mbox{ } \\

We end the section with some explicit computations. These are meant to be a guideline as to how the results above lead us to computations. We compute $\uH^\alpha_{\GG}(S^0)$ for $\alpha=\xi^p-n, \xi^p+\xi^q - n, n-\xi^p, n - \xi^p-\xi^q$.

We start with the exact sequence \eqref{ex-p} where $\alpha$ is of the form $\xi^p - n$. The first place from the left where $\uH^{\xi^p-n}_{\GG} (S(\xi^p)_+)$ is non-zero is for $n=2$. We have the sequence
$$\uH^{-2}_{\GG}(S^0) \to \uH^{\xi^p-2}_{\GG}(S^0) \to \uH^{\xi^p-2}_{\GG}(S(\xi^p)_+)\to \uH^{-1}_{\GG}(S^0)$$
which implies the isomorphism $\uH^{\xi^p -2}_{\GG}(S^0)\cong \uH^{\xi^p -2}_{\GG} (S(\xi^p)_+)\cong \KK_p\uA_p$ from Proposition \ref{G/Cp}. Moving one step further we obtain  
$$0\to  \uH^{\xi^p -1}_{\GG}(S^0) \to  \uH^{\xi^p -1}_{\GG} ({S(\xi^p)}_+) \stackrel{\psi}{\to} \uH^{0}_{\GG}(S^0) \to \uH^{\xi^p}_{\GG}(S^0) \to 0.$$
We know that $\uH^{0}_{\GG}(S^0)\cong \uA$ and  $\uH^{\xi^p -1}_{\GG} ({S(\xi^p)}_+) \cong \CC_p \uA_p$. From Proposition \ref{Cpqmap}, we have that the map $\psi$ is classified by $\psi(\GG/C_p)$ and $\psi(\GG/e)$. However, Propositions \ref{G/Cp} and \ref{Lewis orbits} imply that $\Phi_p^\ast\uH^{\xi^p -1}_{\GG}(S^0) =0$, and thus, the Mackey functor $\uH^{\xi^p -1}_{\GG}(S^0)$ is zero at $\GG/C_p$ and $\GG/e$. Therefore the maps $\psi(\GG/e)$ and $\psi(\GG/C_p)$ are isomorphisms and this allows us to completely compute the map $\psi$. It follows that  $\psi$ is injective and hence, $\uH^{\xi^p-1}_{\GG}(S^0)=0$ and 
$\uH^{\xi^p}_{\GG}(S^0)\cong \QQ_p\uA_p$, the cokernel of $\psi$. The terms further right in the sequence are all $0$. Therefore we obtain
\begin{equation} \label{xip+}
\uH^{\xi^p-n}_{\GG}(S^0) \cong \begin{cases} 0 &\mbox{if}~n\neq 0,2 \\
                                                                                    \KK_p\uA_p &\mbox{if}~n=2\\
                                                                                    \QQ_p\uA_p &\mbox{if}~n=0. \end{cases}   
\end{equation} 
Analogously we deduce the case 
\begin{equation} \label{xip-}
\uH^{n-\xi^p}_{\GG}(S^0) \cong \begin{cases} 0 &\mbox{if}~n\neq 0,2 \\
                                                                                    \CC_p\uA_p &\mbox{if}~n=2\\
                                                                                    \QQ_p\uA_p &\mbox{if}~n=0. \end{cases}   
\end{equation} 
Now we add a copy of $\xi^q$ to the groups in \eqref{xip+} using the exact sequence \eqref{ex-p} for $\xi^q$. The first non-zero part computes $\uH^{\xi^p+\xi^q-4}_{\GG}(S^0)\cong \KK_qR_q\cong R_{pq}$. The next part is 
$$ 0 \to \uH^{\xi^p+\xi^q-3}_{\GG}(S^0) \to \uH^{\xi^p+\xi^q-3}_{\GG}(S(\xi^q)_+) \to \cdots $$
The Mackey functor $ \uH^{\xi^p+\xi^q-3}_{\GG}(S(\xi^q)_+) \cong \CC_qR_q$ by Proposition \ref{p-sphere}. We note that the map from $\CC_qR_q$ to the next term in the sequence is injective. Therefore, $ \uH^{\xi^p+\xi^q-3}_{\GG}(S^0) = 0$. The next terms in the sequence are 
$$ 0 \to \CC_qR_q \to \uH^{\xi^p-2}_{\GG}(S^0) \to \uH^{\xi^p+\xi^q-2}_{\GG}(S^0)\to \uH^{\xi^p+\xi^q-2}_{\GG}(S(\xi^q)_+)\to 0$$
After substituting the values from above this looks like
$$0 \to \CC_qR_q \to \KK_p\uA_p \to \uH^{\xi^p+\xi^q-2}_{\GG}(S^0)\to \KK_q\bZ \to 0$$
The map $\CC_qR_q\to \KK_p\uA_p$ is determined by the map at $\GG/C_q$ where it is an isomorphism. We may compute the cokernel as $\KK_p\bZ$. So we obtain the short exact sequence 
$$0 \to  \KK_p\bZ \to \uH^{\xi^p+\xi^q-2}_{\GG}(S^0)\to \KK_q\bZ \to 0$$
and by symmetry also the exact sequence the other way around. It follows that the exact sequence is split and $\uH^{\xi^p+\xi^q-2}_{\GG}(S^0) \cong \KK_p\bZ\oplus \KK_q\bZ$. The next terms in the exact sequence are 
$$ 0 \to  \uH^{\xi^p+\xi^q-1}_{\GG}(S^0)\to \uH^{\xi^p+\xi^q-1}_{\GG}(S(\xi^q)_+)\to \cdots$$
The Mackey functor $\uH^{\xi^p+\xi^q-1}_{\GG}(S(\xi^q)_+) \cong \CC_q\bZ$ an
d by Proposition \ref{shex-p} b) the map from this onto the next term is injective. Thus, $\uH^{\xi^p+\xi^q-1}_{\GG}(S^0)=0$. The next terms are 
$$ 0 \to \CC_q\bZ \to \uH^{\xi^p}_{\GG}(S^0) \to \uH^{\xi^p+\xi^q}_{\GG}(S^0)\to \uH^{\xi^p+\xi^q}_{\GG}(S(\xi^q)_+)$$
with the last term being $0$ by Proposition \ref{p-sphere}. Therefore, $\uH^{\xi^p+\xi^q}_{\GG}(S^0)$ is the cokernel of the map $\CC_q\bZ \to \QQ_p\uA_p$ which is an isomorphism at level $\GG/C_q$. It follows that the cokernel is $\bbZ$. Thus we obtain 
\begin{equation} \label{xipq+}
\uH^{\xi^p+\xi^q-n}_{\GG}(S^0) \cong \begin{cases} 0 &\mbox{if}~n\neq 0,2,4 \\
                                                                                     R_{pq} &\mbox{if}~n=4\\
                                                                                    \KK_p\bZ \oplus \KK_q\bZ &\mbox{if}~n=2\\
                                                                                    \bbZ &\mbox{if}~n=0. \end{cases}   
\end{equation} 
Analogously we start with the computation \eqref{xip-} and put the values in the exact sequence \eqref{ex-p} to obtain 
\begin{equation} \label{xipq-}
\uH^{n-\xi^p-\xi^q}_{\GG}(S^0) \cong \begin{cases} 0 &\mbox{if}~n\neq 0,2,4 \\
                                                                                     L_{pq} &\mbox{if}~n=4\\
                                                                                    \CC_p\bZ \oplus \CC_q\bZ &\mbox{if}~n=2\\
                                                                                    \bbZ &\mbox{if}~n=0. \end{cases}   
\end{equation} 

\section{Initial computations and the independence theorem}\label{comp1}
In this section, we start the computations of $\uH^\alpha_{\GG}(S^0)$ in a systematic way. One of the main theorems of the section states that the groups in the Mackey functor are a function of the fixed point dimensions. This is proved at the end of the section following certain initial computations. 

Our approach to the computations follows certain clues from Remarks \ref{cor-e} and \ref{cor-p} in the last section. We get the idea that the computations are perhaps going to be easier if we let the fixed points to be non-zero. The computations for $C_p$-cohomology (Proposition \ref{Lewis orbits}) suggest that these are likely to just depend on the fixed point dimensions. In the Proposition below, we allow $|\alpha^H|$ to be non-zero and of the same sign. In addition when $|\alpha^{\GG}|$ is also of the same sign the computation is easy, the other cases are more involved.  

\begin{prop} \label{sign}
Let $\alpha \in RO(\GG)$.\\
a) Suppose $|\alpha^H|>0$ for all $H\neq \GG$. Then
$$\uH^{\alpha}_{\GG}(S^0) = 
\begin{cases} 
0 & \mbox{if}~|\alpha^{\GG}|>0,~\mbox{or}~|\alpha^{\GG}|<0~\mbox{odd}\\
\bbZ & \mbox{if}~|\alpha^{\GG}|=0 \\
\bbZpq &  \mbox{if}~|\alpha^{\GG}|<0 ~ \mbox{even} .
\end{cases}
$$
b) Suppose $|\alpha^H|<0$ for all $H\neq \GG$. Then
$$\uH^{\alpha}_{\GG}(S^0) = 
\begin{cases} 
0 & \mbox{if}~|\alpha^{\GG}|\leq 1\, \mbox{odd},~\mbox{or}~|\alpha^{\GG}|\neq 0\,\mbox{even}\\
\bbZ & \mbox{if}~|\alpha^{\GG}|=0 \\
\bbZpq &  \mbox{if}~|\alpha^{\GG}|\geq 3 ~ \mbox{odd} .
\end{cases}
$$
\end{prop}

\begin{proof}
We first observe that the Mackey functors $\uH^\alpha_{\GG}(S^0)$ in a) or b) depend only on the value of $|\alpha^{\GG}|$. If $|\alpha_1^{\GG}| = |\alpha_2^{\GG}|$, then we must have that for some integers $j_i$ and $l_k$ between $1$ and $pq-1$, 
$$\alpha_1 + \sum_i \xi^{j_i} = \alpha_2 + \sum_k \xi^{l_k}$$

 Start with the case a). We note under the given conditions on $\alpha$, $\uH^\alpha_{\GG}(S^0) \cong \uH^{\alpha+\xi^j}_{\GG}(S^0)$ for $j$ relatively prime to $pq$ by Proposition \ref{shex-e}, and for $j$ divisible by exactly one of $p$ or $q$  by Proposition \ref{shex-p}. Thus, we deduce 
$$\uH^{\alpha_1}_{\GG}(S^0) \cong \uH^{\alpha_1 + \sum_i \xi^{j_i}}_{\GG}(S^0) = \uH^{\alpha_2+ \sum_k \xi^{l_k}}_{\GG}(S^0) \cong \uH^{\alpha_2}_{\GG}(S^0)$$   
For the case b), we have $\uH^\alpha_{\GG}(S^0) \cong \uH^{\alpha-\xi^j}_{\GG}(S^0)$ for all $j$ between $1$ and $pq-1$ by Propositions \ref{shex-e} and \ref{shex-p}, so, a similar argument applies. As a consequence, we may compute by making specific choices of $\alpha$. We directly deduce the case $|\alpha^{\GG} |> 0$ of a) and  $|\alpha^{\GG}|<0$ of b) as $\uH^n_{\GG}(S^0)=0$ when $n\neq 0$. Further, the computations \eqref{xipq+} and \eqref{xipq-} allow us to deduce the cases $|\alpha^{\GG}|=0$ of a) and b) and $|\alpha^{\GG}|=1$ of b). 

At this point it might seem like a few more explicit computations may be able to prove the Proposition easily. We avoid this approach because the explicit computations are quite involved, and also because the other path traverses through computations of $\uH^\alpha_{\GG}(S^0)$ as a function of the fixed points. 

Our approach  for the rest of the cases for a) involves starting with $|\alpha^H|<0$ for all $H$, where we know the value to be zero, then flip the signs of $|\alpha^H|$ by adding copies of $\xi$, $\xi^p$ and $\xi^q$ and make calculations using the exact sequences \eqref{ex-e} and \eqref{ex-p} along the way. Start with an odd element $\alpha$ with negative dimensional fixed points and apply Proposition \ref{shex-p} b) and Proposition \ref{shex-e} b) to deduce that if  $|\alpha^H|\leq 1$ odd for all $H\neq \GG$ and $|\alpha^{\GG}|<0$, $\uH^\alpha_{\GG}(S^0)=0$. This proves the first statement of a). Additionally, we may apply c) of Proposition \ref{shex-p} to deduce \\
1)  If $\alpha$ is odd and $|\alpha|<0, |\alpha^{C_p}|>1, |\alpha^{C_q}|\leq 1, |\alpha^{\GG}|<0$, then, $\uH^\alpha_{\GG}(S^0)\cong \KK_p\bZp$. \\ 
2)  If $\alpha$ is odd and $|\alpha|<0$, $|\alpha^{C_p}|>1$, $|\alpha^{C_q}|>1$, $|\alpha^{\GG}|<0$, then, $\uH^\alpha_{\GG}(S^0)\cong \KK_p\bZp \oplus \KK_q\bZq$. \\

Next, we consider $|\alpha^{\GG}|<0$ even. We take $|\alpha|\leq -4$, $|\alpha^{C_p}|=-2$, $|\alpha^{C_q}|=-2$ so that $\uH^{\alpha+\xi^p}_{\GG}(S(\xi^p)_+) \cong \KK_p\bZ$ and the exact sequence \eqref{ex-p} indexed at $\alpha+\xi^p$ reduces to  
$$0 \to \uH^\alpha_{\GG}(S^0) \to  \uH^{\alpha+\xi^p}_{\GG}(S^0) \to \KK_p\bZ \to \cdots$$
Observe that $ \uH^\alpha_{\GG}(S^0)=0$ and  $\uH^{\alpha+1}_{\GG}(S^0)=0$ as all the fixed points have dimension $<0$. It follows that for even $\alpha$ with $|\alpha|<0, |\alpha^{C_p}|=0, |\alpha^{C_q}|<0, |\alpha^{\GG}|<0$, $ \uH^\alpha_{\GG}(S^0)=\KK_p\bZ$. Now we consider $|\alpha| \leq -4$ in this expression and increase $\alpha$ by another copy of $\xi^p$. It follows from b) of Proposition \ref{shex-p} that  $ \uH^{\alpha+\xi^p}_{\GG}(S^0)$ is the cokernel of a map 
$$\CC_p\bZ \to \KK_p\bZ$$
which is an isomorphism at $\GG/C_p$ by Proposition \ref{sus-p}. It follows that the cokernel is $\bbZq$. Therefore for even $\alpha$ with $|\alpha|<0, |\alpha^{C_p}|>0, |\alpha^{C_q}|<0, |\alpha^{\GG}|<0$, $ \uH^\alpha_{\GG}(S^0)=\bbZq$. Next we increase $\alpha$ by $\xi^q$ starting with $|\alpha^{C_q}|=-2$, so that an analogous argument  implies that for $|\alpha|<0, |\alpha^{C_p}|>0, |\alpha^{C_q}|=0, |\alpha^{\GG}|<0$, we have the exact sequence (\eqref{ex-p} for the prime $q$)
$$0\to \bbZq \to \uH^\alpha_{\GG}(S^0) \to \KK_q\bZ\to 0 $$
which may be considered as an element of $\Ext^1_{\BB_{\GG}}(\KK_q\bZ,\bbZq)$. By Proposition \ref{Cpqext}, this group is $0$. Hence, the sequence splits and we have that  $ \uH^\alpha_{\GG}(S^0)\cong \KK_q\bZ \oplus \bbZq$. Increasing $\alpha$ by a further $\xi^q$, analogously as above we realize that for  even $\alpha$ with $|\alpha| <0, |\alpha^{C_p}|>0, |\alpha^{C_q}|>0, |\alpha^{\GG}|<0$, $ \uH^\alpha_{\GG}(S^0)$ is the cokernel of a map 
$$ \CC_q\bZ \to \KK_q\bZ \oplus \bbZq$$ 
The map $\CC_q\bZ \to \bbZq$ is $0$ by Proposition \ref{Cpqmap}, and thus, the cokernel may be computed as in the $p$-case to be $\bbZp \oplus \bbZq \cong \bbZpq$. Hence for even $\alpha$ with $|\alpha|<0, |\alpha^{C_p}|>0, |\alpha^{C_q}|>0, |\alpha^{\GG}|<0$, $ \uH^\alpha_{\GG}(S^0)\cong \bbZpq$. Now start with $|\alpha|=-2$ in this case, and add a copy of $\xi$ to obtain the exact sequence \eqref{ex-e} for $|\alpha| =0, |\alpha^{C_p}|>0 , |\alpha^{C_q}| > 0, |\alpha^{\GG}|<0$, 
$$0\to \bbZpq \to \uH^\alpha_{\GG} \to R_{pq} \to \KK_p\bZp\oplus \KK_q\bZq \to 0 $$
The last two terms in the exact sequence are computed from 2) above and the first statement of a) proved above. It is routine to  compute the kernel of $R_{pq} \to \KK_p\bZp\oplus \KK_q\bZq$ which turns out to be $L_{pq}$. Therefore the above sequence turns out to be 
$$0\to \bbZpq \to \uH^\alpha_{\GG} \to L_{pq}\to 0 $$ 
which may be considered as an element of $\Ext^1_{\BB_{\GG}}(L_{pq},\bbZpq)$. By Proposition \ref{Cpqext}, the sequence splits and we have that $\uH^\alpha_{\GG}(S^0) \cong L_{pq}\oplus \bbZpq$. Adding another copy of $\xi$ to such an $\alpha$ reduces to the case b) of Proposition \ref{shex-e} and hence, we obtain the last statement of a). 

Now we turn to b). Here we start with $\alpha$ with all fixed points positive and reverse the signs of $|\alpha|$, $|\alpha^{C_p}|$ and $|\alpha^{C_q}|$ by subtracting copies of $\xi$, $\xi^p$ and $\xi^q$ using the exact sequences \eqref{ex-e} and \eqref{ex-p}. We also note that Propositions \ref{Lewis orbits} and \ref{G/Cp} imply that $\Phi_p^\ast\uH^\alpha_{\GG}(S^0)=0$ and $\Phi_q^\ast\uH^\alpha_{\GG}(S^0)=0$ in the remaining cases of b), so that, it suffices to compute the group $\tilde{H}^\alpha_{\GG}(S^0)$. Consider the case where $|\alpha^{\GG}|>0$ even. First note that Proposition \ref{shex-e} b) implies \\
3) If $\alpha$ is even with  $|\alpha| =0, |\alpha^{C_p}|>0, |\alpha^{C_q}|>0, |\alpha^{\GG}|> 0$, then, $ \uH^\alpha_{\GG}(S^0)\cong L_{pq}$.\\
Now start with $|\alpha| = -2, |\alpha^{C_p}|=2, |\alpha^{C_q}|=2, |\alpha^{\GG}|>0$ so that we have, $\uH^{\alpha+\xi}_{\GG}(S^0)\cong L_{pq}$ and then Proposition \ref{shex-e} c) implies that $\tilde{H}^\alpha_{\GG}(S^0)=0$. Therefore for  even $\alpha$ with $|\alpha|<0, |\alpha^{C_p}|>0, |\alpha^{C_q}|>0, |\alpha^{\GG}|> 0$, we have $ \uH^\alpha_{\GG}(S^0)\cong 0$. Now Proposition \ref{shex-p} b) implies  for  $|\alpha| <0, |\alpha^{C_p}|=0, |\alpha^{C_q}|>0, |\alpha^{\GG}|> 0$, $ \uH^\alpha_{\GG}(S^0)\cong \CC_p\bZ$. Now we apply Proposition \ref{shex-p} e) to deduce that for  even $\alpha$ with $|\alpha|<0, |\alpha^{C_p}|<0, |\alpha^{C_q}|>0, |\alpha^{\GG}|> 0$, $ \uH^\alpha_{\GG}(S^0)\cong 0$. Analogously we also obtain the following results\\
4) If $\alpha$ is even with  $|\alpha| <0, |\alpha^{C_p}|<0, |\alpha^{C_q}|=0, |\alpha^{\GG}|> 0$, then, $ \uH^\alpha_{\GG}(S^0)\cong \CC_q\bZ$.\\
5) If $\alpha$ is even with  $|\alpha| <0, |\alpha^{C_p}|<0, |\alpha^{C_q}|<0, |\alpha^{\GG}|> 0$, then, $ \uH^\alpha_{\GG}(S^0)\cong 0$.\\
This implies the first statement of b). In fact, an analogous argument also gives the short exact sequence 
$$0\to \CC_p\bZ \to \uH^\alpha_{\GG}(S^0)\to \CC_q\bZ\to 0$$
 whenever $|\alpha|<0,|\alpha^{C_p}|=0, |\alpha^{C_q}|=0, |\alpha^{\GG}|> 0$, and by the symmetry of $p$ and $q$, also the sequence the other way around. Therefore, the sequence splits and we obtain \\
6) If $\alpha$ is even with  $|\alpha| <0, |\alpha^{C_p}|=0, |\alpha^{C_q}|=0, |\alpha^{\GG}|> 0$, then, $ \uH^\alpha_{\GG}(S^0)\cong \CC_p\bZ\oplus \CC_q\bZ$.

Finally, it remains to prove the case in b) where $\alpha$ is odd with $|\alpha^{\GG}|\geq 3$. Consider an odd $\alpha$ with $|\alpha|<0, |\alpha^{C_p}|>1, |\alpha^{C_q}|>1, |\alpha^{\GG}|>1$. Then from \eqref{ex-e} we have an exact sequence 
$$ \cdots \to \uH^{\alpha + \xi -1}_{\GG}(S^0) \to R_{pq} \to \uH^\alpha_{\GG}(S^0) \to 0.$$
We know  that $\uH^{\alpha+\xi -1}_{\GG}(S^0)\cong L_{pq}$ from 3) above, and from Proposition \ref{sus-e}, that the map $L_{pq} \to R_{pq}$ is an isomorphism at $\GG/e$. It follows that the map $L_{pq} \to R_{pq}$ is injective and its cokernel $\uH^\alpha_{\GG}(S^0) \cong \KK_p\bZp \oplus \KK_q \bZq$. Now Proposition \ref{shex-p} c) implies that \\
7) If $\alpha$ is odd with  $|\alpha| <0, |\alpha^{C_p}|=1, |\alpha^{C_q}|>1, |\alpha^{\GG}|>1$, $ \uH^\alpha_{\GG}(S^0)\cong \KK_q\bZq$. \\
8) If $\alpha$ is odd with  $|\alpha| <0, |\alpha^{C_p}|=1, |\alpha^{C_q}|=1, |\alpha^{\GG}|>1$, $ \uH^\alpha_{\GG}(S^0)=0$. \\
Now we start with $\alpha$ as in 8) and subtract copies of $\xi^p$ and $\xi^q$. Subtracting $\xi^p$ we obtain that for  $|\alpha|<0, |\alpha^{C_p}|=-1, |\alpha^{C_q}|=1, |\alpha^{\GG}|>1$,  the sequence \eqref{ex-p} indexed at $\alpha+\xi^p$ looks like 
$$  \cdots \to \CC_p\bZ \oplus \CC_q\bZ \to \KK_p\bZ \to \uH^\alpha_{\GG}(S^0) \to 0 $$
after substituting the values from  6) above and Proposition \ref{p-sphere}. The map from $\CC_q\bZ$ to $\KK_p\bZ$ is trivial by Proposition \ref{Cpqmap} and the map from $\CC_p\bZ$ to $\KK_p\bZ$ has cokernel $\bbZq$. Therefore, we have \\
9) If $\alpha$ is odd with $|\alpha|<0, |\alpha^{C_p}|<0, |\alpha^{C_q}|=1, |\alpha^{\GG}|>1 $, then, $ \uH^\alpha_{\GG}(S^0)\cong \bbZq$.\\
Note that this implies by Proposition \ref{shex-p} c) \\
10) If $\alpha$ is odd with $|\alpha|<0, |\alpha^{C_p}|<0, |\alpha^{C_q}|>1, |\alpha^{\GG}|>1 $, then, $ \uH^\alpha_{\GG}(S^0)\cong \KK_q\bZq \oplus \bbZq$.\\
Remove a copy of $\xi^q$ from an $\alpha$ as in 9), and apply the exact sequence \eqref{ex-p} indexed at $\alpha+\xi^q$.  Thus, we deduce the following exact sequence for $|\alpha|<0, |\alpha^{C_p}|<0, |\alpha^{C_q}|=-1, |\alpha^{\GG}|\geq 3$,
$$  \cdots \to \CC_q\bZ \to \KK_q\bZ \to \uH^\alpha_{\GG}(S^0) \to \bbZq \to \CC_q\bZ  \to \cdots$$
The last map must be zero as all the groups in $\CC_q\bZ$ are free and the ones in $\bbZq$ are finite. We compute the cokernel of $\CC_q\bZ \to \KK_q\bZ$ as before to be $\bbZp$ to deduce $\uH^\alpha_{\GG}(S^0) \cong \bbZpq$.  This completes the proof. 
\end{proof}

As stated in the proof of the Proposition above, we intentionally followed a route that traversed through a number of useful computations. We now note them down in the three Propositions below. These results are arranged according to $|\alpha|<0$ even, $|\alpha|<0$ odd, and $|\alpha|=0$. 

\begin{prop} \label{calc1odd} Let $\alpha$ be an odd element with  $|\alpha|<0$. \\
 a) Suppose $|\alpha^{C_p}|=1$. Then,
$$\uH^{\alpha}_{\GG}(S^0) = 
\begin{cases} 
0 & \mbox{if}~|\alpha^{C_q}|=1 ~\mbox{and}~|\alpha^{\GG}|\neq 1\\
0 & \mbox{if}~|\alpha^{C_q}|<0~\mbox{and}~|\alpha^{\GG}|<0 \\
\bbZp & \mbox{if}~|\alpha^{C_q}|<0 ~\mbox{and}~|\alpha^{\GG}|>1.
\end{cases}$$
b)Suppose that $ |\alpha^{C_p}|>1$. Then,  
$$\uH^{\alpha}_{\GG}(S^0) = 
\begin{cases} 
\KK_p\bZp & \mbox{if}~|\alpha^{C_q}|=1 ~\mbox{and}~|\alpha^{\GG}|\neq 1\\
\KK_p\bZp \oplus \KK_q\bZq & \mbox{if}~|\alpha^{C_q}|>1~\mbox{and}~|\alpha^{\GG}|\neq 1 \\
\KK_p\bZp & \mbox{if}~|\alpha^{C_q}|<0 ~\mbox{and}~|\alpha^{\GG}|<0\\
\KK_p\bZp \oplus \bbZp & \mbox{if}~|\alpha^{C_q}|<0 ~\mbox{and}~|\alpha^{\GG}|>1.
\end{cases}
$$
\end{prop}

\begin{prop}\label{calc1even}
Let $\alpha$ be an even element with  $|\alpha|<0$. \\
a) Suppose $|\alpha^{C_p}|=0$. Then, 
$$\uH^{\alpha}_{\GG}(S^0) = 
\begin{cases} 
\KK_p\bZ & \mbox{if}~|\alpha^{C_q}|<0 ~\mbox{and}~|\alpha^{\GG}|<0\\
\CC_p\bZ & \mbox{if}~|\alpha^{C_q}|\neq 0~\mbox{and}~|\alpha^{\GG}|>0 \\
\CC_p\bZ\oplus \CC_q\bZ &\mbox{if}~ |\alpha^{C_q}|=0 ~\mbox{and}~|\alpha^{\GG}|>0 \\ 
\KK_p\bZ\oplus \bbZp & \mbox{if}~|\alpha^{C_q}|>0 ~\mbox{and}~|\alpha^{\GG}|<0. 
\end{cases}$$
b) Suppose $|\alpha^{C_p}|>0$. Then, 
$$\uH^{\alpha}_{\GG}(S^0) = 
\begin{cases} 
0 & \mbox{if}~|\alpha^{C_q}|\neq 0 ~\mbox{and}~|\alpha^{\GG}|>0\\
\bbZq & \mbox{if}~|\alpha^{C_q}|<0~\mbox{and}~|\alpha^{\GG}|<0 \\
\bbZpq & \mbox{if}~|\alpha^{C_q}|>0 ~\mbox{and}~|\alpha^{\GG}|<0 . 
\end{cases}$$
\end{prop}

\begin{prop}\label{calc1b}
a) Suppose $|\alpha| =0, |\alpha^{C_p}|>0, |\alpha^{C_q}|>0$. Then, 
$$\uH^{\alpha}_{\GG}(S^0) = 
\begin{cases} 
L_{pq} & \mbox{if}~|\alpha^{\GG}|>0\\
L_{pq} \oplus \bbZpq & \mbox{if}~ |\alpha^{\GG}|< 0.
\end{cases}$$
b) If $\alpha$ is an odd element with all fixed points having dimension $\leq 1$, then $\uH^\alpha_{\GG}(S^0)=0$. 
\end{prop}

The next step in our computations is to start with the results of Proposition \ref{sign} and flip the sign of $|\alpha|$ using multiples of $\xi$. Some of these have already been computed and are stated in the Propositions \ref{calc1odd}, \ref{calc1even} and \ref{calc1b} above. The rest are computed in the Proposition below.
\begin{prop}\label{calc2}
Let $\alpha\in RO(\GG)$. \\
a) Suppose $|\alpha^{C_p}| > 0, |\alpha^{C_q}| >0$. Then 
$$\uH^{\alpha}_{\GG}(S^0) = 
\begin{cases} 
\bbZ & \mbox{if}~|\alpha|< 0~\mbox{and}~|\alpha^{\GG}|=0\\
L_{pq} \oplus \bbZ & \mbox{if}~ |\alpha|=0 ~\mbox{and}~ |\alpha^{\GG}|=0  .
\end{cases}
$$
b) Suppose $|\alpha^{C_p}| < 0, |\alpha^{C_q}| <0$. Then 
$$\uH^{\alpha}_{\GG}(S^0) = 
\begin{cases} 
0 & \mbox{if}~ |\alpha|>0~ \mbox{odd},~\mbox{and}~|\alpha^{\GG}|\leq 1 \\
R_{pq} &\mbox{if}~ |\alpha|=0~\mbox{and}~ |\alpha^{\GG}|\neq 0\\ 
R_{pq} \oplus \bbZ &\mbox{if}~ |\alpha|=0 ~\mbox{and}~ |\alpha^{\GG}|=0  \\
\bbZ \oplus \CC_p\bZp\oplus \CC_q\bZq &\mbox{if}~ |\alpha|> 0~\mbox{and}~|\alpha^{\GG}|=0\\
\CC_p\bZp \oplus \CC_q\bZq & \mbox{if}~ |\alpha|>0 ~ \mbox{even},~\mbox{a
nd} ~ |\alpha^{\GG}|\neq 0\\
\bbZpq &\mbox{if}~  |\alpha|>0~\mbox{odd},~\mbox{and}~|\alpha^{\GG}|>1.
\end{cases}
$$
\end{prop}

\begin{proof}
For a) we start with the computation $|\alpha|>0, |\alpha^{\GG}|=0$ of Proposition \ref{sign} b) and subtract copies of $\xi$. Thus, when $|\alpha|=0, |\alpha^{\GG}|=0$, we have the exact sequence by Proposition \ref{shex-e} b) 
$$0\to L_{pq} \to \uH^{\alpha}_{\GG}(S^0) \to \bbZ\to 0$$
which is split exact by Proposition \ref{Cpqext}. Next, for $|\alpha|=-2$, Proposition \ref{Lewis orbits} implies that the only non-zero part of $\uH^\alpha_{\GG}(S^0)$ is at $\GG/\GG$. This is computed from Proposition \ref{shex-e} c) to be $\Z$, thus proving a).  

Now turn to b). We start with  $|\alpha|=0$ and consider the exact sequence \eqref{ex-e} 
$$0\to \uH^{\alpha-\xi}_{\GG}(S^0) \to \uH^\alpha_{\GG}(S^0) \to R_{pq} \to \uH^{\alpha-\xi+1}_{\GG}(S^0) \to \uH^{\alpha+1}_{\GG}(S^0) \to L_{pq} \to \cdots $$
Proposition \ref{shex-e} b) implies that the map from $L_{pq}$ onto the next term is injective, and Proposition \ref{sign} b) implies that $\uH^{\alpha -\xi +1}_{\GG}(S^0) \cong \bbC$ for some $C$. Proposition \ref{Cpqmap} implies that any map of Mackey functors from $R_{pq}$ to $\bbC$ is $0$. Thus  for $|\alpha|=0$ one has the short exact sequence 
$$0\to \uH^{\alpha - \xi}_{\GG}(S^0) \to \uH^\alpha_{\GG}(S^0) \to R_{pq}\to 0$$ 
and for $|\alpha|=1$, 
$$\uH^{\alpha - \xi}_{\GG}(S^0) \cong \uH^\alpha_{\GG}(S^0).$$ 
We directly conclude the first, second and last statements of b). If $|\alpha|=|\alpha^{\GG}|=0$, the short exact sequence looks like
$$0\to \bbZ \to \uH^\alpha_{\GG}(S^0) \to R_{pq}\to 0$$ 
which is realized to be an element of $\Ext_{\BB_{\GG}}^1(R_{pq}, \bbZ)$. By Proposition \ref{Cpqext}, the sequence splits proving the third statement of b). 

Lastly consider $|\alpha|=2$. We have the short exact sequence by Proposition \ref{shex-e} b)
$$0\to L_{pq} \to \uH^{\alpha-\xi}_{\GG}(S^0) \to \uH^\alpha_{\GG}(S^0)\to 0$$
In the remaining cases for b) we have $\uH^{\alpha-\xi}_{\GG}(S^0)\cong R_{pq}$ or $R_{pq}\oplus \bbZ$. The factor $L_{pq}$ maps to $R_{pq}$ (and trivially onto $\bbZ$ by Proposition \ref{Cpqmap}) the cokernel of which has been computed in the proof of Proposition \ref{sign} as $\CC_p\bZp \oplus \CC_q\bZq$. This completes the proof of the Proposition.
\end{proof}

We move on to the next result which states that the $RO(\GG)$-graded cohomology groups are a function of the dimension of the fixed points. This is proved for the group $C_p$ in Stong's calculation of $\uH^\alpha_{C_p}(S^0)$ (Proposition \ref{Lewis orbits}).  

\begin{thm}\label{ind}
For every subgroup $H$, up to isomorphism, the groups $H^\alpha_{\GG}(\GG/H)$ depends only on the four numbers $|\alpha|, |\alpha^{C_p}|, |\alpha^{C_q}|, |\alpha^{\GG}|$.
\end{thm}

\begin{proof}
If $H\neq \GG$, this follows from Propositions \ref{AG}, \ref{G/Cp} and \ref{Lewis orbits}. Thus, it suffices to prove this for $H=\GG$. Given a $4$-tuple $\ua=(a_e,a_p,a_q,a_{pq})$ in which all the four numbers have the same parity, we say $\ua$ satisfies $\PP$ if for every $\alpha$ with fixed points as in $\ua$ have isomorphic values of $\tilde{H}^\alpha_{\GG}(S^0)$. We need to prove that every $\ua$ satisfies $\PP$. 

From Proposition \ref{shex-e} we see that, if we know this for those $\ua$ with $a_e\gg 0$ and $a_e\ll 0$, then we can prove this for every $\ua$. Analogously from Proposition \ref{shex-p}, we see that it suffices to assume additionally $a_p\gg 0$ or $a_p\ll 0$ and $a_q\gg 0$ or $a_q \ll 0$. This leaves us with $8$ cases determined by whether $a_e, a_p $ or $a_q$ are $\gg 0$ or $\ll 0$. The two cases where all have the same sign are dealt with in Proposition \ref{sign}. 

We divide the argument into two cases $a_e\ll 0$ and $a_e \gg 0$ starting with the first case. We note that Proposition \ref{calc1odd} implies $\PP$ for all combinations $a_e\ll 0$ odd except the case $a_{pq}=1$. We apply Proposition \ref{calc1b} b) with $|\alpha|<0$ and $|\alpha^{\GG}|=1$ together with Proposition \ref{shex-p} c) to deduce property $\PP$ for $a_e\ll 0$ odd and $a_{pq}=1$. Proposition \ref{calc1even} implies property $\PP$ for all possible combinations for $a_e\ll 0$ even except the one where $a_{pq}=0$. Proposition \ref{sign} implies $\PP$ for $a_e, a_p, a_q <0$ and $a_{pq}=0$ and Proposition \ref{calc2} implies $\PP$ for $a_e\ll 0$, $a_p, a_q > 0$ and $a_{pq}=0$. Using the symmetry between $p$ and $q$, the only remaining case to prove here is $a_e\ll 0$, $a_p < 0$, $a_q>0$ and $a_{pq}=0$. Proposition \ref{shex-p} b) and Proposition \ref{calc2} a) implies that if $|\alpha| <0, |\alpha^{C_p}|=0, |\alpha^{C_q}|>0, |\alpha^{\GG}|=0$, $\tilde{H}^\alpha_{\GG}(S^0) \cong \Z\oplus \Z$. This computation and  Proposition \ref{shex-p} e) implies that   if $|\alpha| <0, |\alpha^{C_p}|=-2, |\alpha^{C_q}|>0, |\alpha^{\GG}|=0$, $\tilde{H}^\alpha_{\GG}(S^0) \cong \Z$. This can now be used to prove $\PP$ for $a_e \ll 0$, $a_p\ll 0$, $a_q\gg 0$, $a_{pq}=0$. 

The remaining cases are $a_e \gg 0$. Proposition \ref{calc2} b) implies $\PP$ whenever $a_e \gg 0$, $a_p \ll 0 $ and $a_q \ll 0$. The symmetry between $p$ and $q$ leaves the only remaining case as $a_e \gg 0$, $a_p \ll 0$ and $a_q \gg 0$. Suppose that $a_e$ is even. In view of Proposition \ref{shex-p} d), it suffices to prove $\PP$ for $a_e\gg 0$ $a_p =0$ and $a_q \gg 0$. Proposition \ref{shex-p} b) and Proposition \ref{sign} a) implies $\PP$ for all cases other than $a_{pq} <0$ even. 

We know that if $|\alpha|>0$, $|\alpha^{C_p}|<0$, $|\alpha^{C_q}|=-2$ and $|\alpha^{\GG}|< 0$, $\uH^\alpha_{\GG}(S^0) \cong \CC_p\bZp  \oplus \CC_q \bZq$. Adding a factor of $\xi^q$ gives the exact sequence \eqref{ex-p}
$$  0 \to \CC_q \bZq \to  \CC_p\bZp  \oplus \CC_q \bZq \to \uH^{\alpha+\xi^q}_{\GG}(S^0) \to \KK_q\bZ \to 0$$
The term to the left of $\CC_q\bZq$ and the term to the right of $\KK_q\bZ$ are $0$ by Proposition \ref{calc2} a).  One readily deduces that $\uH^{\alpha+\xi^q}_{\GG}(S^0) \cong \CC_p\bZp \oplus \KK_q\bZ$ from Proposition \ref{cpzp}. Adding a further factor of $\xi^q$ induces the short exact sequence (Proposition \ref{shex-p} b))
$$0\to \CC_q \bZ \to \CC_p\bZp \oplus \KK_q\bZ \to \uH^{\alpha+2\xi^q}_{\GG}(S^0) \to 0$$
Any map from $\CC_q\bZ \to \CC_p\bZp$ is zero by Proposition \ref{Cpqmap} and the cokernel of the map $\CC_q\bZ \to \KK_q \bZ$ is $\bbZp$. Thus, we obtain $\uH^{\alpha+2\xi^q}_{\GG}(S^0) \cong \CC_p\bZp \oplus \bbZp$. Therefore, if $\alpha$ is even with $|\alpha|>0$, $|\alpha^{C_p}|<0$, $|\alpha^{C_q}|>0$ and $|\alpha^{\GG}|< 0$, then, $\uH^\alpha_{\GG}(S^0)\cong  \CC_p\bZp \oplus \bbZp$. This completes the verification of $\PP$ for $a_e\gg 0$, $a_p\ll 0$, $a_q \gg 0$ and $a_{pq}< 0$ with all four integers even.

Suppose $a_e$ is odd. If $a_{pq} \leq 1$, Proposition \ref{calc2} b) and Proposition \ref{shex-p} b) allows us to deduce $\PP$ in this case by proving that the cohomology group is actually $0$. Thus, assume $a_{pq} \geq 3$. We start with $\alpha$ an odd element with $|\alpha|\gg 0 , |\alpha^{C_q}|\geq 3$, $|\alpha^{\GG}|\geq 3$ and $|\alpha^{C_p}| =1$ and subtract $\xi^p$ to obtain \eqref{ex-p}
$$       \cdots  \to \uH^{\alpha -1}_{\GG}(S^0)    \to     \KK_p \bZ \to  \uH^{\alpha - \xi^p}_{\GG}(S^0)  \to 0$$
The Mackey functor $\uH^{\alpha -1}_{\GG}(S^0) \cong \CC_p\bZ$ by Proposition \ref{shex-p} b). It follows that $\uH^{\alpha - \xi^p}_{\GG}(S^0) \cong \bbZq$. Thus we deduce $\PP$ for $a_e \gg 0 $ odd, $a_p \ll 0$, $a_q \gg 0$ and $a_{pq} \geq 3$ which was the remaining case. 
\end{proof}

We summarize some computations carried out in the proof above in the Proposition below. 
\begin{prop}\label{calc2b}
1) If $|\alpha|>0$, $|\alpha^{C_p}|<0$, $|\alpha^{C_q}|=0$ and $|\alpha^{\GG}|< 0$, then $\uH^{\alpha}_{\GG}(S^0) \cong \CC_p\bZp \oplus \KK_q\bZ$. \\
2) If $\alpha$ is even with $|\alpha|>0$, $|\alpha^{C_p}|<0$, $|\alpha^{C_q}|>0$ and $|\alpha^{\GG}|< 0$, then $\uH^{\alpha}_{\GG}(S^0) \cong \CC_p\bZp \oplus \bbZp$. \\
3) If $\alpha$ is odd with $|\alpha|>0$, $|\alpha^{C_p}|<0$, $|\alpha^{C_q}|>0$ and $|\alpha^{\GG}|\geq 3$, then $\uH^{\alpha}_{\GG}(S^0) \cong  \bbZq$.
\end{prop}

\section{Computations of $\uH^\alpha_{\GG}(S^0)$}\label{comp2}
In this section, we complete the rest of the computations of the cohomology Mackey functors $\uH^\alpha_{\GG}(S^0)$  left over from the previous sections. 

Let us start with $\alpha$ odd. If $|\alpha|<0$ all the computations except the case $|\alpha^{\GG}|=1$ is done in Propositions \ref{sign} and \ref{calc1odd}. Further using Proposition \ref{calc1b}, the only case remaining is either $|\alpha^{C_q}|>1$ or $|\alpha^{C_q}|>1$. For $|\alpha|>0$, the case $\{ |\alpha^{C_p}|>0, |\alpha^{C_q}|>0\}$ is covered in Proposition \ref{sign} and the case  $\{ |\alpha^{C_p}|<0, |\alpha^{C_q}|<0\}$ is covered in Proposition \ref{calc2} b).  Moreover, Proposition \ref{calc2b} computes the case $|\alpha^{C_p}|<0$, $|\alpha^{C_q}|>0$ and $|\alpha^{\GG}| >1$. The remaining cases are dealt with in the Proposition below. 
\begin{prop} \label{calc3}
Let $\alpha\in RO(\GG)$ be an odd element. \\
a) Suppose $|\alpha| < 0$ and $|\alpha^{\GG}|=1$. Then 
$$\uH^{\alpha}_{\GG}(S^0) = 
\begin{cases} 
\KK_p \bZp. & \mbox{if}~ |\alpha^{C_p}| > 1 ~\mbox{and}~ |\alpha^{C_{q}}|\leq 1 \\
\KK_p\bZp \oplus \KK_q \bZq. &\mbox{if}~ |\alpha^{C_p}|> 1 ~\mbox{and}~ |\alpha^{C_{q}}|>1 .
\end{cases}
$$
b) If $|\alpha|>0$, $|\alpha^{C_p}|<0$, $|\alpha^{C_{q}}|>0$ and $|\alpha^{\GG}|\leq 1$, then $\uH^{\alpha}_{\GG}(S^0) = 0.$
\end{prop}

\begin{proof}
Note that a) follows from the Proposition \ref{calc1b} b) and Proposition \ref{shex-p} c). The statement b) is deduced by applying Proposition \ref{calc1b} b) for $|\alpha|=|\alpha^{C_q}|=1$ and then adding copies of $\xi$ and $\xi^q$ using Propositions \ref{shex-e} a) and \ref{shex-p} a).
\end{proof}

Next turn to the case $\alpha$ even. Among all such $\alpha$, we first consider $|\alpha|\neq 0$ and $|\alpha^{\GG}|\neq 0$. We note from Propositions \ref{sign} and \ref{calc1even} that all cases with $|\alpha|<0$ and $|\alpha^{\GG}|\neq 0$ have been computed except $|\alpha|<0, |\alpha^{\GG}|<0$ and $|\alpha^{C_p}|=|\alpha^{C_q}|=0$. We have 
\begin{prop}\label{calc3evneg}
Let $\alpha \in RO(\GG)$ be even with $|\alpha| < 0$, $|\alpha^{\GG}| < 0$, $|\alpha^{C_p}|=|\alpha^{C_q}|=0$ . Then, $ \uH^{\alpha}_{\GG}(S^0) = \KK_p\bZ \oplus \KK_q \bZ $.
\end{prop}

\begin{proof}
We note $\uH^{\alpha -\xi^p}_{\GG}(S^0)\cong \KK_q\bZ$ from Proposition \ref{calc1even} and $\uH^{\alpha -\xi^p+1}_{\GG}(S^0)=0$ from Proposition \ref{calc1b},  so that the exact sequence \eqref{ex-p} is of the form
$$0 \to \KK_q \bZ  \to \uH^{\alpha}_{\GG}(S^0) \to \KK_p \bZ \to 0.$$
which splits because from the symmetry of $p$ and $q$ we obtain the opposite short exact sequence as well. 
\end{proof}

Now let $|\alpha|>0$. We note that in Propositions \ref{sign} and \ref{calc2} we have made all computations when $|\alpha^{C_p}|$ and $|\alpha^{C_q}|$ are both non-zero of the same sign. Further, in Proposition \ref{calc2b}, computations are done for $|\alpha^{C_p}|\geq 0$, $|\alpha^{C_q} | < 0$ and $|\alpha^{\GG}|<0$. The rest of the cases are computed in the Proposition below.   
\begin{prop}\label{calc3evpos} Let $\alpha$ be even with $|\alpha|>0$.\\ 
a) Suppose $|\alpha^{\GG}|>0$.  Then, we have
$$\uH^{\alpha}_{\GG}(S^0) = 
\begin{cases} 
\CC_q \bZ  & \mbox{if}~ |\alpha^{C_p}| >0 ~\mbox{and}~ |\alpha^{C_{q}}| =0 \\
\CC_p \bZ  \oplus \CC_q \bZ  & \mbox{if}~|\alpha^{C_p}| =0 ~\mbox{and}~ |\alpha^{C_{q}}| =0  \\
\CC_q \bZ \oplus \CC_p \bZp  & \mbox{if}~|\alpha^{C_p}|<0 ~\mbox{and}~ |\alpha^{C_{q}}| =0 \\
\CC_p \bZp  &\mbox{if}~ |\alpha^{C_p}| <0  ~\mbox{and}~ |\alpha^{C_{q}}|>0.
\\ 
\end{cases}
$$
b)  Suppose $|\alpha^{\GG}|<0$. Then, we have 
$$\uH^{\alpha}_{\GG}(S^0) = 
\begin{cases}
\KK_p \bZ \oplus \KK_q \bZ & \mbox{if}~ |\alpha^{C_p}| =0 ~\mbox{and}~ |\alpha^{C_{q}}| =0 \\
\bbZq \oplus \KK_q \bZ & \mbox{if}~|\alpha^{C_p}|>0 ~\mbox{and}~ |\alpha^{C_{q}}| =0 . \\
\end{cases}
$$
\end{prop}
\begin{proof}
The first statement of a) follows from Proposition \ref{shex-p} b). We apply the same result to the second case to deduce the short exact sequence
$$ 0 \to \CC_q \bZ  \to \uH^{\alpha}_{\GG}(S^0) \to \CC_p \bZ  \to 0$$ 
which splits because from the symmetry of $p$ and $q$ we also obtain the exact sequence with $p$ and $q$ interchanged. For the third case, we let $|\alpha^{C_p}|=-2$ and note the cohomology at $\alpha+\xi^p$ from the previous case. Then, the exact sequence \eqref{ex-p} at the index $\alpha+\xi^p$ is of the form
$$0 \to \CC_p \bZp \to \uH^{\alpha}_{\GG}(S^0) \to \CC_p \bZ \oplus \CC_q \bZ  \to \KK_p \bZ \to \cdots,$$ 
after substituting values of $\uH^\ast_{\GG}(S(\xi^p)_+)$ from Proposition \ref{p-sphere}. The left map is injective by Proposition \ref{sus-p}. The map from $\CC_q\bZ \to \KK_p\bZ$ is $0$ by Proposition \ref{Cpqmap} and the map $ \CC_p \bZ  \to \KK_p \bZ $ is injective by Proposition \ref{sus-p}.  Furthermore, the leftmost map in the above exact sequence is split injection by Proposition \ref{cpzp}. Hence the computation follows. In the final statement of a), we proceed analogously and apply the computation in the first case. Then, the exact sequence \eqref{ex-p} at the index $\alpha+\xi^p$ looks like 
$$ 0 \to \CC_p \bZp \to \uH^{\alpha}_{\GG}(S^0) \to \CC_p \bZ   \to \KK_p \bZ \to \cdots$$
and so the computation follows from the fact that the last map is injective.  

Now we turn to b). Suppose $\alpha$ is as in the first case. Observe from Proposition \ref{calc2b} 1) and Proposition \ref{p-sphere} that the exact sequence \eqref{ex-p} at $\alpha$ looks like
$$0 \to  \CC_p \bZp \to \CC_p \bZp  \oplus \KK_q \bZ \to  \uH^{\alpha}_{\GG}(S^0) \to \KK_p \bZ \to 0.$$
The map  $\CC_p\bZp \to \KK_q\bZ$ is $0$ by Proposition \ref{Cpqmap} and thus, the cokernel of $ \CC_p \bZp \to \CC_p \bZp  \oplus \KK_q \bZ $ is $\KK_q \bZ.$ So the above sequence turns out to be
 $$0 \to \KK_q \bZ \to \uH^{\alpha}_{\GG}(S^0) \to \KK_p \bZ \to 0$$
which splits again from the symmetry of $p$ and $q$. Hence the result follows.

Finally, for the second statement, start with $|\alpha^{C_p}|=2$ so that  Proposition \ref{shex-p} b) at the index $\alpha - \xi^p$ gives the short exact sequence (applying the computation above)
$$0 \to \CC_p \bZ \to \KK_p \bZ  \oplus \KK_q \bZ \to \uH^{\alpha}_{\GG}(S^0) \to 0.$$
The map $\CC_p\bZ \to \KK_q\bZ$ is $0$ by Proposition \ref{Cpqmap}. Therefore, $\uH^{\alpha}_{\GG}(S^0)$ is the cokernel of  $\CC_p \bZ \to \KK_p \bZ$ plus a copy of  $\KK_q \bZ$. It follows that $\uH^{\alpha}_{\GG}(S^0) \cong \bbZq \oplus \KK_q \bZ.$
\end{proof}

The Propositions \ref{calc3evneg} and \ref{calc3evpos} complete the computations of $\uH^\alpha_{\GG}(S^0)$ for $\alpha$ even in all cases except when either $|\alpha|=0$ or $|\alpha^{\GG}|=0$. In these cases, we first assume $|\alpha^{C_p}|$ and $|\alpha^{C_q}|$ are non-zero. In Propositions \ref{calc1even} and \ref{calc2}, all computations are made when $|\alpha^{C_p}|$ and $|\alpha^{C_q}|$ are of the same sign. In the following, we compute the case where $|\alpha^{C_p}|$ and $|\alpha^{C_q}|$ have opposite sign. 
\begin{prop}\label{calczero}
Let $\alpha$ be even such that $|\alpha^{C_p}|>0$ and $|\alpha^{C_q}|<0$. Then, 
$$\uH^{\alpha}_{\GG}(S^0) = 
\begin{cases} 
\KK_p L_p & \mbox{if}~ |\alpha|=0 ~\mbox{and}~  |\alpha^{\GG}| >0 \\
 \bbZq \oplus \KK_p L_p &\mbox{if}~ |\alpha| =0 ~\mbox{and}~ |\alpha^{\GG}| <0 \\
 \bbZ  &\mbox{if}~|\alpha| <0 ~\mbox{and}~ |\alpha^{\GG}| =0 \\
 \bbZ \oplus \KK_p L_p  &\mbox{if}~ |\alpha| =0 ~\mbox{and}~ |\alpha^{\GG}| =0 \\
\bbZ \oplus \CC_q \bZq & \mbox{if}~ |\alpha| > 0 ~\mbox{and}~ |\alpha^{\GG}| =0.
\end{cases}
$$
\end{prop}

\begin{proof}
We start with the first case. Note that $\uH^{\alpha-\xi}_{\GG}(S^0)=0$ by Proposition \ref{calc1even} b), $\uH^{\alpha-\xi+1}_{\GG}(S^0)\cong \KK_p\bZp \oplus\bbZp$ by Proposition \ref{calc1odd} b) and $\uH^{\alpha+1}_{\GG}(S^0)\cong \bbZp$ by Proposition \ref{calc2b} 3). Therefore the exact sequence \eqref{ex-e} at the index $\alpha$ looks like
$$0 \to \uH^{\alpha}_{\GG}(S^0) \to R_{pq} \to \KK_p\bZp \oplus\bbZp \to \bbZp $$ 
Proposition \ref{Cpqmap} implies that the map  $R_{pq} \to \bbZp$ is $0$. It follows that the map $R_{pq} \to \KK_p\bZp$ is surjective and the kernel is computed to be $\KK_pL_p$. In the second case, we have $\uH^{\alpha-\xi}_{\GG}(S^0)=\bbZq$ by Proposition \ref{calc1even} b), $\uH^{\alpha-\xi +1}_{\GG}(S^0)=\KK_p\bZp$ by Proposition \ref{calc1odd} b), and $\uH^{\alpha+1}_{\GG}(S^0)=0$ by Proposition \ref{calc3} b). Therefore, the exact sequence \eqref{ex-e} at the index $\alpha$ looks like 
$$0 \to \bbZq \to \uH^{\alpha}_{\GG}(S^0) \to R_{pq} \to \KK_p\bZp  \to 0$$ 
The kernel of the right hand map is $\KK_pL_p$. The resulting short exact sequence may be considered as a class in $\Ext_{\BB_{\GG}}^1(\KK_p L_p ; \bbZq)$ which is $0$ by Proposition \ref{Cpqext} b). Therefore the sequence splits and the result follows. 

In the last three cases, we start with the third statement and then proceed by adding copies of $\xi$. If $\alpha$ is as in the third case we note that $\Phi_p^\ast(\uH^\alpha_{\GG}(S^0))$ and $\Phi_q^\ast(\uH^\alpha_{\GG}(S^0))$ are $0$ by Proposition \ref{Lewis orbits}. Thus, it suffices to determine the group $\tilde{H}^\alpha_{\GG}(S^0)$. Assume $|\alpha^{C_q}|=-2$ and $|\alpha|\ll 0$. We note by Proposition \ref{calc2} a) that $\tilde{H}^{\alpha+2 \xi^q}_{\GG}(S^0) \cong \Z$ and thus by Proposition \ref{shex-p} b) that $\tilde{H}^{\alpha+\xi^q}_{\GG}(S^0)\cong \Z\oplus \Z$. Proposition \ref{shex-p} e) now implies that $\tilde{H}^\alpha_{\GG}(S^0) \cong \Z$ as required. 

For the fourth statement, we deduce the result by proceeding exactly as in the second case above, replacing $\bbZq$ with $\bbZ$. In the final case, Proposition \ref{shex-e} b) at the index $\alpha-\xi$ gives the short exact sequence
$$0 \to L_{pq} \to \bbZ \oplus \KK_pL_p \to \uH^\alpha_{\GG}(S^0)\to 0$$
The map $L_{pq} \to \bbZ$ is $0$ by Proposition \ref{Cpqmap} and the cokernel of the inclusion $L_{pq} \to \KK_p L_p$ is computed to be $\CC_q \bZq$. Hence the result follows. 
\end{proof}
 
Proposition \ref{calczero} completes the computation of $\uH^\alpha_{\GG}(S^0)$ except in the cases where one of $|\alpha^{C_p}|$ or $|\alpha^{C_q}|$ is $0$ and one of $|\alpha|$ and $|\alpha^{\GG}|$ is $0$. Apart from the explicit formulas, we also note 
\begin{thm}\label{indmack}
Suppose $\alpha \in RO(\GG)$ which is such that at least one of $|\alpha^H|$ or $|\alpha^K|$ is non-zero whenever $(K,H)\in \{(C_p,e), (C_q,e),(\GG,C_p),(\GG,C_q)\}$. Then, up to isomorphism, the Mackey functor $\uH^\alpha_{\GG}(S^0)$ depends only on the fixed points $|\alpha|, |\alpha^{C_p}|, |\alpha^{C_q}|, |\alpha^{\GG}|$.   
\end{thm}

Finally, we are left with the cases where one of $|\alpha^{C_p}|$ or $|\alpha^{C_q}|$ is $0$ and one of $|\alpha|$ and $|\alpha^{\GG}|$ is $0$. Considering the symmetric nature of $p$ and $q$, there are two cases to consider \\
1) $|\alpha|=|\alpha^{C_p}|=0$. \\
2) $|\alpha^{C_p}|=|\alpha^{\GG}|=0$. \\
  In these cases, the Mackey functor is actually dependent on the specific values of $\alpha$. In the first case we note $\Phi_p^\ast \uH^\alpha_{\GG}(S^0) \cong \uH^\alpha_{C_p}(S^0)$ is dependent on specific choices of $\alpha$ by the formulas in \cite{Lew88} (see Proposition \ref{Lewis orbits}). In the second case, we observe that after applying $\Phi_p^\ast$ the resultant Mackey functor is independent, so the above argument does not apply. However, these functors may also be checked to be dependent, and we provide an outline below.

Let us start with an $\alpha\in RO(C_{q})$ with $|\alpha|=|\alpha^{C_q}|=0$ so that $\uH^\alpha_{C_q}(S^0)=\uA_q[d_\alpha]$. We also treat $\alpha$ as an element $RO(\GG)$ by putting the trivial $C_p$-action, and proceed to compute the Mackey functor $\uH^\alpha_{\GG}(S^0)$ using the ideas in \cite{Lew88}, Definitions 4.5. The Mackey functor $\uM=\uH^\alpha_{\GG}(S^0)$ is of the form (applying Theorem \ref{ind})    
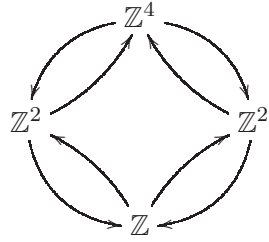
$$\xymatrix{ & \Z^4  \ar@/_1pc/[dl]  \ar@/^1pc/[dr]  \\
  \Z^2 \ar@/_1pc/[dr] \ar@/_/[ur] & & \Z^2 \ar@/^/[ul] \ar@/^1pc/[dl] \\
 &  \Z \ar@/_/[ul] \ar@/^/[ur] }$$ 
The generators may be written as 
$$\mu_e \in \uM(\GG/e),\, \mu_p, \tau^p(\mu_e) \in \uM(\GG/C_p),\, \mu_q, \tau^q(\mu_e) \in \uM(\GG/C_q),$$
$$\mu_{pq},\tau_p (\mu_p), \tau_q (\mu_q),  \tr^{\GG}_e (\mu_e) \in \uM(\GG/\GG).$$
 The given conditions on $\alpha$ imply that we may assume $\res^p(\mu_p)=\mu_e$, $\res_q(\mu_{pq}) = \mu_q$ and $\res_q(\mu_{pq}) = d_\alpha \mu_q$. It follows that $\uH^\alpha_{\GG}(S^0) \cong \cA_q\uA_q[d_\alpha]$. Thus, the value of the Mackey functor depends on $\alpha$ but we observe that $\alpha$ satisfies 1) and 2). Now let $\beta=\alpha+\xi^q$ and apply \eqref{ex-p} for the index $\beta$ and $q$ in place of $p$. After substituting the known computations from above, we obtain a short exact sequence 
$$0\to \CC_q\uA_q[d_\alpha] \to \cA_q \uA_q[d_\alpha] \to \uH^\beta_{\GG}(S^0) \to 0. $$
We may compute the cokernel using the fact that the left map is an isomorphism at levels $\GG/C_q$ and $\GG/e$ to deduce $\uH^\beta_{\GG}(S^0) \cong  \QQ_q\uA_q[d_\alpha]$. Now $\beta$ satisfies 2) but not 1) and demonstrates that in the case 2), the cohomology Mackey functors depend on specific choices. 
 
At this point, it is possible to formalize the ideas above into a complete computation of $\uH^\alpha_{\GG}(S^0)$ in the tune of \cite{Lew88}. However, this would lead us to a different direction from the techniques of the paper and so we avoid it. 

\section{Computations with constant coefficients} \label{constcoeff} 
In this section we compute $\uH^\alpha_\GG(S^0;R_{pq})$, the Bredon cohomology of $S^0$ for the constant Mackey functor $R_{pq}$. One method is to proceed along the lines of Sections \ref{sph-coh}, \ref{comp1}, \ref{comp2} and deduce the computations for the constant Mackey functor. We take a different route by using the computations of $H^\alpha_\GG(S^0;\uA)$ along with certain relations between $\GG$-Mackey functors. 

Recall from the proof of Proposition \ref{Cpmack} that one has short exact sequences 
$$0\to \bZ \to \uA_p \to R_p \to 0, \; \; \; \;  0\to L_p \to \uA_p \to \bZ \to 0,$$
and that the composite $\bZ \to \uA_p \to \bZ$ using the maps above is multiplication by $p$. This induces short exact sequences of $\GG$-Mackey functors 
\begin{equation} \label{shex-mack}
\begin{array}{l}
0\to \cA_q\bZ \to \uA \to \cA_qR_q \to 0, \; \; \; \;  0\to \cA_qL_q \to \uA \to \cA_q\bZ \to 0, \\
0\to \KK_p\bZ \to \KK_p\uA_p \to R_{pq} \to 0, \; \; \; \;  0\to \KK_pL_p \to \KK_p\uA_p \to \KK_p\bZ \to 0.
\end{array}
\end{equation}
We also note that $\cA_qR_q \cong \KK_p\uA_p$. The composite $\cA_q\bZ \to \uA \to \cA_q\bZ$ is multiplication by $q$, and the composite $\KK_p\bZ \to \KK_p\uA_p \to \KK_p \bZ$ is multiplication by $p$. 

We start by computing cohomology with $\cA_q\bZ$-coefficients and $\KK_p\bZ$-coefficients. Next, using the top line of \eqref{shex-mack} and the computations with  $\uA$-coefficients, we compute cohomology with  $\cA_qR_q$-coefficients. Finally using results from $\cA_qR_q$-coefficients together with the bottom line of \eqref{shex-mack}, we complete the computation for $R_{pq}$-coefficients. 

\begin{prop}\label{bZcomp} \mbox{  }\\
a) $ \uH^\alpha_\GG(S^0;\cA_q\bZ) = \begin{cases} 
\KK_q\bZ  & \mbox{if}~|\alpha^{C_q}|=0,~|\alpha^\GG|<0 \\ 
\CC_q\bZ  & \mbox{if}~|\alpha^{C_q}|=0,~|\alpha^\GG|>0 \\
\bbZ  & \mbox{if}~|\alpha^{C_q}|\neq 0,~|\alpha^\GG|=0 \\ 
\bbZp  & \mbox{if}~|\alpha^{C_q}|>0,~|\alpha^\GG|<0~\mbox{even} \\ 
\bbZp  & \mbox{if}~|\alpha^{C_q}|<0,~|\alpha^\GG|>1 ~\mbox{odd}\\ 
 0           &\mbox{otherwise}. 
\end{cases} $\\
b)  $ \uH^\alpha_\GG(S^0;\KK_q\bZ) = \begin{cases} 
\KK_q\bZ  & \mbox{if}~|\alpha^{C_q}|=0,~|\alpha^\GG|\leq 0 \\ 
\CC_q\bZ  & \mbox{if}~|\alpha^{C_q}|=0,~|\alpha^\GG|>0 \\ 
\bbZp  & \mbox{if}~|\alpha^{C_q}|>0,~|\alpha^\GG|\leq 0~\mbox{even} \\ 
\bbZp  & \mbox{if}~|\alpha^{C_q}|<0,~|\alpha^\GG|>1 ~\mbox{odd}\\ 
 0           &\mbox{otherwise}. 
\end{cases} $
\end{prop}

\begin{proof}
For $\uM=\cA_q\bZ$ and $\KK_q\bZ$, we note that $\uM(\GG/e)=0$, $\uM(\GG/C_p)=0$ and $\uM(\GG/C_q)=\Z$. It follows from the techniques of Proposition \ref{G/Cp} that $\uH^\alpha_\GG(\GG/e_+;\uM)=0,$ $\uH^\alpha_\GG ({\GG/C_p}_+; \uM)=0,$ and $\uH^\alpha_\GG({\GG/C_q}_+ ;\uM) = E_q\uH^\alpha_{C_q}(S^0;\bZ) $. Thus we have (using the cofibres \eqref{cofe}, \eqref{cofp})  
$$\uH^\alpha_\GG(S(\xi)_+;\uM)=0,~ \uH^\alpha_\GG(S(\xi^p)_+;\uM)=0, ~\uH^\alpha_\GG(S(\xi^q)_+;\uM)= \begin{cases} 
\KK_q \bZ &\mbox{if}~|\alpha^{C_q}|=0 \\ 
\CC_q \bZ &\mbox{if}~|\alpha^{C_q}|=1 \\
0 &\mbox{otherwise}. 
\end{cases} $$
Using \eqref{ex-e} and \eqref{ex-p} we readily have that $\uH^\alpha_\GG(S^0;\uM) \cong \uH^{\alpha +\xi}_\GG(S^0;\uM)$ and $\uH^\alpha_\GG(S^0;\uM) \cong \uH^{\alpha +\xi^p}_\GG(S^0;\uM)$, so that the value of $\uH^\alpha_\GG(S^0;\uM)$ depends only on $\alpha^{C_q}$. Further, we have the following results along the lines of Proposition \ref{shex-p}\\
1) $\uH^\alpha_\GG(S^0;\uM) \cong \uH^{\alpha +\xi^q}_\GG(S^0;\uM)$ if $|\alpha^{C_q}|\geq 1$ or $|\alpha^{C_q}|\leq -3$. \\
2) If $|\alpha^{C_q}|=0$ we have an exact sequence $0\to \CC_q\bZ \to \uH^\alpha_\GG(S^0;\uM) \to \uH^{\alpha+\xi^q}_\GG(S^0;\uM)\to 0$. \\
3) If $|\alpha^{C_q}|=-1$ and $\uH^\alpha_\GG(S^0;\uM)=0$, then $\uH^{\alpha+\xi^q}_\GG(S^0;\uM)=0$. \\
In view of 1), it suffices to prove the Proposition for $-2 \leq |\alpha^{C_q}| \leq 2$. 

We apply 2) for $\alpha=0$ to obtain 
\begin{equation}\label{xiqbz}
\uH^{\xi^q}_\GG(S^0;\cA_q\bZ) \cong \bbZ, \;  ~~ \uH^{\xi^q}_\GG(S^0;\KK_q\bZ) \cong \bbZp.
\end{equation}
We also apply the exact sequence \eqref{ex-p}
$$0\to \uH^{-\xi^q}_\GG(S^0;\uM) \to \uH^0_\GG(S^0;\uM)\to \uH^0_\GG(S(\xi^q)_+;\uM) \to \uH^{1-\xi^q}_\GG(S^0;\uM)\to 0 $$
and use the fact that $\uH^{1-\xi^q}_\GG(S^0;\uM)(\GG/C_q)=0$ to obtain 
\begin{equation}\label{-xiqbz}
\uH^{-\xi^q}_\GG(S^0;\cA_q\bZ) \cong \bbZ, \;  ~~ \uH^{-\xi^q}_\GG(S^0;\KK_q\bZ) =0, \; ~~\uH^{1-\xi^q}_\GG(S^0;\uM) =0 .
\end{equation}

Next, we apply the same argument as in the proof of Proposition \ref{sign} to deduce that in the case $|\alpha^{C_q}|\neq 0$, the value of $\uH^\alpha_\GG(S^0;\uM)$ only depends on  $\mathit{sign}(|\alpha^{C_q}|)$ and $|\alpha^\GG|$. Therefore, if $|\alpha^\GG|\neq 0$ and has the same sign as $|\alpha^{C_q}|$, $\uH^\alpha_\GG(S^0;\uM)\cong \uH^{|\alpha^\GG|}_\GG(S^0;\uM)=0$. Further, from 2) and  3) above, we obtain the required computation for $(|\alpha^{C_q}|=0,|\alpha^\GG|>0)$ and $(|\alpha^{C_q}|=1, |\alpha^\GG|<0)$. Using these identifications together with \eqref{xiqbz} and \eqref{-xiqbz}, we have verified all cases with $|\alpha^{C_q}|\neq 0$ except $(|\alpha^{C_q}|=2, |\alpha^\GG|<0)$,  $(|\alpha^{C_q}|=-1, |\alpha^\GG|>1)$, and  $(|\alpha^{C_q}|=-2, |\alpha^\GG|>0)$. 

Suppose $|\alpha^{C_q}|=0$ and $|\alpha^\GG|>0$. The exact sequence \eqref{ex-p} for this grading looks like 
$$0\to  \uH^{\alpha -\xi^q}_\GG(S^0;\uM) \to \CC_q\bZ \to \KK_q\bZ \to \uH^{\alpha+ 1-\xi^q}_\GG(S^0;\uM) \to 0$$
The map $\CC_q\bZ \to \KK_q\bZ$ is an isomorphism at $\GG/C_q$, and it follows that $ \uH^{\alpha -\xi^q}_\GG(S^0;\uM)=0$ and $ \uH^{\alpha+ 1-\xi^q}_\GG(S^0;\uM) \cong \bbZp$. These calculate the cases $(|\alpha^{C_q}|=-2, |\alpha^\GG|>0)$ and $(|\alpha^{C_q}|=-1, |\alpha^\GG|>1)$.

Now suppose $|\alpha^{C_q}|=0$ and $|\alpha^\GG|<0$. We know $\uH^{\alpha - \xi^q}_\GG(S^0;\uM) = 0$ and $\uH^{\alpha +1 - \xi^q}_\GG(S^0;\uM)=0$, so that $\uH^\alpha_\GG(S^0;\uM) \cong \uH^\alpha_\GG(S(\xi^q)_+;\uM)\cong \KK_q\bZ$. From 2), we have $\uH^{\alpha+\xi^q}_\GG(S^0;\uM)\cong \bbZp$ which verifies the case $(|\alpha^{C_q}|=2, |\alpha^\GG|<0)$. 

Finally, it remains to compute for $|\alpha^{C_q}|=0, |\alpha^\GG|=0$ for $\KK_q\bZ$-coefficients. In this case, note from the above that $\uH^{\alpha-\xi^q}_\GG(S^0;\KK_q\bZ)=0$ and $\uH^{\alpha+1-\xi^q}_\GG(S^0;\KK_q\bZ)=0$. Hence, $\uH^\alpha_\GG(S^0;\KK_q\bZ) \cong \uH^\alpha_\GG(S(\xi^q)_+;\KK_q\bZ)= \KK_q\bZ$. 
\end{proof}

We directly observe from the expressions in Proposition \ref{bZcomp} that multiplication by $q$ is injective on $\uH^\alpha_\GG(S^0;\cA_q\bZ)$ whenever $(|\alpha^{C_q}|,|\alpha^\GG|)\neq (0,0)$. It follows that $\uH^\alpha_\GG(S^0;\cA_q\bZ)\to \uH^\alpha_\GG(S^0;\uA)$ induced from the top row of \eqref{shex-mack} is injective. Analogously, $\uH^\alpha_\GG(S^0;\KK_p\bZ)\to \uH^\alpha_\GG(S^0;\KK_p\uA_p)$ is also injective. It follows that we have short exact sequences 
\begin{equation}\label{coeff-ex1}
0\to \uH^\alpha_\GG(S^0;\cA_q\bZ)\to \uH^\alpha_\GG(S^0;\uA) \to \uH^\alpha_\GG(S^0;\cA_q R_q)\to 0  ~\mbox{if}~(|\alpha^{C_q}|,|\alpha^\GG|)\neq (0,0),
\end{equation}
\begin{equation}\label{coeff-ex2}
0\to \uH^\alpha_\GG(S^0;\KK_p\bZ)\to \uH^\alpha_\GG(S^0;\KK_p\uA_p) \to \uH^\alpha_\GG(S^0; R_{pq})\to 0 . 
\end{equation}
For our computations, we also need a relation between Mackey functors $\cA_qR_q$ and $\cA_qL_q$, and between $R_{pq}$ and $\KK_pL_p$. We prove this in the Proposition below. 
\begin{prop}\label{rel-RL}
There are equivalences of $\GG$-spectra 
$$\Sigma^{2-\xi^p}H(\cA_qR_q) \simeq H(\cA_qL_q), \; ~~ \Sigma^{2-\xi^q}HR_{pq} \simeq H(\KK_pL_p).$$ 
It follows that 
$$\uH^\alpha_\GG(S^0;\cA_qL_q)\cong \uH^{\alpha+2-\xi^p}_\GG(S^0;\cA_qR_q), ~ \uH^\alpha_\GG(S^0;\KK_pL_p) \cong \uH^{\alpha+2-\xi^q}_\GG(S^0;R_{pq}).$$
\end{prop}

\begin{proof}
We compute the Mackey functor valued homotopy groups $\underline{\pi_n}(\Sigma^{2-\xi^p}H(\cA_qR_q))$ which is isomorphic to  $ \uH^{2-n-\xi^p}_\GG(S^0;\cA_qR_q)$. This is computed using the long exact sequence 
$$ \uH^{2-n-\xi^p}_\GG(S^0;\cA_qR_q) \to \uH^{2-n}_\GG(S^0;\cA_qR_q) \to \uH^{2-n}_\GG(S(\xi^p)_+;\cA_qR_q) \to \uH^{2-(n-1)-\xi^p}_\GG(S^0;\cA_qR_q) $$
The value of $\uH^{2-n}_\GG(S(\xi^p)_+;\cA_qR_q)$ is easily computed using \eqref{cofp} as $\cA_qR_q$ if $n=2$ and $\cA_qL_q$ if $n=1$. For $n=2$, the map $\cA_qR_q\cong \uH^{2-n}_\GG(S^0;\cA_qR_q) \to \uH^{2-n}_\GG(S(\xi^p)_+;\cA_qR_q)$ is an isomorphism, and for $n=1$, the map $ \uH^{2-n}_\GG(S(\xi^p)_+;\cA_qR_q) \to \uH^{2-(n-1)-\xi^p}_\GG(S^0;\cA_qR_q)$ is an isomorphism. Therefore, we have the computation 
 $$\underline{\pi_n}(\Sigma^{2-\xi^p}H(\cA_qR_q))= \begin{cases} \cA_qL_q &\mbox{if}~n=0 \\ 0 & \mbox{if}~n\neq 0. \end{cases} $$
Hence, $\Sigma^{2-\xi^p}H(\cA_qR_q) \simeq H(\cA_qL_q)$. A similar computation proves $\Sigma^{2-\xi^q}HR_{pq} \simeq H(\KK_pL_p)$.
\end{proof}

Now we begin the computations of $\uH_\GG^\alpha(S^0;\cA_qR_q)$ starting with the $|\alpha|$ odd case. In this case, the only non-zero values for $\uH_\GG^\alpha(S^0;\cA_q\bZ)$ are $\bbZp$ when $(|\alpha^{C_q}|<0, |\alpha^\GG|>1|)$  (Proposition \ref{bZcomp}). The values for the cohomology can be directly read off from the exact sequence \eqref{coeff-ex1},  and Table \ref{odd-tab}. We list these in Table \ref{odd-tabaqrq}. 

\begin{table}[ht]

\begin{tabular}{ |p{4.8cm}|p{2.5cm}||p{4.8cm}| p{1.7cm}|  }
 \hline
{\tiny $\alpha$ } & { \tiny $\uH_\GG^\alpha(S^0;\cA_qR_q)$} & {\tiny $\alpha$ } & {\tiny $\uH_\GG^\alpha(S^0;\cA_qR_q)$} \\
 \hline
{\tiny $|\alpha|>0,|\alpha^\GG|\leq 1$}  & {\tiny $0$ } 
& {\tiny $|\alpha|>0,|\alpha^{C_p}|>0, |\alpha^\GG|> 1$}  &  {\tiny $0$ }  \\
\hline
{\tiny $|\alpha|>0,|\alpha^{C_p}|<0,|\alpha^\GG|> 1$}  &  {\tiny $\bbZq$ }
& {\tiny $|\alpha|<0,|\alpha^{C_p}|>1,|\alpha^{C_q}|\leq 1$}  &  {\tiny $\KK_p\bZp$ }\\
\hline
{\tiny $|\alpha|<0,|\alpha^{C_p}|> 1,|\alpha^{C_q}|> 1$}  &  {\tiny $\KK_p \bZp\oplus \KK_q \bZq$ }
& {\tiny $|\alpha|<0,|\alpha^{C_p}|=1,|\alpha^{C_q}|\leq1$}  &  {\tiny $0$ }\\
\hline
{\tiny $|\alpha|<0,|\alpha^{C_p}|=1,|\alpha^{C_q}|>1$}  &  {\tiny $\KK_q\bZq$ }
& {\tiny $|\alpha|<0,|\alpha^{C_p}|<0, |\alpha^{C_q}|>1, |\alpha^{\GG}|\leq 1$}  &  {\tiny $\KK_q\bZq$ }\\
\hline
{\tiny $|\alpha|<0,|\alpha^{C_p}|<0,|\alpha^{C_q}|\leq 1,|\alpha^{\GG}| \leq 1$}  &  {\tiny $0$ }
& {\tiny $|\alpha|<0,|\alpha^{C_p}|<0,|\alpha^{C_q}|\leq 1, |\alpha^\GG|>1$}  &  {\tiny $\bbZq$ }\\
 \hline
{\tiny $|\alpha|<0,|\alpha^{C_p}|<0,|\alpha^{C_q}|>1, |\alpha^\GG|>1$}  &  {\tiny $\KK_q \bZq\oplus \bbZq$ }
&  & \\
 \hline
\end{tabular}
\vspace{.2cm}
\caption{Formula for $\uH_{\GG}^\alpha(S^0;\cA_qR_q)$ for $|\alpha|$ odd.}
\label{odd-tabaqrq}
\end{table}

Next, we compute $\uH_\GG^\alpha(S^0;\cA_qR_q)$ for $|\alpha|<0$ even. The values for $\uH_\GG^\alpha(S^0;\uA)$ are listed in Table \ref{evneg-tab}. We list the values of our calculations in Table \ref{evneg-tabaqrq}. These calculations can be directly obtained from the exact sequence \eqref{coeff-ex1} if both $|\alpha^{C_q}|$ and $|\alpha^\GG|$ are non-zero. 

We demonstrate one computation in the case where $|\alpha^{C_p}|<0$, $|\alpha^{C_q}|<0$ and $|\alpha^\GG|=0$. The other cases where at least one of $|\alpha^{C_q}|$ or $|\alpha^\GG|$ is $0$ are similar.  For this case, $\uH^\alpha_\GG(S^0;\cA_q\bZ) \cong \bbZ$ (Proposition \ref{bZcomp}) and $\uH^\alpha_\GG(S^0;\uA) \cong \bbZ$. We know that $\cA_q\bZ \to \uA \to \cA_q\bZ$ is multiplication by $q$, so that the map $\uH^\alpha_\GG(S^0;\cA_q\bZ) \to \uH^\alpha_\GG(S^0;\uA)$ is either an isomorphism or multiplication by $\pm q$. In the latter case, the map from $\uH^\alpha_\GG(S^0; \uA) \to \uH^\alpha_\GG(S^0; \cA_q\bZ)$ is an isomorphism. We have the exact sequence 
$$\uH^\alpha_\GG(S^0; \uA) \to \uH^\alpha_\GG(S^0; \cA_q\bZ)\to \uH^{\alpha+1}_\GG(S^0; \cA_qL_q)  \to \uH^{\alpha+1}_\GG(S^0; \uA)$$
associated to $0\to \cA_qL_q \to \uA \to \cA_q\bZ \to 0$. From Proposition \ref{rel-RL}, we have 
$$  \uH^{\alpha+1}_\GG(S^0; \cA_qL_q) \cong \uH^{\alpha+3-\xi^p}_\GG(S^0; \cA_qR_q) \cong \bbZq$$
from Table \ref{odd-tabaqrq}. Note also that $\uH^{\alpha+1}_\GG(S^0; \uA)=0$ from Table \ref{odd-tab}. It follows that $\uH^\alpha_\GG(S^0; \uA) \to \uH^\alpha_\GG(S^0; \cA_q\bZ)$ cannot be an isomorphism and so,  $\uH^\alpha_\GG(S^0;\cA_q\bZ) \to \uH^\alpha_\GG(S^0;\uA)$ is an isomorphism. Therefore, in this case  $\uH^\alpha_\GG(S^0;\cA_qR_q)=0$.
\begin{table}[ht]

\begin{tabular}{ |p{4.2cm}|p{2cm}||p{4cm}| p{2cm}|  }
 \hline
{\tiny $\alpha$ } & { \tiny $\uH_\GG^\alpha(S^0;\cA_qR_q)$} & {\tiny $\alpha$ } & {\tiny $\uH_\GG^\alpha(S^0;\cA_qR_q)$} \\
 \hline
{\tiny $|\alpha^{C_p}|< 0$, ($|\alpha^{C_q}|\neq 0$ or $|\alpha^\GG|\neq 0$)}  & {\tiny $0$ } 
& {\tiny $|\alpha^{C_p}|> 0,|\alpha^{C_q}|= 0,|\alpha^\GG|< 0$}  &  {\tiny $\bbZq$ }  \\
\hline
{\tiny $|\alpha^{C_p}|> 0, |\alpha^\GG|>0$}  &  {\tiny $0$ }
& {\tiny $|\alpha^{C_p}|>0,|\alpha^{C_q}|\neq 0,|\alpha^\GG|\leq 0$}  &  {\tiny $\bbZq$ }\\
\hline
 {\tiny $|\alpha^{C_p}|= 0, |\alpha^\GG|>0$}  &  {\tiny $\CC_p \bZ$ }
&{\tiny $|\alpha^{C_p}|=0, |\alpha^\GG| < 0$}  &  {\tiny $\KK_p\bZ$ }\\
\hline
\end{tabular}
\vspace{.2cm}
\caption{Formula for $\uH_\GG^\alpha(S^0;\cA_qR_q)$ for $|\alpha|<0$ even.}
\label{evneg-tabaqrq}
\end{table}

Starting from Table \ref{evneg-tabaqrq} we compute the cases where $|\alpha^\GG|=0$ and $|\alpha^{C_q}|=0$. For example, if $|\alpha^{C_p}|<0$, $|\alpha^{C_q}|=0=|\alpha^\GG|$, we read off $\uH^{\alpha-\xi^q}_\GG(S^0;\cA_qR_q)=0$ from Table \ref{evneg-tabaqrq}. We also note that $\uH^{\alpha}_\GG({\GG/C_q}_+;\cA_qR_q)=0$ and $ \uH^{\alpha-1}_\GG({\GG/C_q}_+;\cA_qR_q)=0$, so that $\uH^{\alpha}_\GG(S(\xi^q)_+;\cA_qR_q)=0$. It follows that $\uH^\alpha_\GG(S^0;\cA_qR_q)=0$. Proceeding in this manner we obtain for $|\alpha|<0$ 
\begin{equation} \label{aqrqzer}
\uH^\alpha_\GG(S^0;\cA_qR_q) \cong \begin{cases} 
                                           0  & \mbox{if}~|\alpha^{C_p}|<0, |\alpha^{C_q}|=0, |\alpha^\GG|=0 \\
                                   \bbZq & \mbox{if}~|\alpha^{C_p}|>0, |\alpha^{C_q}|=0, |\alpha^\GG|=0 \\
                           \KK_p\bZ  & \mbox{if}~|\alpha^{C_p}|=0, |\alpha^{C_q}|=0, |\alpha^\GG|=0 \\   
\end{cases}
\end{equation}
The computations in Table \ref{odd-tabaqrq}, Table \ref{evneg-tabaqrq} and \eqref{aqrqzer} are used to compute $\uH^\alpha_\GG(S^0;R_{pq})$ in the Theorem below. 
\begin{thm}\label{compr}
$$\uH^\alpha_\GG(S^0;R_{pq}) \cong \begin{cases}
\KK_p\bZp  & \mbox{if}~|\alpha|<0, |\alpha^{C_p}|>1, |\alpha^{C_q}|\leq 1~\mbox{odd} \\ 
\KK_q\bZq  & \mbox{if}~|\alpha|<0, |\alpha^{C_p}|\leq 1, |\alpha^{C_q}|> 1~\mbox{odd} \\
\KK_p\bZp \oplus \KK_q\bZq & \mbox{if}~|\alpha|<0, |\alpha^{C_p}|>1, |\alpha^{C_q}|>1~\mbox{odd} \\
\KK_p\bZp \oplus \KK_q\bZq & \mbox{if}~|\alpha|>0, |\alpha^{C_p}|\leq 0, |\alpha^{C_q}|\leq 0~ \mbox{even} \\
\KK_p\bZp  & \mbox{if}~|\alpha|>0, |\alpha^{C_p}|\leq 0, |\alpha^{C_q}|> 0~ \mbox{even} \\
\KK_q\bZq  & \mbox{if}~|\alpha|>0, |\alpha^{C_p}|> 0, |\alpha^{C_q}|\leq 0~ \mbox{even} \\
R_{pq}  & \mbox{if}~|\alpha|=0, |\alpha^{C_p}|\leq 0, |\alpha^{C_q}|\leq 0 \\
L_{pq}  & \mbox{if}~|\alpha|=0, |\alpha^{C_p}|> 0, |\alpha^{C_q}|> 0 \\
\KK_pL_p  & \mbox{if}~|\alpha|=0, |\alpha^{C_p}|> 0, |\alpha^{C_q}|\leq 0 \\
\KK_qL_q  & \mbox{if}~|\alpha|=0, |\alpha^{C_p}|\leq 0, |\alpha^{C_q}|> 0 \\
0  &\mbox{otherwise}. 
\end{cases} $$
\end{thm}

\begin{proof}
We start this computation at odd $\alpha$, whence the result directly follows from Table \ref{odd-tabaqrq}, Proposition \ref{bZcomp} and \eqref{coeff-ex2}.  Next we consider $|\alpha|<0$ even. It is clear from Table \ref{evneg-tabaqrq}, Proposition \ref{bZcomp} and \eqref{coeff-ex2} that $\uH^\alpha_\GG(S^0;R_{pq})=0$ in this case unless $|\alpha^{C_p}|=0$. In this case we have 
$$\uH^\alpha_\GG(S^0;\cA_qR_q)= \begin{cases} \CC_p\bZ &\mbox{if}~|\alpha^\GG|>0 \\ 
                                                                                                 \KK_p\bZ &\mbox{if}~|\alpha^\GG|\leq 0.
\end{cases} $$
Consider the former case with the additional condition $|\alpha^{C_q}|\leq 0$. The other cases are similar. Here \eqref{coeff-ex2} looks like 
$$0 \to \CC_p\bZ \to \CC_p \bZ \to  \uH^\alpha_\GG(S^0;R_{pq})\to 0.$$ 
In order to verify whether $\CC_p\bZ \to \CC_p\bZ$ is an isomorphism we consider the exact sequence 
$$\uH^\alpha_\GG(S^0; \KK_p\uA_p) \to \uH^\alpha_\GG(S^0; \KK_p\bZ)\to \uH^{\alpha+1}_\GG(S^0; \KK_pL_p)  \to \uH^{\alpha+1}_\GG(S^0; \KK_p\uA_p)$$
associated to $0\to \KK_pL_p\to \KK_p\uA_p \to \KK_p\bZ \to 0$. Proposition \ref{rel-RL} implies 
$$ \uH^{\alpha+1}_\GG(S^0; \KK_pL_p) \cong  \uH^{\alpha+3 - \xi^q}_\GG(S^0; R_{pq}) = \KK_p\bZp,$$
the last isomorphism coming from the above calculation in the odd case. Note also from Table \ref{odd-tabaqrq} that $\uH^{\alpha+1}_\GG(S^0; \KK_p\uA_p)=0$. It follows that $\uH^\alpha_\GG(S^0; \KK_p\uA_p) \to \uH^\alpha_\GG(S^0; \KK_p\bZ)$ must be multiplication by $p$, and hence, 
$$\CC_p \bZ \cong  \uH^\alpha_\GG(S^0; \KK_p\bZ) \to  \uH^\alpha_\GG(S^0; \KK_p\uA_p) \cong \CC_p \bZ$$ 
is an isomorphism. Thus, $\uH^\alpha_\GG(S^0;R_{pq})=0$. Similar arguments in the other cases allow us to conclude that $\uH^\alpha_\GG(S^0;R_{pq})=0$ if $|\alpha|<0$ even. 

Finally we compute the cases $|\alpha|=0$ and $|\alpha|>0$ even from the exact sequence 
\begin{equation}\label{exr}
\cdots \uH^{\alpha-\xi}_\GG(S^0;R_{pq}) \to \uH^\alpha_\GG(S^0;R_{pq}) \to \uH^\alpha_\GG(S(\xi)_+; R_{pq}) \to \uH^{\alpha+1-\xi}_\GG(S^0;R_{pq}) \to \cdots 
\end{equation}
associated to the cofibre sequence $S(\xi)_+ \to S^0 \to S^\xi$. We know the value for $\uH^{\alpha+1-\xi}_\GG(S^0;R_{pq})$ from the odd case above, and we have 
$$ \uH^{\alpha}_{\GG}({S(\xi)}_+; R_{pq}) = 
  \begin{cases}
    R_{pq}, & \mbox{for } |\alpha| = 0 \\
    L_{pq}, & \mbox{for } |\alpha| = 1 \\
    0,  &\mbox{Otherwise}  
  \end{cases}
$$
by a similar argument as Proposition \ref{AG}. Therefore, we obtain for $|\alpha|=0$, $  \uH^\alpha_\GG(S^0;R_{pq})$ is the kernel of a surjective map $R_{pq}$ to one of $\KK_p\bZp$, $\KK_q\bZq$ or $\KK_p\bZp\oplus \KK_q\bZq$ depending on whether $|\alpha^{C_p}|$ or $|\alpha^{C_q}|$ are $\leq 0$ or $>0$. These kernels are $\KK_pL_p$, $\KK_qL_q$ and $L_{pq}$ respectively, and this finishes the $|\alpha|=0$ case. 

Finally if $|\alpha|>0$ even, \eqref{exr} implies $\uH^\alpha_\GG(S^0;R_{pq})\cong \uH^{\alpha+\xi}(S^0;R_{pq})$ if $|\alpha|\geq 2$. Thus it suffices to compute for $|\alpha|=2$ where \eqref{exr} gives the exact sequence 
$$0 \to L_{pq} \to \uH^{\alpha-\xi}_\GG(S^0;R_{pq}) \to \uH^\alpha_\GG(S^0;R_{pq}) \to 0.$$
As $|\alpha - \xi|=0$, we may read off $\uH^\alpha_\GG(S^0;R_{pq})$ from the case above.
\end{proof}

 \section{Freeness theorem}\label{free}
In this section we use the computations in the previous sections to deduce a freeness theorem for $\GG$. We start with some definitions.

\begin{defn}
Let $V$ be a representation of $H< \GG$. A cell $\GG \times _H D(V)$ is called an even cell if $V$ is even in $RO(H)$. 
\end{defn}
Note that in the above definition a representation $V$ of $H<\GG$ may also be viewed as the restriction of a representation $\hat{V}$ of $\GG$, and $V$ is even in $RO(H)$ $\iff$ $\hat{V}$ is even in $RO(\GG)$. Therefore, in the following we only consider cells of the type $\GG\times_H D(V)$ where $V$ is a representation of $\GG$.  

\begin{defn}
Given two representations $W$ and $V$ of $\GG$ we say $W\ll V$ if $|W^S|<|V^S|$ for some subgroup $S$ of $\GG$ implies $|W^T|\leq|V^T|$ for all subgroups $T$ containing $S$.
\end{defn}

A generalized equivariant cell complex defined by inductively attaching cells of the type $\GG\times_H D(V)$ is said to be of even type if every cell is even and if the cell $\GG\times_H D(W)$ is attached before $\GG\times_H D(V)$, then $W\ll V$. We will conclude that generalized equivariant cell complexes of even type have free cohomology. We start by proving a crucial Lemma. 

\begin{lemma}\label{prop 1}
If $W\ll V$, any $\uH^{\ast}_{\GG}(S^0)$- module map 
$$\uH^{\ast - W}_{\GG}({\GG/K}_+) \to \uH^{\ast - V +1}_{\GG}({\GG/K'}_+)$$ 
is $0$ whenever either $K$ or $K'$ is not equal to $\GG.$
\end{lemma}

\begin{proof}
First let $K=\GG$  so that a  $\uH^{\ast}_{\GG}(S^0)$-map  $f: \uH^{\ast -W}_{\GG}(S^0) \to \uH^{\ast+1 -V}_{\GG}({\GG/K'}_+)$ is determined by the image of $1$. We compute the group $\uH^{W+1-V}_{\GG}({\GG/K'}_+)(\GG/\GG) = \tilde{H}^{W+1-V}_{\GG}({\GG/K'}_+).$ If $K' = e,$ then Proposition \ref{AG} implies that $\tilde{H}^{W+1-V}_{\GG}({\GG/K'}_+)$ is zero as $|W+1-V|$ is odd. Propositions \ref{Lewis orbits} and \ref{G/Cp} state that $\tilde{H}^{W+1-V}_{\GG}({\GG/K'}_+)$ is again zero for cases $K' = C_p$ and $C_q$, as either $|W+1-V|$ is positive and odd or $|W+1-V|$ is negative and $|(W+1-V)^{C_p}|,| (W+1-V)^{C_q}| \leq 1$.

Next consider the case $K=e$ where Proposition \ref{AG}, states that  $\uH^{\alpha -W}_{\GG}({\GG/e}_+)$ is non-zero only for $|\alpha|=|W|$ and equal to $\uA_{\GG/e}$ in this case. In this case, a $\uH^{\ast}_{\GG}(S^0)$-module map  $\uH^{\alpha -W}_{\GG}({\GG/e}_+) \to \uH^{\alpha+1 -V}_{\GG}({\GG/K'}_+)$ is classified by an element of 
$$\uH^{\alpha+1 -V}_{\GG}({\GG/K'}_+)(\GG/e)\cong \tilde{H}^{\alpha+1-V}_{\GG}(\GG/e_+ \wedge {\GG/K'}_+)  \cong \oplus_{|{\GG/K^{\prime}}|} \tilde{H}^{\alpha+1 -V}_{\GG}({\GG/e}_+).$$
This is equal to zero as $|\alpha +1 -V| = |W|+1 -|V|$ is odd. 

Finally, it remains to prove the case $K=C_p$ (and thus by symmetry for $K=C_q$). We note from Propositions \ref{epadj} and \ref{G/Cp} that a map of $\GG$-Mackey functors 
$$\uH^{\ast - W}_{\GG}({\GG/C_p}_+) \to \uH^{\ast - V +1}_{\GG}({\GG/K'}_+)$$ 
is equivalent to a map of $C_p$-Mackey functors 
\begin{equation}\label{mcp}
\uH^{\ast - W}_{C_p}(S^0) \to \Phi_p^\ast\uH^{\ast - V +1}_{\GG}({\GG/K'}_+)= \uH^{\ast - V +1}_{C_p}({\GG/K'}_+) .
\end{equation}
Further note that a $\uH^\ast_{\GG}(S^0)$-module map $\uH^{\alpha-W}_{\GG}({\GG/C_p}_+) \to \uH^{\alpha-V+1}_{\GG}({\GG/K'}_+)$ induces a $\uH^{\ast}_{C_{p}}(S^0)$-module map between $\Phi_p^\ast\uH^{\ast-W}_{\GG}({\GG/C_p}_+) \to \Phi_p^\ast\uH^{\ast-V+1}_{\GG}({\GG/K'}_+)$. Also, note that $\Phi_p^\ast \uH^\alpha_{\GG}(\GG/C_p)$ equals $\oplus_q \uH^\alpha_{C_p}(S^0)$ and the unit map $\eta$ in Proposition \ref{epadj} given by the inclusion of the first factor is a $\uH^\ast_{C_p}(S^0)$-module map. Therefore the map in \eqref{mcp} is a $\uH^\ast_{C_p}(S^0)$-module map. It follows that this is determined by an element in the group $\tilde{H}^{W-V+1}_{C_p}(\GG/K'_+)$, a direct sum of copies of $\tilde{H}^{W-V+1}_{C_p}(C_p/e_+)$ and $\tilde{H}^{W-V+1}_{C_p}({C_p/C_p}_+)$. This reduces to the freeness theorem for $C_p$ as in \cite{Lew88}. Hence the proof is complete. 
\end{proof}

Next, we state and prove the freeness theorem. 
\begin{thm} \label{main}
Let $X$ be a $\GG$-space with a filtration $\{X_n\}_{n \geq 0}$ where each $X_n$ is a finite generalized $\GG$-cell complex with only even cells.  Suppose that the filtration satisfies\\
(1) If $\GG \times _K D(W)$ is a cell in $X_n$ and $\GG \times _L D(V)$ be a cell in $X_{n+1}$, then $W\ll V$. \\
(2) For each positive integer $N$, only finitely many cells of the form $\GG \times_{K} D(V)$ arise in the resulting cell complex structure of $X$ that have $|V^H| \leq N$ for all $H$ subgroups of $\GG.$\\
Then $\uH^{\ast}_{\GG}(X_+)$ is a free $\uH^{\ast}_{\GG}(S^0)$-module with summands consisting of  $\uH^{\ast}_{\GG}({X_0}_+)$ and one copy of $\Sigma^{V} \uH^{\ast}_{\GG}({\GG/K}_+)$ for each cell of type ${\GG}_+ \wedge _K D(V)$ occurring in the cell structure of $X$.
\end{thm}

\begin{proof}
We prove the theorem inductively accordingly as cells are attached to $X$. That is, we assume the conclusion for $X_n$, and assume that  $X_{n+1}$ is obtained out of $X_n$ by attaching a single cell $\GG\times_H D(V)$, and conclude the result for $X_{n+1}$. Observe that the attachment implies that the cofibre of the map ${X_n}_+ \to {X_{n+1}}_+$ is ${\GG}_+\wedge_HS^V$. Thus we have the associated long exact sequence 
$$\cdots \to \uH^{\alpha}_{\GG}({\GG}_+ \wedge_H S^V) \to \uH^{\alpha}_{\GG}({X_{n+1}}_+) \to \uH^{\alpha}_{\GG}({X_n}_+) \stackrel{\delta}{\to} \uH^{\alpha +1}_{\GG}({\GG}_+ \wedge_H S^V) \to \cdots$$
We prove that under the hypothesis on cell attachments in Theorem \ref{main}, the map $\delta$ is forced to be zero. In order to show this, it suffices to consider each summand of the domain separately. That is, it suffices to consider possible $\uH^\ast_{\GG}(S^0)$-module maps 
$$\uH^{\ast - W}_{\GG}({\GG/K}_+) \to \uH^{\ast - V +1}_{\GG}({\GG/K'}_+)$$ 
for all $K$ and $K^{\prime}$  are  subgroups  of $\GG$. The computations split into two cases where both $K$ and $K'$ are not $\GG$, and where $K,K'=\GG$. The former is proved in the Lemma \ref{prop 1} above and we conclude the latter follows from our computations in the previous sections. 

Let $K,K'=\GG$ so that we have a $\uH^\ast_{\GG}(S^0)$-module map 
$$\uH^{\ast - W}_{\GG}(S^0) \to \uH^{\ast - V +1}_{\GG}(S^0)$$ 
Such a map is determined by the image of $1$, which lies in the group $\tilde{H}^{W+1-V}_{\GG}(S^0)$. Let $\alpha = W+1-V$ and note that $\alpha$ is odd as $W$ and $V$ are even. If $|W|<|V|$, then $W\ll V$ implies that all the fixed points have dimension $\leq 1$, and so, by Proposition \ref{calc1b}, $\uH^\alpha_{\GG}(S^0)=0$. If $|W|\geq |V|$ and $|W^{C_p}|< |V^{C_p}|$ or $|W^{C_q}|< |V^{C_q}|$, then $|W^{\GG}|\leq |V^{\GG}|$ and we note $\alpha$ satisfies $|\alpha|\geq 1$, one of $|\alpha^{C_p}|$ or $|\alpha^{C_q}|$ is $<0$ and $|\alpha^{\GG}|\leq 1$. In this case Proposition \ref{calc2} and Proposition \ref{calc3} implies $\uH^\alpha_{\GG}(S^0)=0$. Finally when $|W|\geq |V|$, $|W^{C_p}|\geq |V^{C_p}|$ and $|W^{C_q}| \geq |V^{C_q}|$, we have that $|\alpha^H|$ is positive if $H\neq \GG$ so that by Proposition \ref{sign}, $\uH^\alpha_{\GG}(S^0)=0$. 

As a consequence we have the following short exact sequence 
$$0 \to \uH^{\alpha-V}_{\GG}({\GG/H}_+) \to \uH^{\alpha}_{\GG}({X_{n+1}}_+) \to \uH^{\alpha}_{\GG}({X_{n}}_+) \to 0.$$
 Since $\uH^{\alpha}_{\GG}({X_n}_+)$ is free, the Mackey functor $\uH^{\alpha}_{\GG}({X_{n+1}}_+)$ is the direct sum of $\uH^{\alpha}_{\GG}({X_{n}}_+)$ and $\uH^{\alpha-V}_{\GG}(\GG/H_+)$. This completes the induction argument. 

In order to verify the conclusion for $X$, we need to verify that the direct sum decomposition passes to the colimit. However, we have that the map $\uH^{\alpha}_{\GG}({X_{n+1}}_+) \to \uH^{\alpha}_{\GG}({X_{n}}_+)$ is surjective so that  the $\lim^{1}$ term vanishes. Thus we have $$\uH^{\alpha}_{\GG}(X_+)  =  \prod \uH^{\alpha -V}_{\GG}({\GG/H}_+).$$ 
Now by condition 2), given a value of $\alpha$, for all but finitely many $V$ the fixed points of $\alpha - V$ are all negative dimensional. Hence, Proposition \ref{sign} implies that the product is actually finite for every grading $\alpha$, and hence it is isomorphic to the direct sum. Therefore, the result follows.
\end{proof}

A couple of remarks about the Theorem are in order, comparing it to the $C_p$-freeness Theorem in \cite{Lew88}.  
\begin{remark}
We note that the condition $W\ll V$ in the Theorem does not precisely resemble the condition $|W|<|V|$ implies $|W^G|\leq |V^G|$ in \cite{Lew88} for $G=C_p$. However, for the group $C_p$, the only way in which $H$ is a proper subgroup of $K$ is when $H=e$ and $K=C_p$. So our condition reduces to the condition of Lewis in \cite{Lew88}. Also, note that under the weaker assumption ``$|W|<|V|$ implies $|W^G|\leq |V^G|$", the cohomology at $W+1-V$ is not $0$, thus, our arguments would not prove freeness results in this case. An example of this case lies in the computations of Proposition \ref{calc1odd} where $|W|<|V|$, $|W^{C_p}| > |V^{C_p}|$ and $|W^{\GG}| \leq |V^{\GG}|$ and the cohomology at $W+1-V$ is either $\KK_p\bZp$ or $\KK_p\bZp \oplus \KK_q\bZq$. 
\end{remark}

\begin{remark}
An important observation is that a slightly stronger version of the freeness Theorem may be deduced from the above. If we assume $W$ and $V$ are even with $|W|<|V|$, $|W^{C_p}|=|V^{C_p}|$, and $|W^{C_q}|=|V^{C_q}|$, then the cohomology at $W+1-V$ is $0$ by Proposition \ref{calc1odd} a) (for any values of the $\GG$-fixed points). Therefore, if one is allowed to attach a cell $D(V)$ after $D(W)$, the freeness theorem would still hold. Hence, we may include this in the definition of $\ll$ to obtain a slightly stronger freeness theorem. 
\end{remark}

\subsection{Applications} 
We prove that the freeness theorem above applies to certain $\GG$-complex projective spaces and Grassmannians by checking the criteria of Theorem \ref{main} in these cases. We first outline the case of projective spaces, and then point out the generalization for Grassmannians. For a complex $\GG$-representation $V$, denote by $\C P(V)$ the set of lines in $V$ which inherits a natural $\GG$-action. We may write such a $V$ as a direct sum of irreducibles as $\sum n_i \xi^i$ and ask when the cohomology of $\C P(V)$ is free depending on the numbers $n_i$. While we don't answer this question directly, we verify that in the cases $V=\sum_{i=0}^k \xi^i$ and $V=\UU_\C$ a complete $\GG$-universe, the conditions of the freeness theorem are satisfied.    

Recall that for a finite group $G$, a complete $G$-universe is defined as  a countably infinite dimensional unitary representation $\UU_\C$ which contains infinitely many copies of each finite dimensional representation. Any two such complete universes are unitarily equivalent, and an example is obtained by taking the sum of countably many copies of the regular representation. In our specific case of $G=\GG$, an example of such an universe is the direct sum
$$\UU_\C = \oplus_{k =0}^{\infty} \xi^k$$ 
built up from the increasing sequence of sub-representations $\UU_\C(n) :=\oplus_{k=0}^{n} \xi^k$. 

Consider the complex projective space $\C P(V)$ where $V=\sum_{k=0}^n \phi_k$ with each $\phi_k$ an irreducible (hence, one-dimensional) $\GG$-representation. Points in $\C P(V)$ may be expressed by homogeneous coordinates of the form $\langle x_0, x_1, \cdots, x_n \rangle$, $x_k \in \phi_k$ such that not all of $x_k$ are zero. The $\GG$-action on such points is coordinate-wise. We note from \cite{Lew88} that $\C P(V)$ and $\C P(V\otimes \xi)$ are homeomorphic as $\GG$-spaces. For, if $\langle x_0, x_1, \cdots, x_n \rangle$ is in $\C P(V)$, the action of the generator $g$ of $\GG$ on such an element is of the form 
$$g. \langle x_0, x_1, \cdots, x_n \rangle = \langle \zeta^{l_0} x_0, \zeta^{l_1} x_1, \cdots, \zeta^{l_n} x_n \rangle$$
where $\zeta$, a fixed primitive $pq^{th}$ root of unity. However, as we are in a projective space  multiplying throughout by $\zeta$ does not change the result, and so, 
$$g. \langle x_0, x_1, \cdots, x_n \rangle = \langle \zeta^{l_0+1} x_0, \zeta^{l_1+1} x_1, \cdots, \zeta^{l_n+1} x_n \rangle$$
This is precisely the action on $\C P(V\otimes \xi)$. We specialize to $V=\UU_\C(n)$. Thus, we may represent points as $\langle x_0, x_1, \cdots, x_n \rangle$ for $x_k \in \xi^k$, but also as $\langle x_0, x_1, \cdots, x_n \rangle$, $x_k \in \xi^s \otimes \xi^k$ where $s$ is fixed. 

We observe that $\C P(\UU_\C(k-1))$ is an equivariant subspace of $\C P(\UU_\C(k))$ as those points whose last coordinate is $0$. Such inclusions are clearly equivariant cofibrations and the $\GG$-projective space $\C P(\UU_\C)$ is the colimit of the sequence of spaces $\{ \C P (\UU_\C(n)) \}_{n \geq 1}.$

Now we show that the space $\C P(\UU_\C)$ is a generalized $\GG$-cell complex with a filtration as in Theorem \ref{main}. We do this by inductively constructing generalized cell-structures on  $\C P(\UU_\C(n))$.

  Let $W_n$ be the representation $\xi^{-n} \otimes \UU_\C(n-1).$ The unit disc $D(W_n)$ is described as the set of points $(x_0, x_1, \cdots, x_{n-1})$, where $x_j \in \xi^{-n} \otimes \xi^j$ and $\sum |x_j|^2 <1$.  Define a $\GG$-equivariant map from $ \phi_n : D(W_n) \to \C P(\UU_\C(n)\otimes \xi^{-n})$ by 
$$\phi_n(x_0, x_1, \cdots, x_{n-1}) = \langle x_0, x_1, \cdots, x_{n-1}, 1- \sum_{k=0}^{n-1}|x_k|^2\rangle.$$ 
From the discussion above, $\C P(\UU_\C(n)\otimes \xi^{-n})\cong \C P(\UU_\C(n))$ so we may also write this as $\phi_n : D(W_n) \to \C P(\UU_\C(n))$. It is clear that the restriction  ${\phi_n}_{|S(W_n)}$ is contained in $\C P(\UU_\C(n-1))$ and that the map 
$$\phi_n : D(W_n) \setminus S(W_n) \to \C P(\UU_\C(n)) \setminus \C P(\UU_\C(n-1))$$
 is a homeomorphism. Therefore, the space $\C P(\UU_\C(n))$ is obtained by attaching a cell $\GG\times_{\GG}D(W_n)$ along the map ${\phi_n}_{|S(W_n)} : S(W_n) \to \C P(\UU_\C(n-1))$. Hence, we conclude that $\C P(\UU_\C(n))$ is a generalized $\GG$-cell complex with cells $D(W_k)$ for $1 \leq k \leq n.$  

 Observe that the values of $|W_k^H|$ for various $H$ are given by  
$$(|W_k|, |W_k^{C_p}|, |W_k^{C_q}|, |W_k^{\GG}|) =(2k, 2\lfloor\frac{k}{p}\rfloor, 2\lfloor\frac{k}{q}\rfloor, 2\lfloor\frac{k}{pq}\rfloor)$$  
Thus each $W_k$ is even and we have $W_k\ll W_{k+1}$. Therefore, we deduce the following theorem from Theorem \ref{main}.

\begin{thm} \label{freeproj}
For $1\leq n \leq \infty$, the cohomology  $\uH^{\ast}_{\GG}(\C P(\UU_\C(n)))$ is a free module over $\uH^{\ast}_{\GG}(S^0)$ generated by one element in dimension zero and one in each dimension $W_k$ for $1\leq k \leq n$, that is,   
$$ \uH^{\ast}_{\GG}(\C P(\UU_\C(n))) \cong \bigoplus_{k = 0}^n\uH^{\ast - W_k}_{\GG}(S^0).$$
\end{thm}

Theorem \ref{main} also applies for the Grassmannian manifold $G(\UU_\C(n), k)$ of complex $k$-dimensional subspaces of $\UU_\C(n)$. Recall the generalized $\GG$-equivariant cell structure of $G(\UU_\C(n),k)$ from Chapter 7 of \cite{FL04}, which is an equivariant version of the Schubert cell structure on Grassmannians.  Given a sequence of integers $0 \leq a_1 \leq a_2 \leq \cdots \leq a_k \leq n-k$ which we denote by $\ua$, one has a representation cell $D(W_{\ua})$ where the representation $W_{\ua}$ is  
$$W_{\ua} = \bigoplus _{i=1}^k \bigoplus_{\substack{j=1\\  j \notin \{a_1+1, \cdots a_{i-1}+i-1\}}}^{a_i +i-1} \xi^{-(a_i +i)} \otimes \xi^j.$$  
We note that the attaching map of the cell $D(W_{\ub})$ hits the cell $D(W_{\ua})$ if and only if $\ua\leq \ub$ (that is, each entry of $\ua$ is $\leq$ the corresponding entry of $\ub$). 

Observe that the values of $|W_{\ua}^H|$ for various $H$ are given by  
$$(|W_{\ua}|, |W_{\ua}^{C_p}|, |W_{\ua}^{C_q}|, |W_{\ua}^{\GG}|) =(2 \sum_{i=1}^k a_i, 2\sum_{i=1}^k \lfloor\frac{ a_i}{p}\rfloor, 2 \sum_{i=1}^k \lfloor\frac{a_i}{q}\rfloor, 2 \sum_{i=1}^k \lfloor\frac{a_i}{pq}\rfloor)$$  
Thus each $W_{\ua}$ is even and for $\ua \leq \ub$ we have $W_{\ua}\ll W_{\ub}$. Hence Theorem \ref{main}  applies and we deduce

\begin{thm} \label{freegra}

For $1\leq n \leq \infty$,   $\uH^{\ast}_{\GG}(G(\UU_\C(n)))$ is the free over $\uH^{\ast}_{\GG}(S^0)$ generated by one element in each dimension $W_{\ua}$ for $\ua = (a_1,\cdots, a_k)$ satisfying $0 \leq a_1 \leq a_2 \leq \cdots \leq a_k \leq n-k$. 
\end{thm}

\mbox{ }\\
\end{document}